%% file: TenHamBd.tex
\documentclass[10pt,regno]{amsart} %ws-jta is a formatting class
\usepackage[margin=1.5in]{geometry}
\usepackage[utf8]{inputenc}
\usepackage{amsthm,amssymb,amsmath,amsfonts,mathrsfs,mathtools,thmtools,dsfont}
\usepackage[all]{xy}
\usepackage{color,graphicx}
\usepackage{enumitem,hyperref,cleveref,epigraph}

\usepackage{times} %sets the font to times new roman
\usepackage{sidecap}
\usepackage[T1]{fontenc}
\newtheorem{twr}{Theorem}
\newtheorem{prop}{Proposition}[section]
\newtheorem{define}{Definition}[section]
\newtheorem{lem}[prop]{Lemma}%[section]
\numberwithin{equation}{section}
\newcommand{\dist}{{\operatorname{dist}}}
\newcommand{\Ker}{{\operatorname{Ker}}}
\newcommand{\Crit}{{\operatorname{Crit}}}
\newcommand{\Hess}{{\operatorname{Hess}}}

\newcommand{\bigslant}[2]{{\raisebox{.2em}{$#1$}\left/\raisebox{-.2em}{$#2$}\right.}}
\makeatletter
\def\namedlabel#1#2{\begingroup
   \def\@currentlabel{#2}%
   \label{#1}\endgroup
}
\makeatother

\newcommand{\etat}{{\fontfamily{cmr}\selectfont \ng}}

\begin{document}

\title{BOUNDS FOR TENTACULAR HAMILTONIANS}

\author{FEDERICA PASQUOTTO}

\address{Department of Mathematics, Vrije Universiteit Amsterdam\\
De Boelelaan 1081, 1081 HV Amsterdam, Netherlands\\
f.pasquotto@vu.nl}

\author{JAGNA WI\'{S}NIEWSKA}

\address{Department of Mathematics, Eidgen\"{o}ssische Technische Hochschule Z\"{u}rich\\
Raemistrasse 101, 8092 Z\"{u}rich, Switzerland \\
jagna.wisniewska@math.ethz.ch}

\begin{abstract}
This paper represents a first step towards the extension of the definition of Rabinowitz Floer homology to non-compact energy hypersurfaces in exact symplectic manifolds. More concretely, we study under which conditions it is possible to establish $L^\infty$-bounds for the Floer trajectories of a Hamiltonian with non-compact energy levels. Moreover, we introduce a class of Hamiltonians, called \textit{tentacular Hamiltonians} which satisfy the conditions: how to define RFH for these examples will be the subject of a follow-up paper. 
\end{abstract}

\maketitle

\tableofcontents

\input{intro2}
\input{prel2}
\input{ActBound3A}

\input{ActDeriv3}
\input{CritMfld2}
\input{oscillations3}
\input{MaxPrinc3}

\appendix
\input{geomTen2}
\input{appendix}
\input{Acknow}

\bibliographystyle{amsplain}
%\addcontentsline{toc}{section}{References}
%\nocite{*}

\bibliography{./bibliography.bib}
\end{document}

%% file: intro2.tex
\section{Introduction}
The application of symplectic techniques to Hamiltonian and Reeb dynamics is one of the principal and most exciting directions within the field of symplectic topology. 
In particular, the problem of existence of periodic orbits, generally referred to as the \emph{Weinsten Conjecture}, has received a lot of attention. 
Ever since the introduction of Gromov's holomorphic curves and Floer's techniques to the symplectic world, the approach to this problem has often been based on the properties of spaces of holomorphic curves (or Floer trajectories) with prescribed asymptotics.
For instance, Hofer's proof of the Weinstein Conjecture for the three-dimensional sphere in \cite{Hofer1993} and Taubes' proof for arbitrary closed, contact $3$-manifolds in \cite{Taubes2007} are obtained by studying holomorphic curves in the symplectization of the given contact manifold. 
Rather recently, there has been increasing interest in the question of existence of periodic orbits on non-compact energy levels (in the Hamiltonian case) or non-compact contact manifolds (in the case of Reeb dynamics). 
While the question of existence of periodic orbits in a non-compact setting is very natural from the point of view of applications (think, for instance, of celestial mechanics or hydrodynamics), it was only addressed at a later stage because of the obvious additional geometric and analytical difficulties. The first significant results in this direction, applying to mechanical hypersurfaces in standard symplectic space and cotangent bundles,  were obtained in \cite{Pasquotto2009} and its sequel \cite{rot2016}, and in \cite{Suhr2016}. In \cite{DelPino2017}, the authors prove the Weinstein conjecture for a special, very interesting class of open contact manifolds, namely those that occur as overtwisted leaves of a contact foliation in a closed manifold.  
  
Rabinowitz Floer Homology is a symplectic invariant defined by Cieliebak and Frauenfelder in \cite{CieliebakFrauenfelder2009} and associated to exact contact hypersurfaces in an exact convex symplectic manifold. Its generators are the critical points of the so called Rabinowitz action functional for a suitable choice of defining Hamiltonian. This action functional was originally used by Rabinowitz in \cite{Rabinowitz1978} to prove existence of periodic orbits on star-shaped hypersurfaces in $\mathbb{R}^{2n}$ and it differs from Floer's original symplectic action functional by a Lagrangian multiplier, which forces the periodic solutions to lie on a prescribed energy surface. In this sense, it looks like a very suitable tool to study the Weinstein Conjecture. 
In fact, Rabinowitz Floer Homology vanishes if the hypersurface is displaceable and is isomorphic to the singular homology if the hypersurface does not carry any contractible periodic orbits \cite{CieliebakFrauenfelder2009}. These properties can be used to recover some (known) instances of the Weinstein Conjecture. 
The introduction of a Lagrange multiplier, though, has the following implications: first of all, if we denote by $\mathcal{L}$ the loop space of the symplectic manifold, then  the critical set, as a subset of $\mathcal{L}\times \mathbb{R}$, is always unbounded. Moreover, the critical set of the action functional is not generically Morse any more, but only Morse-Bott. In particular, it always contains a component diffeomorphic to the fixed energy level, which is also unbounded if the energy level is non-compact.  

After it was defined, Rabinowitz Floer Homology has been connected to several other important topics in symplectic and contact topology. For instance, there exists a long exact sequence relating RFH to another symplectic invariant, namely symplectic homology (or Floer homology of symplectic manifolds with contact type boundary): this sequence was constructed in \cite{Cieliebak2010}, and it was subsequently used to compute the RFH groups for the unit cotangent sphere bundle inside the cotangent bundle of a closed manifold. In \cite{Abbon2009} and \cite{merry2011} the Rabinowitz Floer chain complex has been related to the Morse (co)chain complex, and this resulted in the computation of the Rabinowitz Floer homology of some hypersurfaces in twisted cotangent bundles. RFH of  Brieskorn manifolds has been studied in \cite{Fauck2015}, and used to study fillable contact structures on closed manifolds. Moreover, RFH has been linked to the existence of leaf-wise intersection points \cite{Albers2010}, via a particular perturbation of the Rabinowitz action functional, and to questions of orderability and non-squeezing in contact geometry \cite{Albers2013}.

In this paper we lay the foundations for the extension of Rabinowitz Floer Homology to certain families of non-compact hypersurfaces in convex symplectic manifolds. More precisely, we study and prove compactness results for the moduli spaces of the Floer trajectories which are involved in the definition of the differential for the homology complex. We plan to address in a follow-up paper the question of which conditions guarantee that the homology groups are well-defined. Ultimately, our hope is that we will be able to compute the Rabinowitz Floer homology for some examples of non-compact hypersurfaces and that this will eventually lead to a proof of the Weinstein conjecture for these examples.

Let us now take a closer look at the construction of Rabinowitz Floer Homology. Consider an exact symplectic manifold $(M, \omega=d\lambda)$ and a hypersurface $\Sigma\subseteq M$, which is exact, that is, $\lambda |_{\Sigma}$ is a contact form on $\Sigma$. If we additionally require $(M, \omega)$ to be convex at infinity and $\Sigma$ to be compact and bounding, i.e. $M\setminus \Sigma$ consists of two connected components, a pre-compact and a non-compact one, then $\Sigma$ is called an exact convex hypersurface in the exact convex symplectic manifold $(M,\omega)$. 
%To such an object, Cieliebak and Frauenfelder have associated in \cite{CieliebakFrauenfelder2009} an algebraic invariant called
%\emph{Rabinowitz Floer homology}. The main goal of the present paper is to study under which conditions the definition of Rabinowitz Floer %homology can be extended so as to include examples of non-compact exact hypersurfaces. 
%In this setting $\Sigma$ was required to be compact and bounding, i.e. $M\setminus \Sigma$ consists of two connected components, one pre-compact and one noncompact
%Consider an exact symplectic manifold $(M, \omega=d\lambda)$ with hypersurface $\Sigma\subseteq M$, which is exact namely $\lambda\Big|_{\Sigma}$ is a contact form on $\Sigma$. \emph{Rabinowitz Floer homology} has been defined by Cieliebak and Frauenfelder in \cite{CieliebakFrauenfelder2009} as an algebraic invariant of an exact convex hypersurface $\Sigma$ in an exact convex symplectic manifold $(M,\omega)$.  $\Sigma$ is an exact convex hypersurface in an exact convex symplectic manifold $(M,\omega)$ if we additionally require $(M, \omega)$ to be convex at infinity and $\Sigma$ to be compact and bounding, i.e. $M\setminus \Sigma$ consists of two connected components, one pre-compact and one noncompact. 
To construct Rabinowitz Floer homology, one starts by choosing a Hamiltonian $H\in C^{\infty}(M)$ defining $\Sigma$ as a regular level set, say  $\Sigma= H^{-1}(0)$ and $dH|_{\Sigma}\neq 0$. To such a Hamiltonian $H$ one can associate the functional 
\begin{align*}
\mathcal{A}^{H} & :C^{\infty}(S^{1},M)\times \mathbb{R} \to \mathbb{R}, \\
\mathcal{A}^{H}(v,\eta) & = \int_{S^{1}} \lambda(\partial_{t}v)-\eta\int_{S^{1}} H(v),
\end{align*}
where $\lambda$ is a primitive of $\omega$. Note that $\mathcal{A}^{H}$ does not depend on the choice of $\lambda$. The above functional was first defined by Rabinowitz in  \cite{Rabinowitz1978} and it is therefore called the \emph{Rabinowitz action functional}. The periodic orbits of the Hamiltonian vector field $X^{H}$ induced by $H$ on $\Sigma$ are the critical points of the Rabinowitz action functional. The parameter $\eta$ in the functional can be regarded as the period of the loop. Indeed, there is a correspondence between elements of $C^{\infty}(S^{1},M)\times \mathbb{R}\setminus\{0\}$ and loops of period $\eta\neq 0$ given by:
$$
(v(t),\eta)\quad \eta\neq 0 \quad \longleftrightarrow\quad v\Big(\frac{t}{\eta}\Big)=:\tilde{v}(t) \in C^{\infty}\big(\bigslant{\mathbb{R}}{|\eta| \mathbb{Z}},M\big).
$$

Besides being the period of the loop, the parameter $\eta$ has another function: it acts as a Lagrange multiplier, picking out only the periodic orbits of the Hamiltonian vector field which lie on $\Sigma$. Indeed, if we calculate the derivative of $\mathcal{A}^{H}$ with respect to a variation $(\xi, \sigma) \in T_{(v,\eta)}(C^{\infty}(S^{1},M)\times \mathbb{R}) =C^{\infty}(S^{1},v^{*}TM)\times \mathbb{R} $ we obtain:
$$
d\mathcal{A}^{H}_{(v,\eta)}(\xi, \sigma) = \int \omega (\xi,\partial_{t}v-\eta X^{H}) - \sigma\int H(v),
$$
where $X^{H}$ denotes the Hamiltonian vector field on $(M ,\omega)$ induced by $H$.

Notice that every point of $M$ can be viewed as a constant loop. This yields an embedding of the manifold of constant loops $M\times \{0\}$ into the domain of the action functional:
$$
M\times \{0\} \hookrightarrow C^{\infty}(S^{1},M)\times \mathbb{R}.
$$
The critical set of $\mathcal{A}^{H}$ consists of the constant loops in $\Sigma\times\{0\}$ and the periodic orbits of the Hamiltonian flow on the hypersurface $\Sigma$:
$$
d\mathcal{A}^{H}_{(v,\eta)} = 0 \qquad \iff \qquad \left\lbrace\begin{array}{c}
                                                                           \partial_{t}v=\eta X^{H}(v)\\
                                                                            v(t)\in \Sigma\quad \forall\ t \in S^{1}
                                                                            \end{array}\right.
$$

The generators of the chain complex of Rabinowitz Floer homology are in the critical set of $\mathcal{A}^{H}$. Notice that the critical points of $\mathcal{A}^{H}$ are never isolated since $\Sigma\times\{0\} \subseteq \Crit(\mathcal{A}^{H})$, hence the functional is never Morse. In order to construct the homology we will require the Rabinowitz action functional to be Morse-Bott generically, i.e., always after a small perturbation of the Hamiltonian. In the case of a compact hypersurface $\Sigma$ this property is satisfied, as shown in Appendix B of \cite{CieliebakFrauenfelder2009}. On the other hand, if $\Sigma$ is non-compact, ensuring the Morse-Bott property of $\mathcal{A}^{H}$ poses a much more difficult challenge.
%However, in certain cases one can perturb the hypersurface $\Sigma$ within $M$ to obtain 

In order to define the boundary operator, one first has to endow $C^{\infty}(S^{1},M)$ with the $L^{2}$-Riemannian structure induced by an $\omega$-compatible $2$-parameter family of almost complex structures $\{J_{t}(\cdot,\eta)\}_{(t,\eta)\in S^{1}\times\mathbb{R}}$ on $M$.\footnote{Actually, in \cite{CieliebakFrauenfelder2009} Cieliebak and Frauenfelder work with only $t$-dependent families of almost complex structures and the dependence on $\eta$ was introduced later by Abbondandolo and Merry in \cite{AbbonMerry2014} to prove transversality explicitly.} Then we can calculate the gradient of the Rabinowitz action functional with respect to that metric:
$$
\nabla^{J} \mathcal{A}^{H}(v,\eta)=\left( \begin{array}{c}
-J_{t}(v(t),\eta)(\partial_{t}v-\eta X^{H}(v)) \\
-\int H(v) \end{array} \right),
$$
and analyze the corresponding gradient flow lines of $\mathcal{A}^{H}$. By gradient flow lines of $\mathcal{A}^{H}$ or \emph{Floer trajectories}, we mean maps $u(s,t)=(v(s,t),\eta(s)))$ from the cylinder $\mathbb{R} \times S^{1}$, endowed with coordinates $(s, t)$, to $M\times \mathbb{R}$, which are solutions to the \emph{Rabinowitz-Floer equations}
$$
\left\lbrace\begin{array}{c}
\partial_{s}v+J_{t}(v(t),\eta)(\partial_{t}v-\eta X^{H}(v))=0,\\
\partial_{s}\eta + \int H(v)=0
\end{array}\right.
$$ 

The boundary operator is defined by counting the solutions of Rabinowitz-Floer equations between distinct critical points of $\mathcal{A}^{H}$. As we stated before, the critical points of $\mathcal{A}^{H}$ are never isolated, but if we require the action functional to be Morse-Bott, we can choose an auxiliary Morse function $f$ on $\Crit(\mathcal{A}^{H})$ and consider so called \emph{flow lines with cascades}, which consists of Floer and Morse trajectories. The generators of the chain complex for Rabinowitz Floer homology are the critical points of $f$ with grading given by the sum of the Conley-Zehnder index and the signature index, whereas the boundary operator is defined by counting the flow lines with cascades between critical points of index difference $1$. However, to ensure that the boundary operator $\partial$ is well defined and $\partial^{2}=0$, one first has to prove that the moduli spaces of Floer trajectories are compact. 
%Naturally, the Floer trajectories depend on the choice of a Hamiltonian. 
In the case of a compact hypersurface $\Sigma$, $L^{\infty}$ bounds for Floer trajectories corresponding to Hamiltonians which are constant outside a compact set have been established in \cite{CieliebakFrauenfelder2009}, whereas a proof of boundedness of Floer trajectories corresponding to Hamiltonians which are radial outside a compact set can be found in \cite{Abbon2009}. The main challenge in the compact case is to prove boundedness of the Lagrange multiplier $\eta$. However, in the case of a non-compact hypersurface $\Sigma$, proving $L^{\infty}$ bounds on the loop $v$ is harder: because of the non-compactness of $\Sigma$, in fact, one cannot apply standard Floer theoretic arguments.

%Therefore, establishing $L^{\infty}$ bounds for Floer trajectories corresponding to a class of Hamiltonians with a non-compact zero level set is the main focus of this paper as a crucial step of defining Rabinowitz Floer homology for non-compact hypersurfaces. 

The main result of this paper is presented in Theorem \ref{twr:ModuliCompact}, where we establish $L^{\infty}$ bounds for Floer trajectories corresponding to a class of Hamiltonians with non-compact regular level sets: notice that these bounds represent one of the most crucial steps towards defining Rabinowitz Floer homology for non-compact hypersurfaces. Because of the additional challenges posed by considering non-compact hypersurfaces, in this paper we restrict ourselves to hypersurfaces in the standard symplectic manifold $(\mathbb{R}^{2n},\omega_{0})$ and require the Hamiltonians to satisfy suitable analytic conditions, which are introduced in section \ref{ssec:CondHam}. 

The novel approach introduced here is the method to bound the Floer trajectories along the non-compact hypersurface $\Sigma$. In Section \ref{sec:critMfld}, we analyze the behavior of Floer trajectories in the neighborhood of the non-compact component of the critical set of the Rabinowitz action functional, $\Sigma\times\{0\}\subseteq \Crit(\mathcal{A}^{H})$. Due to the Morse-Bott property of the action functional the Floer trajectories near $\Sigma\times\{0\}$ are transverse to the hypersurface. In other words, in the neighborhood of $\Sigma\times\{0\}$ the tangential component of the Floer trajectories can be bounded by the energy growth and as a result the Floer trajectories cannot escape along $\Sigma\times\{0\}$.

In order to achieve the $L^{\infty}$ bounds for the Floer trajectories we apply the Aleksandrov maximum principle: this principle has already been applied in the study of Rabinowitz Floer Homology, for example in \cite{Abbon2009}, but the original contribution of our work is that we use it in a non standard setting, where the level sets of the coercive, plurisubharmonic function intersect the level sets of our Hamiltonian.

Finally, in this paper we also introduce a very special, but also natural class of Hamiltonians, which we call \emph{tentacular} Hamiltonians: these satisfy the assumptions necessary to obtain the appropriate bounds on the moduli spaces of Floer trajectories. In fact, the conditions in the definition are even more restrictive and we claim that for this class one can actually define Rabinowitz Floer homology in such a way that it is invariant under sufficiently small compactly supported perturbations. The proof of this will be the content of a second paper.
%In the first part of the paper, we identify conditions under which it is possible to establish $L^{\infty}$ bounds for Floer trajectories corresponding to a class of Hamiltonians with non-compact regular level sets: notice that this bounds represent one of the most crucial steps towards defining Rabinowitz Floer homology for non-compact hypersurfaces.  In Theorem \ref{twr:ModuliCompact} we show that one can obtain $L^{\infty}$ bounds for Floer trajectories corresponding to Hamiltonians belonging to this class.

%\todo{Where do we introduce the conditions? Here or in the next section?}

%In this paper we consider the symplectic manifold $(\mathbb{R}^{2n},\omega_{0})$ with the canonical symplectic form. On this manifold we introduce a class of Hamiltonians called \emph{tentacular Hamiltonians}, which include examples of non-compact zero level sets. Later, we will show that one can obtain $L^{\infty}$ bounds for Floer trajectories corresponding to Hamiltonians belonging to this class.

%\todo{H1, H2, H3 \& PO?}

%% file: prel2.tex
\section{The setting}
\subsection{Moduli spaces of Floer trajectories}
Recall that in the introduction we already encountered Floer trajectories: they are the solutions of the gradient-like Rabinowitz-Floer equations. To define them precisely, we first have to equip the loop space with a metric, which comes a family of $\omega$-compatible, almost complex structures.

Let $J=\{J_{t}(\cdot, \eta)\}_{(t,\eta)\in S^{1}\times \mathbb{R}}$ be a smooth $2$-parameter family of $\omega$-compatible, almost complex structures. For compactness reasons we additionally require
$$
\sup_{(t,\eta)\in S^{1}\times \mathbb{R}}\|J_{t}(\cdot,\eta)\|_{C^{l}}<+\infty, \qquad \forall\ l\in \mathbb{N}.
$$
The set of such $2$-parameter families of $\omega$-compatible almost complex structures we denote by $\mathscr{J}(M,\omega)$. The subset of $\mathscr{J}(M,\omega)$ consisting of $2$-parameter families of $\omega$-compatible, almost complex structures fixed outside an open set $\mathcal{V} \subseteq M\times \mathbb{R}$ we denote by 
\begin{equation}
\mathscr{J}(M,\omega,\mathcal{V}).
\label{eqn:JMoV}
\end{equation}

The notion of Floer trajectories is basic in defining the boundary operator in the Rabinowitz Floer homology. However, to prove that the homology does not depend on the choice of the almost complex structure used in defining Floer trajectories, one has to introduce a notion of Floer trajectory for a homotopy of Hamiltonians and almost complex structures. Let $\{H_{s}\}_{s\in \mathbb{R}}$ be a smooth homotopy of Hamiltonians and $\{J_{s}\}_{s\in\mathbb{R}}$ be a homotopy of almost complex structures, such that for all $s\in \mathbb{R},\ J_{s} \in \mathscr{J}(M,\omega)$ and the homotopy is constant outside the interval $[0,1]$, namely
\begin{equation}
(H_{s},J_{s}) = \left\lbrace\begin{array}{l l}
(H_{0}, J_{0}) & s\leq 0,\\
(H_{1}, J_{1}) & s \geq 1.
\end{array}\right.
\label{eqn:Gamma1}
\end{equation}
Let us denote $\Gamma=\{(H_{s},J_{s})\}_{s\in \mathbb{R}}$ the associated homotopy and abbreviate
$$
\|\partial_{s}H_{s}\|_{L^{\infty}}:=\max_{\substack{x\in M,\\ s\in [0,1]}}\|\partial_{s}H_{s}(x)\|.
$$
The moduli spaces corresponding to $1$-parameter families of Hamiltonians and almost complex structures are defined as follows:
\begin{define}
Let $(M,\omega)$ be a symplectic manifold and let $\Gamma=\{(H_{s},J_{s})\}_{s\in \mathbb{R}}$ be a homotopy of Hamiltonians and almost complex structures as in (\ref{eqn:Gamma1}). A \textbf{Floer trajectory} is a solution $u\in C^{\infty}(\mathbb{R}\times S^{1}, M\times \mathbb{R})$, $u(s) = (v(s),\eta(s))$ to the equation
\begin{equation}
\partial_{s} u = \nabla^{J_{s}}\mathcal{A}^{H_{s}}(u)=\left(\begin{array}{c}
-J_{s,t}(v,\eta)(\partial_{t}v-X^{H_{s}}(v))\\
-\int_{S^{1}}H_{s}(v)dt
\end{array}\right) \label{eqn:Gamma2}.
\end{equation}
For a pair $(\Lambda_{0},\Lambda_{1})$ of connected components $\Lambda_{i} \subseteq \Crit(\mathcal{A}^{H_{i}})$ the set of all Floer trajectories converging in the limits to $\Lambda_{0}$ and $\Lambda_{1}$
$$
\lim_{s\to  -\infty}u(s)\in \Lambda_{0}, \qquad \lim_{s\to  \infty}u(s)\in \Lambda_{1},
$$
is called a \textbf{moduli space of homotopy $\Gamma$} and denoted by $\mathscr{M}^{\Gamma}(\Lambda_{0},\Lambda_{1})$.
\label{def:modGam}
\end{define}

Our aim is to prove bounds in more generality on the moduli spaces of Floer trajectories corresponding to $1$-parameter families of Hamiltonians and almost complex structures. 
%This Section presents an overview of the structure of the proof, which occupies Sections \ref{sec:ActBound} to \ref{sec:MaxPrinc}.
However, since in the parametric case the action along the Floer trajectories may not be monotonically increasing, we need to introduce the Novikov finiteness condition, which ensures that we have a bound on how much the action can decrease along a Floer trajectory.
\begin{define}
Let $\Gamma=\{(H_{s},J_{s})\}_{s\in \mathbb{R}}$ be a smooth homotopy of Hamiltonians and almost complex structures constant outside $[0,1]$. 
For a pair $a,b\in\mathbb{R}$ denote
\begin{align}
A(\Gamma,a,b) &:=\inf \left\lbrace\begin{array}{r | l}
&\exists\  \Lambda_{i}\subseteq \Crit(\mathcal{A}^{H_{i}})\quad i=0,1,\\
A \in [(-\infty,b] & \mathcal{A}^{H}(\Lambda_{0}) \geq a,\ \mathcal{A}^{H_{1}}(\Lambda_{1})=A,\\
& \mathscr{M}^{\Gamma}(\Lambda_{0},\Lambda_{1})\neq \emptyset
\end{array}\right\rbrace,\label{eqn:Aab}\\
B(\Gamma,a,b) &:=\sup \left\lbrace\begin{array}{r | l}
& \exists\ \Lambda_{i}\subseteq \Crit(\mathcal{A}^{H_{i}})\quad i=0,1,\\
B \in [a,+\infty) & \mathcal{A}^{H}(\Lambda_{0}) = B,\ \mathcal{A}^{H_{1}}(\Lambda_{1})\leq b,\\
& \mathscr{M}^{\Gamma} (\Lambda_{0},\Lambda_{1})\neq \emptyset
\end{array}\right\rbrace. \label{eqn:Bab}
\end{align}
We say that the homotopy $\Gamma$ satisfies the \textbf{Novikov finiteness condition} if for each pair $(a,b)\in CritVal(\mathcal{A}^{H_{0}})\times CritVal(\mathcal{A}^{H_{1}})$ the corresponding $A(\Gamma,a,b)$ and $B(\Gamma,a,b)$ are finite (whenever defined). 
\label{def:Nov}
\end{define}
\subsection{Conditions on the Hamiltonians}
\label{ssec:CondHam}
Given that we do not assume compactness of the energy levels, we will need to introduce additional conditions on the Hamiltonians we consider. Moreover, we will restrict ourselves in this paper to the standard symplectic manifold $(\mathbb{R}^{2n},\omega_{0})$ with the canonical symplectic form. We will be interested in smooth Hamiltonian functions $H$ on $(\mathbb{R}^{2n},\omega_{0})$ satisfying the following properties:
\begin{description}\label{def:H}
\item[\textbf{H1}\namedlabel{item:H1}{H1}]  There exists a global Liouville vector field $X^{\dagger}$ and constants $c_{1},c_{2}>0$, $ c_{3}\geq 0$ such that for all $x\in \mathbb{R}^{2n}$ the following holds true
$$
\|X^{\dagger}_{x}\| \leq  c_{1}(\|x\|+1), \qquad
dH_{x}(X^{\dagger}) \geq  c_{2}\|x\|^{2}-c_{3}.
$$
\item[\textbf{H2}\namedlabel{item:H2}{H2}] The function $H$ grows at most quadratically at infinity, that is there exists a constant $L\geq 0$ such that
$$
\sup_{x\in \mathbb{R}^{2n}} \|D^{3}H_{x} \| \cdot \|x\| \leq  L <+\infty.
$$
\item[\textbf{H3}\namedlabel{item:H3}{H3}] There exist constants $c_{4},c_{5},\nu>0$ and a Liouville vector field $X^{\ddagger}$ defined on $H^{-1}((-\nu,\nu))$, such that
\begin{align*}
 \|X^{\ddagger}(x)\| & \leq c_{4}(\|x\|+1) \qquad \forall\ x\in H^{-1}((-\nu,\nu)),\\
& \inf_{H^{-1}((-\nu,\nu))}dH(X^{\ddagger}) \geq c_{5} >0.
\end{align*}
\end{description}

We will call such Hamiltonians \textbf{admissible}.

Let us analyze the consequences of the properties of admissible Hamiltonians. 
First note that Property (\ref{item:H1}) implies linear growth of the gradient of $H$
\begin{equation}\label{eqn:nablaH2}
\|\nabla H(x)\| \geq \frac{c_{2}}{c_{1}}(\|x\|-h'_{1}) \qquad \forall\ x\in\mathbb{R}^{2n}, \qquad
\textrm{where} \qquad h'_{1} =1+\frac{|c_{2}-c_{3}|}{c_{2}}.
\end{equation}
In a similar way Property (\ref{item:H2}) implies quadratic behavior of the Hamiltonian for all $x\in\mathbb{R}^{2n}$ we have
$$
\|\Hess_{x}H\| \leq M,\qquad \|\nabla H(x)\|  \leq h_{1}+M\|x\|,\qquad |H(x)|  \leq H+h_{1}\|x\|+\frac{1}{2}M\|x\|^{2},
$$
where $M:=\|\Hess_{0}H\|+L$, $H: = |H(0)|$, $h_{1}:=\|\nabla H(0)\|$.

At last, observe that Property (\ref{item:H3}) ensures that $0$ is a regular value of $H$ and $H^{-1}(0)$ is of contact type.

The following property will be very important for us: given an admissible Hamiltonian $\textsf{H}$ and a compact set $K\subset \mathbb{R}^{2n}$ we can find an open neighborhood of $\textsf{H}$ in the affine space $\textsf{H}+C^{\infty}_{0}(K)$ consisting of admissible Hamiltonians, which satisfy all the properties of Definition \ref{def:H} with one set of parameters chosen uniformly for the whole neighborhood.

\begin{lem}
Let $\textsf{H}:\mathbb{R}^{2n}\to \mathbb{R}$ be an admissible Hamiltonian. Then for any compact set $K\subseteq \mathbb{R}^{2n},\ K\neq \emptyset$, there exists an open subset $\mathscr{O}(\textsf{H})\subseteq C^{\infty}_{0}(K)$, such that all the Hamiltonians from the set
$$
\textsf{H} + \mathscr{O}(\textsf{H}):=\{\textsf{H}+h\ |\ h\in \mathscr{O}(\textsf{H})\}
$$
are admissible and the set of parameters can be chosen uniformly for the whole set $\textsf{H} + \mathscr{O}(\textsf{H})$.
\label{lem:open}
\end{lem}
\begin{proof}
Let $c_{1},c_{2},c_{3},c_{4},c_{5}, \nu$ be the constants and $X^{\dagger},X^{\ddagger}$ be the Liouville vector fields associated to $\textsf{H}$ through Properties (\ref{item:H1}), (\ref{item:H2}) and (\ref{item:H3}). Let $h\in C^{\infty}_{0}(K)$. Then from Property (\ref{item:H1}) of $\textsf{H}$ one has
\begin{align*}
 d(\textsf{H}+h)_{x}(X^{\dagger})  & = d(\textsf{H})_{x}(X^{\dagger})+dh_{x}(X^{\dagger}), \\
  &  \geq c_{2}\|x\|^{2}-c_{3} - \|h\|_{C^{1}(K)}\sup_{x\in K}\|X^{\dagger}\| ,\\
  & \geq c_{2}\|x\|^{2}-c_{3} - \|h\|_{C^{1}(K)}c_{1}(\sup_{x\in K}\|x\|+1).
\end{align*}
On the other hand, from Property (\ref{item:H2}) of $\textsf{H}$ one obtains
\begin{align*}
  \sup_{x\in\mathbb{R}^{2n}}\|D^{3}(\textsf{H}+h)\|\ \|x\| &\leq \|h\|_{C^{3}(K)} \sup_{x\in K}\|x\|+\sup_{x\in \mathbb{R}^{2n}}\|D^{3}(\textsf{H})\|\ \|x\|.
\end{align*}

Now let us analyze Property (\ref{item:H3}). Let $c_{4},c_{5},\nu$ be the constants and $X^{\ddagger}$ is the Liouville vector field associated to $\textsf{H}$ by Property (\ref{item:H3}).

Observe that for every $x\in \mathbb{R}^{2n}$ and $h\in C^{\infty}_{0}(K)$ one has
$$
|\textsf{H}(x)| \leq |(\textsf{H}+h)(x)|+|h(x)| \leq |(\textsf{H}+h)(x)|+\|h\|_{C(K)}.
$$
That means that whenever $\|h\|_{C(K)}<\frac{1}{2}\nu$, then we have the following inclusion
$$
(\textsf{H}+h)^{-1}\left(-\frac{1}{2}\nu,\frac{1}{2}\nu\right)\subseteq \textsf{H}^{-1}(-\nu,\nu).
$$
By property (\ref{item:H3}) we have that for $x\in (\textsf{H}+h)^{-1}(-\frac{1}{2}\nu,\frac{1}{2}\nu)$ we have
\begin{align*}
d(\textsf{H}+h)_{x}(X^{\ddagger}) & \geq d(\textsf{H})_{x}(X^{\ddagger})+dh_{x}(X^{\ddagger})\\
& \geq d(\textsf{H})_{x}(X^{\ddagger})+dh_{x}(X^{\ddagger})\\
& \geq  \inf_{\textsf{H}^{-1}((-\nu,\nu))}d(\textsf{H})_{x}(X^{\ddagger})-\|h\|_{C^{1}(K)}\sup_{K\cap \textsf{H}^{-1}(-\nu,\nu)}\|X^{\ddagger}\|\\
& = c_{5}-c_{4}\|h\|_{C^{1}(K)}(\sup_{K}\|x\|+1).
\end{align*}
Now if we denote 
\begin{equation}
\theta:=\frac{1}{2}\min\left\lbrace \nu,\frac{c_{5} }{c_{4}(\sup_{K}\|x\|+1)}\right\rbrace, \label{eqn:theta}
\end{equation}
and 
$$
\mathscr{O}(\textsf{H}):=\{h \in C^{\infty}_{0}\ |\ \|h\|_{C^{3}(K)}<\theta\},
$$
then for all $h\in \mathscr{O}(\textsf{H})$, we will have
\begin{align*}
 d(\textsf{H}+h)_{x}(X^{\dagger})  &  \geq c_{2}\|x\|^{2}-c_{3} - \theta \sup_{x\in K}\|X^{\dagger}\|\qquad \forall x\in \mathbb{R}^{2n},\\
  \sup_{x\in\mathbb{R}^{2n}}\|D^{3}(\textsf{H}+h)\|\ \|x\| &\leq \theta \sup_{x\in K}\|x\|+\sup_{x\in \mathbb{R}^{2n}}\|D^{3}(\textsf{H})\|\ \|x\|, \\
\end{align*}
and for all $x\in (\textsf{H}+h)^{-1}(-\frac{1}{2}\nu,\frac{1}{2}\nu)$ we have
$$
d(\textsf{H}+h)_{x}(X^{\ddagger}) \geq c_{5}-c_{4}\|h\|_{C^{1}(K)}(\sup_{K}\|x\|+1) \geq \frac{1}{2}c_{5}>0.
$$
In particular, we can choose the parameters uniformly for the whole family. In other words, the set $\textsf{H}+\mathscr{O}(\textsf{H})$ consists of admissible Hamiltonians and the properties (\ref{item:H1}), (\ref{item:H2}), and (\ref{item:H3}) are satisfied with parameters, which depend only on $\textsf{H}$ and $K$.

Moreover, for all $H\in \textsf{H}+\mathscr{O}(\textsf{H})$ we have uniform quadratic behavior: that is, $\forall\ x\in\mathbb{R}^{2n}$, 
\begin{align}
\|\Hess_{x}H\| &\leq M, \label{eqn:HessH}\\
\|\nabla H(x)\| & \leq h_{1}+M\|x\|,\label{eqn:nablaH}\\
| H(x)| & \leq h_{0}+h_{1}\|x\|+\frac{1}{2} M\|x\|^{2}, \label{eqn:H0}\\
\textrm{where} \quad M:= \theta +\|\Hess_{0}\textsf{H}\| &+ L, \quad h_{0}:=\theta +|\textsf{H}(0)|, \quad h_{1}:=  \theta+\|\nabla \textsf{H}(0)\|.\nonumber
\end{align}
\end{proof}

In order to formulate the main theorem of this paper, we will need to introduce one more assumption. In particular, we can formulate this assumption in a more general setting, of any exact, symplectic manifold $(M,\omega)$, not necessarily $(\mathbb{R}^{2n},\omega_{0})$. Let $H$ be a Hamiltonian on an exact, symplectic manifold $(M,\omega)$, having $0$ as a regular value. We say that $H$ satisfies
\begin{description}
\item[\textbf{PO}\namedlabel{item:PO}{PO}] if for any fixed action window all the non-degenerate periodic orbits are contained in a compact subset of $M$. In other words, for any $n\in \mathbb{N}$ there exists a compact subset $K_{n}\subseteq M$, such that whenever 
$$
(v,\eta)\in \Crit(\mathcal{A}^{H})\quad \textrm{and} \quad  0<|\mathcal{A}^{H}(v,\eta)|\leq n,
$$
then $v(S^{1})\subseteq K_{n}$. 
\end{description}
Note that, this is a very natural assumption, which assures that all the Floer trajectories from an action window, start or end in a compact set. In fact, as we will show in the next subsection, (\ref{item:PO}) along with (\ref{item:H1}), (\ref{item:H2}) and (\ref{item:H3}) are sufficient to assure uniform bounds on the moduli space of Floer trajectories. However, property (\ref{item:PO}) does not persist under compact perturbations. Therefore, to assure openness of our properties, we will introduce an additional, even more restrictive condition, under which we claim one can actually define Rabinowitz Floer homology in such a way that it is invariant under sufficiently small compactly supported perturbations.

Let $H$ be a Hamiltonian on an exact, symplectic manifold $(M,\omega)$, having $0$ as a regular value. We say that property
\begin{description}
\item[\ref{item:PO} is \textbf{uniformly continuous} at $H$] if there exists an an open neighborhood $B(H)$ of $0$ in $C^{\infty}_{c}(M)$ and an exhaustion $\{K_{n}\}_{n\in\mathbb{N}}$ of $M$ by compact sets $K_{n}$, such that for every $n\in \mathbb{N}$ and every $h\in B(K_{n})=B(H)\cap C^{\infty}_{0}(K_{n})$, whenever
$$
(v,\eta)\in \Crit(\mathcal{A}^{H+h})\quad\textrm{and}\quad 0<|\mathcal{A}^{H+h}(v,\eta)|\leq n,
$$
then $v(S^{1})\subseteq K_{n}$. In other words, all the nondegenerate periodic orbits of the perturbed Hamiltonian within the action window $[-n,n]$, are contained in $K_{n}$. 
\end{description}
We will call a Hamiltonian $H$ on $(\mathbb{R}^{2n},\omega_{0})$ a \textbf{tentacular\footnote{The Floer trajectories as embedded cylinders escaping along the non-compact hypresurface to infinity reminded me of tentacles of a "moduli monster" thus inspiring the name \textit{tentacular Hamiltonians}. The other reason for the name was an assumption, we have made earlier, on the hypersurface to be topologically a compact set with cylindrical ends. However, later in our work we have discovered that this assumption could be omitted, yet the name \textit{tentacular Hamiltonians} stayed as we grew attached to it.} Hamiltonian}, whenever $H$ is admissible and (\ref{item:PO}) is uniformly continuous at $H$. Since by Lemma \ref{lem:open} the set of admissible Hamiltonians is open under compactly supported perturbations and uniform continuity of (\ref{item:PO}) is by definition an open property, the set of tentacular Hamiltonians is also open under compactly supported perturbations. This way for tentacular Hamiltonians we will not only be able to bound the Floer trajectories, but also to define the Rabinowitz Floer homology, which will be the subject of our followup paper.

%A fact worth mentioning is that Property (\ref{item:SG}) does not play a major part in the proof and is necessary only to achieve boundedness for the moduli spaces between $(\mathcal{A}^{H_{0}})^{-1}(0)$ and $(\mathcal{A}^{H_{1}})^{-1}(0)$ in the case of a family of Hamiltonians. In particular, if we are interested in obtaining the bounds for a fixed Hamiltonian, then the assertions of Theorem \ref{twr:ModuliCompact} still hold without assumption (\ref{item:SG}).

\subsection{Main theorem and outline of the proof}
We can finally formulate the main theorem of this paper, namely the existence of uniform bounds on the moduli space of Floer trajectories. 
Since in future work we would also like to ensure that these bounds lead to the definition of a homology theory that is invariant under small perturbations, we will aim at proving them in more generality, namely for a homotopy of Hamiltonians and almost complex structures. 

We have proved in the previous section that for a fixed, admissible Hamiltonian $H_{0}:\mathbb{R}^{2n}\to \mathbb{R}$ and a compact set $K\subseteq \mathbb{R}^{2n},\ K\neq \emptyset$ by Lemma \ref{lem:open} there exists an open neighborhood $\mathscr{O}(H_{0})$ of $0$ in $C_{0}^{\infty}(K)$, such that all the Hamiltonians from $H_{0}+\mathscr{O}(H_{0})$ are also admissible and the set of parameters can be chosen uniformly for the whole family $H_{0}+\mathscr{O}(H_{0})$. As a result, the constants $\tilde{c}$ and $\varepsilon_{0}$ calculated in Lemma \ref{lem:tilde} can also be chosen uniformly for the whole family $H_{0}+\mathscr{O}(H_{0})$.

To formulate Theorem \ref{twr:ModuliCompact} we will introduce the following notation: for every compact subset $N\subseteq H^{-1}(0)$, denote
\begin{equation}
\mathscr{C}(\mathcal{A}^{H},N) := \{ (v,\eta) \in \Crit(\mathcal{A}^{H})\ |\ |\mathcal{A}^{H}(v,\eta)|>0\ \textrm{or}\ (v,\eta)\in N\times \{0\}\}.
\label{eqn:CAHY}
\end{equation}
Moreover, recall Definition \ref{def:modGam} of a moduli space of Floer trajectories for a homotopy of Hamiltonians and almost complex structures $\Gamma=\{H_{s},J_{s}\}_{s \in \mathbb{R}}$, which satisfies (\ref{eqn:Gamma1}) and the notation for the set of $2$-parameter families of $\omega$-compatible, almost complex structures fixed outside an open set as defined in (\ref{eqn:JMoV}). Now we are ready to present the main theorem of this paper:
\begin{twr}
\label{twr:ModuliCompact}
Fix an admissible Hamiltonian $\mathsf{H}:\mathbb{R}^{2n}\to \mathbb{R}$; a compact set $K\subseteq \mathbb{R}^{2n},\ K\neq \emptyset$; an open, precompact subset $\mathcal{V} \subseteq \mathbb{R}^{2n}$ and a constant \etat$>0$. Let $\mathscr{O}(\mathsf{H})\subseteq C^{\infty}_{0}(K)$ be the open set associated to $\mathsf{H}$ and $K$ as in Lemma \ref{lem:open} and let $\Gamma=\{H_{s},J_{s}\}_{s \in \mathbb{R}}$ be a smooth homotopy of Hamiltonians and $\omega$-compatible almost complex structures, constant outside of $[0,1]$ (i.e. $\Gamma$ satisfies (\ref{eqn:Gamma1})) and such that for all $s\in\mathbb{R}$
$$
H_{s} \in \mathsf{H}+\mathscr{O}(\mathsf{H}),\qquad
J_{s}=\{J_{t,s}(\cdot,\eta)\}_{S^{1}\times\mathbb{R}} \in \mathscr{J}\Big(\mathbb{R}^{2n},\omega_{0},\mathcal{V}\times\big((-\infty,-\textrm{\etat})\cup( \textrm{\etat}, \infty)\big)\Big).
$$
Assume $H_{0}$ satisfies Property (\ref{item:PO}). There exist constants $\tilde{c}<\infty$ and $\varepsilon_{0}>0$ depending only on $\mathsf{H}$ and $K$, such that if the homotopy $\Gamma$ satisfies
\begin{equation}
\Big(\tilde{c}+\frac{\|J\|_{L^{\infty}}^{\frac{3}{2}}}{\varepsilon_{0}}\Big)\|\partial_{s}H_{s}\|_{L^{\infty}} < \frac{1}{8},
\label{eqn:Hs2}
\end{equation}
and we fix $a,b\in \mathbb{R}$ and a compact subset $N\subseteq H^{-1}_{0}(0)$, then for each pair $(\Lambda_{0},\Lambda_{1})$ of connected components
\begin{equation}\label{eqn:lambda}
\Lambda_{0} \subseteq \mathscr{C}(\mathcal{A}^{H_{0}},N)\cap (\mathcal{A}^{H_{0}})^{-1}([a,\infty)), \qquad
\Lambda_{1} \subseteq \Crit(\mathcal{A}^{H_{1}})\cap (\mathcal{A}^{H_{1}})^{-1}((-\infty,b]).
\end{equation}
the associated moduli space $\mathscr{M}^{\Gamma}(\Lambda_{0},\Lambda_{1}) \subseteq C^{\infty}(\mathbb{R} \times S^{1},\mathbb{R}^{2n} \times \mathbb{R})$ is contained in a compact set in $\mathbb{R}^{2n+1}$.%, which depends only on $\Lambda_{0},\Lambda_{1},\ \|J\|_{L^{\infty}}$ and the family $\{H_{s}\}_{s\in \mathbb{R}}$ itself.
\end{twr}

Note that if we take $\Gamma$ to be a constant homotopy, i.e $(H_{s},J_{s})=(H,J)$ for all $s \in \mathbb{R}$, where $H$ is admissible, whereas $J\in \mathscr{J}\Big(\mathbb{R}^{2n},\omega_{0},\mathcal{V}\times\big((-\infty,-\textrm{\etat})\cup( \textrm{\etat}, \infty)\big)\Big)$, then condition (\ref{eqn:Hs2}) is automatically satisfied and the assertions of the theorem hold true.

Secondly, you can observe that Theorem \ref{twr:ModuliCompact} is formulated in a very general way, for homotopies of Hamiltonians and almost complex structures contained in the sets $\mathsf{H}+\mathscr{O}(\mathsf{H})$ and $\mathscr{J}(\mathbb{R}^{2n},\omega_{0},\mathcal{V}\times((-\infty,-\textrm{\etat})\cup( \textrm{\etat}, \infty)))$ respectively. Moreover, throughout the proof we have explicitly calculated all the parameters, in particular the bounding constants. This may seem at first superfluous, but it was in fact done on purpose, in order to trace the dependence of the bounding constants on the parameters of the homotopy. If we now look more closely at the established parameters, we can see that the bounds do not depend on the choice of the homotopy, but only on the set $\mathsf{H}+\mathscr{O}(\mathsf{H})$ and on the limit of norms of the almost complex structures $\|J\|_{L^{\infty}}$ as in (\ref{eqn:Jinfty}). This way we can establish uniform bounds for continuous families of homotopies, so called \textit{homotopies of homotopies}, which are necessary to prove the invariance of the Rabinowitz Floer homology of perturbations in $H$ and of the choice of the almost complex structure. For a reader familiar with the Floer techniques this generality may seem unnecessary, since in the classical, compact setting one can deduce uniform bounds on a homotopy of homotopies using the isolating neighborhood theorem as presented in Theorem 3, \cite{Floer1989}. However, this technique cannot be applied directly in our setting, where the Rabinowitz action functional is Morse-Bott and its critical set has a non-compact component. This is due to the fact that an isolating neighborhood in the sense of Definition 1b.1 of \cite{Floer1989} does not exist for Floer trajectories crossing the non-compact component of the critical set of the Rabinowitz action functional along hypersurface $\Sigma\times\{0\}$. Unable to cite and apply the isolating neighborhood theorem as in \cite{Floer1989}, we had to establish all the uniform bounds explicitly.

The general outline of the proof Theorem 1 is as follows: the first step is to prove global bounds on the action, energy and $\eta$ parameter for the whole moduli space, which is explained in Section \ref{sec:ActBound}. This section also includes a proof that our system satisfies the Novikov finiteness condition.

The second step is to localize the Floer trajectory. In Section \ref{sec:ActDeriv} we  introduce the set of infinitesimal action derivation and divide it into two parts - a compact set and a tubular neighbourhood of the hypersurface $\Sigma\times \{0\}$. In Section \ref{sec:critMfld} we show how to obtain bounds on the $L^{2}\times\mathbb{R}$ norm in the tubular neighbourhood of the hypersurface, using the Morse-Bott property of the action functional. This way we are able to divide the loop space into three sets, on each one bounding the $L^{2}\times\mathbb{R}$ norm of a Floer trajectory differently.

Afterwards, we will analyze how the the Floer trajectory oscillates between those three sets and establish global $L^{2}\times\mathbb{R}$ bounds on the moduli spaces: this is explained in Section \ref{sec:oscillations}.

Finally, in Section \ref{sec:MaxPrinc}, we will establish the $L^{\infty}\times \mathbb{R}$ bounds, using Aleksandrov's maximum principle.

%The general outline of the proof Theorem \ref{twr:ModuliCompact} is as follows:
%the first step will be to prove global bounds on the action, energy and $\eta$ parameter for the whole moduli space, and is the content of %Section \ref{sec:ActBound}. Secondly, we will explain how to divide the space into three sets and bound the $L^{2}\times\mathbb{R}$ norm on %each set differently, which is the subject of Sections \ref{sec:ActDeriv} and \ref{sec:critMfld}. Afterwards, we will try to localize the Floer trajectory %and analyze how the it oscillates between those sets, which is explained in Section \ref{sec:oscillations}. Finally, in Section \ref{sec:MaxPrinc}, %we will establish the $L^{\infty}\times \mathbb{R}$ bounds using the maximum principle.

%% file: ActBound3A.tex
\section{Bounds on the action}
\label{sec:ActBound}

In this section we establish uniform bounds on action, energy and $\eta$ parameter along perturbed Floer trajectories corresponding to admissible Hamiltonians and we prove that the Novikov finiteness condition holds for homotopies as in the statement of Theorem \ref{twr:ModuliCompact}.

We first prove that if we fix a pair of connected components $\Lambda_{0},\Lambda_{1}$ as in (\ref{eqn:lambda}) then the corresponding moduli space $\mathscr{M}^{\Gamma}(\Lambda_{0},\Lambda_{1})$ admits a global bound on the action, energy and $\eta$ parameter. 
In the case of a constant family of Hamiltonians $H_{s}=H$ the action and the energy of $u\in \mathscr{M}^{\Gamma}(\Lambda_{0},\Lambda_{1})$ are trivially bounded namely
$$
\frac{d}{ds} (\mathcal{A}^{H} \circ u)(s) = \|\partial_{s}u(s)\|_{L^{2}\times \mathbb{R}}^{2}\geq 0,
$$
and therefore we have the relations
$$
 a \leq \mathcal{A}^{H}(u(s))  \leq b,\qquad
 \textrm{and} \qquad \|\nabla \mathcal{A}^{H} (u)\|_{L^{2}(\mathbb{R}\times S^{1})}^{2}  =\|\partial_{s}u\|_{L^{2}(\mathbb{R}\times S^{1})}^{2}  \leq b-a.
$$
However, in the case of a non-constant family $\{H_{s}\}_{s\in \mathbb{R}}$
$$
\frac{d}{ds} (\mathcal{A}^{H_{s}} \circ u) (s)=  \|\nabla^{J_{s}} \mathcal{A}^{H_{s}}(u(s))\|^{2}_{g_{J_{s}}} + \eta(s)\int_{0}^{1}\partial_{s}H_{s}(v(s,t))dt,
$$
which makes it more difficult to prove the bounds on the energy and the action. Nevertheless, we have obtained the bounds by following the approach presented in \cite{CieliebakFrauenfelder2009}, i.e. by first proving in Lemma \ref{lem:tilde} that all Hamiltonians in $H_{0}+\mathscr{O}(H_{0})$ satisfy a linearity condition between the action and the $\eta$ parameter
$$
|\eta|\leq \tilde{c}(|\mathcal{A}^{H}(v,\eta)|+1),
$$
and then in Proposition \ref{prop:bounds1} obtaining the bounds on the action, energy and also $\eta$ parameter uniform for the whole $\mathscr{M}^{\Gamma}(\Lambda_{0},\Lambda_{1})$  provided the homotopy $\Gamma=\{(H_{s},J_{s})\}_{s\in\mathbb{R}}$ satisfies inequality (\ref{eqn:Hs2}).

We start by proving some inequalities which hold for admissible Hamiltonians. The following observation will be important in the proof of the lemmas and the proposition below. Let $Y$ be a Liouville vector field, then the Rabinowitz action functional satisfies:
\begin{align}
\mathcal{A}^{H}(v,\eta)-\mathcal{A'}^{H}_{(v,\eta)}(Y, \eta) & = 
\int \lambda(\partial_{t}v)-\eta\int H(v)
-\Big(\int \omega_{0}(Y,\partial_{t}v-\eta X^{H}) - \eta\int H(v)\Big) \nonumber\\
& =  \int \lambda(\partial_{t}v) - \int \lambda(\partial_{t}v-\eta X^{H}) \nonumber\\
& =  \eta \int \lambda(X^{H}) \nonumber \\
& =  \eta \int \omega_{0}(Y,X^{H}) \nonumber \\
& =  \eta \int dH(Y). \label{eqn:dH(Y)}
\end{align}

In this section we fix an admissible Hamiltonian $\mathsf{H}$ and a compact set $K\subseteq \mathbb{R}^{2n},\ K\neq \emptyset$. We denote by $\mathscr{O}(\mathsf{H})\subseteq C^{\infty}_{0}(K)$ the open set given by Lemma \ref{lem:open}.

\begin{lem}
There exists $\varepsilon_{0}>0$ and $\tilde{c}>0$, depending only on $\textsf{H}$ and $K$, such that for every $H \in \mathsf{H}+\mathscr{O}(\mathsf{H})$ whenever $(v,\eta)\in C^{\infty}(S^{1},\mathbb{R}^{2n})\times \mathbb{R}$ satisfies
$$
\|\nabla \mathcal{A}^{H}(v,\eta)\|_{L^{2}\times\mathbb{R}}<\varepsilon_{0},
$$
then
$$
|\eta|\leq \tilde{c}(|\mathcal{A}^{H}(v,\eta)|+1).
$$
\label{lem:tilde}
\end{lem}
\begin{proof}
We will show that the assertion holds true for
$$
\varepsilon_{0}:=\min\Big\{\frac{c_{5}}{2},\frac{\nu}{2}\min\Big\{1,\frac{1}{M \textsl{v}_{0}+h_{1}}\Big\}\Big\}\qquad 
\textrm{and} \qquad \tilde{c} :=\max\Big\{\frac{2}{c_{5}},c_{4}(\textsl{v}_{0}+1)\Big\},
$$
 where the constants are associated to $\mathsf{H}+\mathscr{O}(\mathsf{H})$ through Lemma \ref{lem:open} and $\textsl{v}_{0}$ depends only on these constants. % and is defined in (\ref{eqn:v0}).
 
\textbf{Step 1:}\\
Applying (\ref{eqn:dH(Y)}) to the Liouville vector field $X^{\dagger}$ associated to $H$ by Property (\ref{item:H1}) we obtain
$$
|\mathcal{A}^{H}(v,\eta)|+\|\nabla \mathcal{A}^{H}(v,\eta)\|_{L^{2}\times\mathbb{R}}(c_{1}(\|v\|_{L^{2}}+1)+|\eta|)  \geq |\eta|(c_{2}\|v\|^{2}_{L^{2}}-c_{3}).
$$
Therefore whenever $\|\nabla \mathcal{A}^{H}(v,\eta)\|_{L^{2}\times\mathbb{R}}<\varepsilon_{0}\leq \frac{c_{5}}{2}$ this implies
$$
|\mathcal{A}^{H}(v,\eta)|+\frac{c_{5}}{2}(c_{1}(\|v\|_{L^{2}}+1)+|\eta|)  \geq |\eta|(c_{2}\|v\|^{2}_{L^{2}}-c_{3}),
$$
and thus with $\|v\|_{L^{2}}$ large enough, that is
$$
\|v\|_{L^{2}}>\sqrt{\frac{2c_{3}+c_{5}}{2c_{2}}},
$$
the following inequality holds 
$$
\frac{|\mathcal{A}^{H}(v,\eta)|+\frac{c_{1}c_{5}}{2}(\|v\|_{L^{2}}+1)}{c_{2}\|v\|^{2}_{L^{2}}-c_{3}-\frac{c_{5}}{2}} \geq |\eta|.
$$
In particular, there exists large enough $\textsl{v}_{0}>0$ depending only on $\mathsf{H}+\mathcal{O}(\mathsf{H})$, such that for all $\textsl{v}\geq \textsl{v}_{0}$
$$
\frac{1}{c_{2}\textsl{v}^{2}-c_{3}-\frac{c_{5}}{2}}\leq \frac{2}{c_{5}} \quad \textrm{and} \quad \frac{c_{1}c_{5}(\textsl{v}+1)}{2c_{2}\textsl{v}^{2}-2c_{3}-c_{5}}\leq c_{4}.
$$
This leads to the conclusion that whenever
$$
\|\nabla \mathcal{A}^{H}(v,\eta)\|_{L^{2}\times\mathbb{R}}< \frac{c_{5}}{2} \quad \textrm{and} \quad \|v\|_{L^{2}}\geq \textsl{v}_{0},
$$
then
$$
|\eta| \leq \frac{2}{c_{5}}|\mathcal{A}^{H}(v,\eta)|+c_{4}.
$$
\textbf{Step 2:}\\
Let us now take $(v,\eta)\in C^{\infty}(S^{1},\mathbb{R}^{2n})\times \mathbb{R}$, such that
$$
\|\nabla \mathcal{A}^{H}(v,\eta)\|_{L^{2}\times\mathbb{R}}<\varepsilon_{0},\quad \textrm{and}\quad
\|v\|_{L^{2}} \leq \textsl{v}_{0},
$$
The fact that
$$
\varepsilon_{0} < \frac{\nu}{2}\min\Big\{1,\frac{1}{M \textsl{v}_{0}+h_{1}}\Big\},
$$
by Lemma \ref{lem:nbhd_close} implies that
$$
v(t)\in H^{-1}(-\nu,\nu) \qquad \forall\ t\in[0,1].
$$
Therefore, we can assume that the Liouville vector field $X^{\ddagger}$ of the Hamiltonian $H$ is defined along the whole loop $v$. Applying (\ref{eqn:dH(Y)}) to $X^{\ddagger}$ and using Property (\ref{item:H3}) we get the following relation
\begin{align*}
%\mathcal{A}^{H}(v,\eta)-\mathcal{A'}^{H}_{(v,\eta)}(X^{\ddagger},\eta)  & = \eta  \int dH(X^{\ddagger}),\\
|\mathcal{A}^{H}(v,\eta)|+\|\nabla\mathcal{A}^{H}(v,\eta)\|_{L^{2}\times\mathbb{R}}(c_{4}(\|v\|_{L^{2}}+1)+|\eta|) & \geq c_{5} |\eta|, \\
|\mathcal{A}^{H}(v,\eta)|+\varepsilon_{0} c_{4} (\textsl{v}_{0}+1) & \geq (c_{5}-\varepsilon_{0})|\eta|.
\end{align*}
Since by assumption $\varepsilon_{0} < \frac{c_{5}}{2}$, we obtain the claimed relation namely
$$
|\eta| \leq \frac{2}{c_{5}} |\mathcal{A}^{H}(v,\eta)|+ c_{4}( \textsl{v}_{0}+1).
$$
\end{proof}

\begin{lem}
For every $\mu>0$, there exists $\varepsilon>0$, such that for every Hamiltonian $H \in \mathsf{H}+\mathscr{O}(\mathsf{H})$ if $(v,\eta)\in C^{\infty}(S^{1},\mathbb{R}^{2n})\times \mathbb{R}$ satisfies
$$
\|\nabla \mathcal{A}^{H}(v,\eta)\|_{L^{2}\times\mathbb{R}}<\varepsilon
$$
then there exists $t_{0}\in [0,1]$ such that $v(t_{0})\in  H^{-1}(-\tfrac{\mu}{2},\tfrac{\mu}{2})$.
 Moreover, if we assume $\|v\|_{L^{2}} \leq \textsl{v}$ then taking a possibly smaller $\varepsilon$ depending only on $\mu,\textsl{v},\mathsf{H}$ and $K$, we get that
$$
v(t)\in H^{-1}(-\mu,\mu)\quad \forall\ t\in [0,1].
$$
\label{lem:nbhd_close}

\end{lem}
\begin{proof}
Assume that $v\in C^{\infty}(S^{1},\mathbb{R}^{2n})$ has the property that $|H(v(t))|\geq \frac{\mu}{2}$ for all $t\in[0,1]$. Then
$$
\|\nabla \mathcal{A}^{H}(v,\eta)\|_{L^{2}\times\mathbb{R}}\geq\Big|\int_{0}^{1}H(v(t))dt\Big|\geq \frac{\mu}{2}.
$$
In particular, if we take $|\nabla \mathcal{A}^{H}(v,\eta)\|<\varepsilon\leq \frac{\mu}{2}$ we get a contradiction, unless there exists a $t_{0}\in [0,1]$ such that
$$
v(t_{0})\in H^{-1}(-\tfrac{\mu}{2},\tfrac{\mu}{2}).
$$

Let us now prove the second part of the lemma. If $\|v\|_{L^{2}}\leq \textsl{v}$, then by (\ref{eqn:nablaH}) we have that the gradient of $H$ is uniformly bounded 
$$
\|\nabla H(v)\|_{L^{2}} \leq h_{1}+M \|v\|_{L^{2}} \leq h_{1}+M \textsl{v}.
$$
Therefore, for every $t\in[0,1]$
\begin{align*}
(h_{1}+M \textsl{v})\ \|\nabla \mathcal{A}^{H}(v,\eta)\|_{L^{2}\times\mathbb{R}} & \geq \|\nabla H(v)\|_{L^{2}} \|\partial_{t}v-\eta X^{H}(v)\|_{L^{2}}\\
& \geq \int_{0}^{1} \|\nabla H(v(\tau))\|\ \|\partial_{t}v(\tau)-\eta X^{H}(v(\tau))\|d\tau\\
& \geq \int_{t_{0}}^{t} \|\nabla H(v(\tau))\|\ \|\partial_{t}v(\tau)-\eta X^{H}(v(\tau))\|d\tau\\
& \geq \int_{t_{0}}^{t} |\langle \nabla H(v(\tau)),\partial_{t}v(\tau)-\eta X^{H}(v(\tau))\rangle|d\tau\\
& = \int_{t_{0}}^{t} |dH(\partial_{t}v(\tau))|d\tau\\
& \geq \Big|\int_{t_{0}}^{t} dH(\partial_{t}v(\tau))d\tau\Big|\\
& = |H(v(t))-H(v(t_{0}))|\\
& \geq \big||H(v(t))|-|H(v(t_{0}))|\big|.
\end{align*}
So now if we take 
$$
\varepsilon \leq \frac{\mu}{2}\min\Big\{1,\frac{1}{M \textsl{v}+h_{1}}\Big\},
$$
then for $(v,\eta)\in C^{\infty}(S^{1},\mathbb{R}^{2n})\times \mathbb{R}$, such that
$\|\nabla \mathcal{A}^{H}(v,\eta)\|<\varepsilon$ and $\|v\|_{L^{2}}\leq \textsl{v}$ we will have that 
$$
\frac{\mu}{2}  \geq   \varepsilon(M \textsl{v}+h_{1})  > \big||H(v(t))|-|H(v(t_{0}))|\big|,\qquad
\textrm {and} \qquad v(t_{0}) \in H^{-1}(-\tfrac{\mu}{2},\tfrac{\mu}{2}),
$$
so in fact
$$
v(t) \in H^{-1}(-\mu,\mu) \qquad \forall \ t\in [0,1].
$$
\end{proof}

%\label{sec:ActBound}
\begin{prop}
Let $\Gamma=\{H_{s},J_{s}\}_{s \in \mathbb{R}}$ be a smooth homotopy of Hamiltonians and almost complex structures, constant outside $[0,1]$, such that $\Gamma$ satisfies inequality (\ref{eqn:Hs2}) and
$$
\forall\ s\in \mathbb{R}\qquad (H_{s},J_{s}) \in (\mathsf{H}+\mathscr{O}(\mathsf{H}))\times \mathscr{J}(\mathbb{R}^{2n},\omega_{0}).
$$

Let $u\in C^{\infty}(\mathbb{R}, C^{\infty}(S^{1}, \mathbb{R}^{2n}) \times \mathbb{R})$ be a Floer trajectory of the corresponding time-dependent action functionals $\mathcal{A}^{H_{s}}$, such that
$$
\lim_{s\to -\infty}\mathcal{A}^{H_{s}}(u(s))\geq a \quad \textrm{and} \quad \lim_{s\to \infty}\mathcal{A}^{H_{s}}(u(s))\leq b.
$$
Then $\|\eta\|_{L^{\infty}(\mathbb{R})},\ \|\mathcal{A}^{H_{s}}(u)\|_{L^{\infty}(\mathbb{R})}$ and
$\|\nabla^{J_{s}} \mathcal{A}^{H_{s}} (u)\|_{L^{2}(\mathbb{R}\times S^{1})}$
are uniformly bounded by constants which depend only on $a,b$ the set of Hamiltonians $\mathsf{H}+\mathscr{O}(\mathsf{H})$ and continuously on $\|J\|_{L^{\infty}}=\max_{s\in[0,1]}\|J_{s}\|_{L^{\infty}}$.
\label{prop:bounds1}
\end{prop}
\begin{proof}
Before we start proving observe that $J_{s} \in \mathscr{J}(\mathbb{R}^{2n},\omega_{0})$ is a continuous family with respect to $s$, constant outside of $[0,1]$ as a consequence we have a limit on
\begin{equation}
\|J\|_{L^{\infty}}=\max_{s\in[0,1]}\|J_{s}\|_{L^{\infty}}<+\infty.
\label{eqn:Jinfty}
\end{equation}
Since $J_{s}^{2}=-Id$, we have $\|J\|_{L^{\infty}}\geq 1$. Moreover, for every $s \in \mathbb{R}$
$$
\frac{1}{\|J\|_{L^{\infty}}}\|\cdot\|_{L^{2}\times \mathbb{R}}^{2} \leq \|\cdot  \|^{2}_{g_{J_{s}}} \leq \|J\|_{L^{\infty}}\|\cdot\|_{L^{2}\times \mathbb{R}}^{2}.
$$
Analogously, for every $s \in \mathbb{R}$ we have
$$
\frac{1}{\|J\|_{L^{\infty}}}\|\nabla \mathcal{A}^{H_{s}}(u(s))\|_{L^{2}\times \mathbb{R}}\leq \|\nabla^{J_{s}} \mathcal{A}^{H_{s}}(u(s))\|_{L^{2}\times \mathbb{R}} \leq\|J\|_{L^{\infty}} \|\|\nabla \mathcal{A}^{H_{s}}(u(s))\|_{L^{2}\times \mathbb{R}}.
$$
Therefore
\begin{align}
\frac{1}{\|J\|_{L^{\infty}}^{3}}\|\nabla \mathcal{A}^{H_{s}} (u)\|_{L^{2}(\mathbb{R}\times S^{1})}^{2} & \leq\frac{1}{\|J\|_{L^{\infty}}}\|\nabla^{J_{s}} \mathcal{A}^{H_{s}} (u)\|_{L^{2}(\mathbb{R}\times S^{1})}^{2}\nonumber\\
& \leq \int_{-\infty}^{\infty}\|\nabla^{J_{s}} \mathcal{A}^{H_{s}}(u(s))\|^{2}_{g_{J_{s}}}ds\nonumber \\
&  \leq \|J\|_{L^{\infty}}\|\nabla^{J_{s}} \mathcal{A}^{H_{s}} (u)\|_{L^{2}(\mathbb{R}\times S^{1})}^{2}\nonumber \\
&  \leq \|J\|_{L^{\infty}}^{3}\|\nabla \mathcal{A}^{H_{s}} (u)\|_{L^{2}(\mathbb{R}\times S^{1})}^{2}.
\label{eqn:Jleq}
\end{align}
Since $u\in C^{\infty}(\mathbb{R}, C^{\infty}(S^{1}, \mathbb{R}^{2n}) \times \mathbb{R})$ is a Floer trajectory one can calculate the derivative of the action functional over $u$:
\begin{align*} 
\frac{d}{ds}\mathcal{A}^{H_{s}}(u(s)) & = \|\nabla^{J_{s}} \mathcal{A}^{H_{s}}(u(s))\|^{2}_{g_{J_{s}}}+(\partial_{s}\mathcal{A}^{H_{s}})(u(s))\\
& = \|\nabla^{J_{s}} \mathcal{A}^{H_{s}}(u(s))\|^{2}_{g_{J_{s}}}+\eta(s)\int_{0}^{1}\partial_{s}H_{s}(v(s,t))dt.
\end{align*}
Using the assumptions on the action of the endpoints one obtains
\begin{align*}
\mathcal{A}^{H_{s_{0}}}(u(s_{0})) & =\lim_{s\to -\infty}\mathcal{A}^{H_{0}}(u(s))+\int_{-\infty}^{s_{0}}\frac{d}{ds}\mathcal{A}^{H_{s}}(u(s))ds\\
& \geq  a + \int_{-\infty}^{s_{0}}\|\nabla^{J_{s}} \mathcal{A}^{H_{s}}(u(s))\|^{2}_{g_{J_{s}}}ds + \int _{0}^{s_{0}}\eta(s)\int_{0}^{1}\partial_{s}H_{s}(v(s,t))dtds\\
& \geq a  - \|\eta\|_{L^{\infty}}\|\partial_{s}H_{s}\|_{L^{\infty}}.
\end{align*}
Analogously, one obtains
\begin{align*}
\mathcal{A}^{H_{s_{0}}}(u(s_{0})) & \leq b - \int^{\infty}_{s_{0}}\|\nabla^{J_{s}} \mathcal{A}^{H_{s}}(u(s))\|^{2}_{g_{J_{s}}}ds - \int^{1}_{s_{0}}\eta(s)\int_{0}^{1}\partial_{s}H_{s}(v(s,t))dtds\\
& \leq b  + \|\eta\|_{L^{\infty}}\|\partial_{s}H_{s}\|_{L^{\infty}},
\end{align*}
which together leads to
\begin{equation}
\|\mathcal{A}^{H_{s}}( u)\|_{L^{\infty}} \leq \max\{|a|,|b|\} + \|\eta\|_{L^{\infty}}\|\partial_{s}H_{s}\|_{L^{\infty}}.
\label{eqn:act1}
\end{equation}
Likewise, we can bound
\begin{align*}
b-a & \geq \int_{-\infty}^{\infty}\frac{d}{ds}\mathcal{A}^{H_{s}}(u(s))ds  \\
& = \int_{-\infty}^{\infty}\|\nabla^{J_{s}} \mathcal{A}^{H_{s}}(u(s))\|^{2}_{g_{J_{s}}}ds+\int_{0}^{1}\eta(s)\int_{0}^{1}\partial_{s}H_{s}(v(s,t))dtds.
\end{align*}
Now combining the above with (\ref{eqn:Jleq}), we obtain the bounds on energy
\begin{align}
\|\partial_{s}u\|_{L^{2}(\mathbb{R}\times S^{1})}^{2} =\|\nabla^{J_{s}} \mathcal{A}^{H_{s}}(u)\|_{L^{2}(\mathbb{R}\times S^{1})}^{2} & \leq  \|J\|_{L^{\infty}} \int_{-\infty}^{\infty}\|\nabla^{J_{s}} \mathcal{A}^{H_{s}}(u(s))\|^{2}_{g_{J_{\tau}}}ds \nonumber \\
& \leq \|J\|_{L^{\infty}} (b-a + \|\eta\|_{L^{\infty}}\|\partial_{s}H_{s}\|_{L^{\infty}}). \label{eqn:energy1}
\end{align}
In particular by (\ref{eqn:Jleq}) the convergence of the integral
$$
\|\nabla \mathcal{A}^{H_{s}}(u)\|_{L^{2}(\mathbb{R}\times S^{1})} \leq \|J\|_{L^{\infty}}  \|\nabla^{J_{s}} \mathcal{A}^{H_{s}}(u)\|_{L^{2}(\mathbb{R}\times S^{1})},
$$
implies that, if we fix $\varepsilon_{0}>0$ as in Lemma \ref{lem:tilde}, then for small enough $s$
$$
\|\nabla \mathcal{A}^{H_{s}}(u(s))\|_{L^{2}\times\mathbb{R}}<\varepsilon_{0}.
$$

This ensures that for all $s \in \mathbb{R}$ the following value $\tau_{0}(s)$ is well defined and finite
$$
\tau_{0}(s): = \inf\{ \tau \leq s\ |\ \|\nabla \mathcal{A}^{H_{\tau}}(u(\tau))\|<\varepsilon_{0}\}.
$$
For $\tau \in [\tau_{0}(s),s]$ we have
$$
\|\nabla \mathcal{A}^{H_{\tau}}(u(\tau))\|_{L^{2}\times\mathbb{R}}\geq \varepsilon_{0}.
$$
therefore by (\ref{eqn:energy1}) and (\ref{eqn:Jleq})
\begin{align*}
|s-\tau_{0}(s)|\varepsilon_{0}^{2} & \leq \int_{\tau_{0}(s)}^{s} \|\nabla \mathcal{A}^{H_{\tau}}(u(\tau))\|_{L^{2}\times\mathbb{R}}^{2}d\tau\\
& \leq \|\nabla \mathcal{A}^{H_{s}}(u)\|^{2}_{L^{2}(\mathbb{R}\times S^{1})}\\
&  \leq \|J\|_{L^{\infty}}^{3} (b-a + \|\eta\|_{L^{\infty}}\|\partial_{s}H_{s}\|_{L^{\infty}}),\\
|s-\tau_{0}(s)| & \leq \frac{\|J\|_{L^{\infty}}^{3}}{\varepsilon_{0}^{2}}(b-a + \|\eta\|_{L^{\infty}}\|\partial_{s}H_{s}\|_{L^{\infty}}).
\end{align*}
Now using the above bounds we can calculate
\begin{align}
|\eta(s)-\eta(\tau_{0}(s))| & \leq \int_{\tau_{0}(s)}^{s}|\partial_{s}\eta(\tau)|d\tau \nonumber\\
& \leq \sqrt{|s-\tau_{0}(s)|}\sqrt{\int_{\tau_{0}(s)}^{s}|\partial_{s}\eta(\tau)|^{2}d\tau}\nonumber\\
& \leq \sqrt{|s-\tau_{0}(s)|}\sqrt{\int_{\tau_{0}(s)}^{s}\|\nabla^{J_{\tau}} \mathcal{A}^{H_{\tau}}(u(\tau))\|_{g_{J_{s}}}^{2}d\tau}\nonumber\\
& \leq \frac{\|J\|_{L^{\infty}}^{\frac{3}{2}}}{\varepsilon_{0}}(b-a +\|\eta\|_{L^{\infty}}\|\partial_{s}H_{s}\|_{L^{\infty}}). \label{eqn:eta1}
\end{align}

On the other hand, by definition of $\tau_{0}(s)$ the result of Lemma \ref{lem:tilde} applies giving
$$
|\eta(\tau_{0}(s))| \leq \tilde{c}(|\mathcal{A}^{H_{\tau_{0}(s)}}(u(\tau_{0}(s)))|+1).
$$
Now combining it with (\ref{eqn:act1}) and (\ref{eqn:eta1}) we can estimate that for any $s\in \mathbb{R}$ we have
\begin{align*}
|\eta (s)| & \leq  |\eta(\tau_{0}(s))|+|\eta(s)-\eta(\tau_{0}(s))|\\
& \leq \tilde{c}(|\mathcal{A}^{H_{\tau_{0}(s)}}(u(\tau_{0}(s)))|+1) + \frac{\|J\|_{L^{\infty}}^{\frac{3}{2}}}{\varepsilon_{0}}(b-a +\|\eta\|_{L^{\infty}}\|\partial_{s}H_{s}\|_{L^{\infty}})\\
& \leq \tilde{c}(\max\{|a|,|b|\} + \|\eta\|_{L^{\infty}}\|\partial_{s}H_{s}\|_{L^{\infty}}+1)
+\frac{\|J\|_{L^{\infty}}^{\frac{3}{2}}}{\varepsilon_{0}}(b-a +\|\eta\|_{L^{\infty}}\|\partial_{s}H_{s}\|_{L^{\infty}})\\
& = \tilde{c}(\max\{|a|,|b|\}+1)+\frac{(b-a)}{\varepsilon_{0}}\|J\|_{L^{\infty}}^{\frac{3}{2}}+\Big(\tilde{c}+\frac{\|J\|_{L^{\infty}}^{\frac{3}{2}}}{\varepsilon_{0}}\Big)\|\eta\|_{L^{\infty}}\|\partial_{s}H_{s}\|_{L^{\infty}}.
\end{align*}
Since the above inequality holds for every $s\in \mathbb{R}$, then in fact it holds for $\|\eta\|_{L^{\infty}}$, which together with (\ref{eqn:Hs2}), (\ref{eqn:act1}) and (\ref{eqn:energy1}) gives the claimed bounds
\begin{align}
\|\eta\|_{L^{\infty}(\mathbb{R})} & \leq \frac{8}{7} \Big(\tilde{c}(\max\{|a|, |b|\}+1)+\frac{b-a}{\varepsilon_{0}}\|J\|_{L^{\infty}}^{\frac{3}{2}}\Big) =: \mathfrak{y},\label{eqn:eta2}\\
\|\mathcal{A}^{H_{s}}(u)\|_{L^{\infty}(\mathbb{R})} & \leq \frac{1}{7}(8\max\{|a|, |b|\}+1+|b-a|)=:\mathfrak{a}, \label{eqn:act2} \\
\|\nabla^{J_{s}} \mathcal{A}^{H_{s}}(u(s))\|^{2}_{L^{2}(\mathbb{R}\times S^{1})} &  \leq \frac{1}{7}\|J\|_{L^{\infty}} \Big(8|b-a|+\max\{|a|,|b|\}+1\Big)=:\mathfrak{e}, \label{eqn:energy2}
\end{align}
where $\tilde{c}$ and $\varepsilon_{0}$ are the constants associated to the set $\mathsf{H}+\mathscr{O}(\mathsf{H})$ by Lemma \ref{lem:tilde}.
\end{proof}
Observe that the bound on the action does not depend on the homotopy, i.e. is uniform for all homotopies satisfying the assumptions of Proposition \ref{prop:bounds1}. Moreover, the bounds on the $\eta$ parameter and on the energy depend continuously on $\|J\|_{L^{\infty}}$, but not on the choice of $\{H_{s}\}_{s\in \mathbb{R}}$ as long as $H_{s}\in\mathsf{H}+\mathcal{O}(\mathsf{H})$ for all $s\in\mathbb{R}$.

Another consequence of the linearity condition established in Lemma \ref{lem:tilde} is that the homotopies that we consider in Theorem \ref{twr:ModuliCompact} satisfy the Novikov finiteness condition, which we prove in the lemma below. As a result, for every pair $a,b\in\mathbb{R}$ there exists a bounded subset of $\mathbb{R}^{2n+1}$ which contains the beginnings of all Floer trajectories in $\mathscr{M}^{\Gamma}(\Lambda_{0},\Lambda_{1})$ provided $\Lambda_{0}$ and $\Lambda_{1}$ satisfy (\ref{eqn:lambda}).
\begin{lem}
Let $\Gamma=\{H_{s},J_{s}\}_{s \in \mathbb{R}}$ be a smooth homotopy of Hamiltonians and almost complex structures defined as in Proposition \ref{prop:bounds1}. Then $\Gamma$ satisfies the Novikov finiteness condition. Moreover, for every pair $a,b\in \mathbb{R}$ and a compact subset $N\subseteq H^{-1}_{0}(0)$ there exists a compact subset of $\mathbb{R}^{2n+1}$, which contains all the connected components
$$
\Lambda_{0}\subseteq \mathscr{C}(\mathcal{A}^{H_{0}},N) \cap (\mathcal{A}^{H_{0}})^{-1}([a,\infty)),
$$
for which there exists a connected component
$$
\Lambda_{1}\subseteq \Crit(\mathcal{A}^{H_{1}}) \cap (\mathcal{A}^{H_{1}})^{-1}((-\infty,b]),
$$
such that
$$
\mathscr{M}^{\Gamma}(\Lambda_{0},\Lambda_{1}) \neq \emptyset .
$$
 \label{lem:ends}
\end{lem}

\begin{proof}
Since $\Gamma=\{H_{s},J_{s}\}_{s \in \mathbb{R}}$ is a homotopy defined as in Theorem \ref{twr:ModuliCompact}, we can apply Lemma \ref{lem:tilde}
to prove the linearity condition between the action and the $\eta$ as stated in (\ref{eqn:Hs2}). Having established (\ref{eqn:Hs2}), we can directly apply Corollary 3.8 from \cite{CieliebakFrauenfelder2009}, which gives us a relation between $\mathcal{A}^{H^{0}}(\Lambda^{0})$ and $\mathcal{A}^{H^{1}}(\Lambda^{1})$, namely if $\mathscr{M}^{\Gamma}(\Lambda_{0},\Lambda_{1})\neq \emptyset$, then
$$
a \leq \max\{2 b, 1\} \qquad \textrm{and} \qquad b \geq \min\{2 a, -1\}.
$$
In particular, $\Gamma$ satisfies the Novikov finiteness condition:
$$
A(\Gamma,a,b)  \geq \min\{2 a, -1\},\qquad \textrm{and} \qquad
B(\Gamma,a,b )  \leq \max\{2 b, 1\},
$$
where $A(\Gamma,a,b)$ and $B(\Gamma,a,b )$ are defined as in (\ref{eqn:Aab}) and (\ref{eqn:Bab}) respectively.

By assumption of Theorem \ref{twr:ModuliCompact} $H_{0}$ satisfies (\ref{item:PO}), hence all the connected components
$$
\Lambda_{0} \subseteq \Crit(\mathcal{A}^{H_{0}})\cap (\mathcal{A}^{H_{0}})^{-1}([a,\max\{2 b, 1\}]\setminus \{0\}) \\
$$
have the property that
$$
N:=\bigcup_{(v,\eta) \in \Lambda_{0}}v(S^{1})
$$
is a bounded subset of $H^{-1}_{0}(0)$. Define
\begin{equation}
\textsl{v}_{3}:=\sup\{\|v\|_{L^{\infty}}\ |\ (v,\eta)\in \mathscr{C}(\mathcal{A}^{H_{0}},N)\cap (\mathcal{A}^{H_{0}})^{-1}([a,\max\{2 b, 1\}]\setminus \{0\})\}.
\label{eqn:v3}
\end{equation}
Then combining with the boundedness of $\eta$ by $\mathfrak{y}$ established in Proposition \ref{prop:bounds1} and Property (\ref{item:H2}), we obtain the uniform boundedness in the $W^{1,2}\times\mathbb{R}$ norm namely for all $(v,\eta)$ in
$$
\mathscr{C}(\mathcal{A}^{H_{0}},N)\cap (\mathcal{A}^{H_{0}})^{-1}([a,\max\{2 b, 1\}]),
$$
one has
$$
|\eta|  \leq \mathfrak{y},\qquad
\|v\|_{L^{2}} \leq \textsl{v}_{3}\qquad \textrm{and}\qquad
\|\partial_{t}v\|_{L^{2}} \leq \mathfrak{y}(h_{1}+M \textsl{v}_{3}).
$$
\end{proof}

%% file: ActDeriv3.tex
\section{Set of infinitesimal action derivation}
\label{sec:ActDeriv}

In Proposition \ref{prop:bounds1} we have established that for Floer trajectories as in the setting of Theorem \ref{twr:ModuliCompact} there exist uniform bounds on $\eta$ parameter, energy and action. Now it will be convenient to analyze the Floer trajectories in the so called set of infinitesimal action derivation $\mathcal{B}^{\Gamma}(\mathfrak{a},\mathfrak{y},\varepsilon)$.
\begin{define}
Let $\Gamma=\{H_{s},J_{s}\}_{s \in \mathbb{R}}$ be a smooth homotopy of
Hamiltonians and almost complex structures, constant outside of $[0,1]$ and let $J_{s}\in \mathscr{J}(\mathbb{R}^{2n},\omega_{0})$  for all $s\in \mathbb{R}$.

Then for every triple of constants $0<\varepsilon,\mathfrak{a},\mathfrak{y}<\infty$ we define the \textbf{set of infinitesimal action derivation} of $\Gamma$ as
\begin{align*}
\mathcal{B}^{\Gamma}(\mathfrak{a},\mathfrak{y},\varepsilon) :=\Big\{ (v,\eta) & \in C^{\infty}(S^{1},\mathbb{R}^{2n})\times \mathbb{R} \ \Big|\  |\eta|\leq \mathfrak{y},\\ & \exists\ s \in [0,1]\ |\mathcal{A}^{H_{s}}(v,\eta)|\leq \mathfrak{a},\ \& \ \|\nabla^{J_{s}}\mathcal{A}^{H_{s}}(v,\eta)\|_{L^{2}\times \mathbb{R}}\leq \varepsilon \Big\}.
\end{align*}
\label{def:B}
\end{define}
Observe that for each pair $\mathfrak{a}, \mathfrak{y}>0$ a subset of critical points of $\mathcal{A}^{H_{s}}$ is contained in $\mathcal{B}^{\Gamma}(\mathfrak{a},\mathfrak{y},\varepsilon)$, namely for all $s\in[0,1]$
$$
 \{(v,\eta) \in \Crit(\mathcal{A}^{H_{s}}) \ \Big|\  |\eta|\leq \mathfrak{y}\ \& \ |\mathcal{A}^{H_{s}}(v,\eta)|\leq \mathfrak{a}\} \subseteq \mathcal{B}^{\Gamma}(\mathfrak{a},\mathfrak{y},\varepsilon),
$$
for all $\varepsilon>0$.

In particular, for all $\mathfrak{a}, \mathfrak{y}>0$ and every $s \in [0,1]$ the whole $0$-level-set of $H_{s}$ is contained in $\mathcal{B}^{\Gamma}(\mathfrak{a},\mathfrak{y},\varepsilon)$,
$$
\forall\ s \in [0,1]\qquad H^{-1}_{s}(0)\times\{0\}\subseteq \mathcal{B}^{\Gamma}(\mathfrak{a},\mathfrak{y},\varepsilon).
$$
Therefore, for any triple $\mathfrak{a}, \mathfrak{y},\varepsilon>0$, whenever the $H^{-1}_{s}(0)$ are non-compact, then  $\mathcal{B}^{\Gamma}(\mathfrak{a},\mathfrak{y},\varepsilon)$ is certainly not bounded.

An important step in finding the bounds on Floer trajectories is to localize $\mathcal{B}^{\Gamma}(\mathfrak{a},\mathfrak{y},\varepsilon)$. The time that a Floer trajectory spends outside of $\mathcal{B}^{\Gamma}(\mathfrak{a},\mathfrak{y},\varepsilon)$ is bounded by the energy, as discussed in Lemma \ref{lem:Kbound}. Unfortunately, the time that a Floer trajectory spends inside of $\mathcal{B}^{\Gamma}(\mathfrak{a},\mathfrak{y},\varepsilon)$ cannot be bounded this way, so in principle the trajectories could escape to infinity inside $\mathcal{B}^{\Gamma}(\mathfrak{a},\mathfrak{y},\varepsilon)$, since it is not bounded. %Therefore, investigating and localizing $\mathcal{B}^{\Gamma}(\mathfrak{a},\mathfrak{y},\varepsilon)$ is crucial to find bounds on Floer trajectories.

Now, recall that in the setting of Proposition \ref{prop:bounds1}, for $\mathfrak{y}$ and $\mathfrak{a}$ as in (\ref{eqn:eta2}) and (\ref{eqn:act2}) and for a pair of connected components $\Lambda_{0}$ and $\Lambda_{1}$ satisfying (\ref{eqn:lambda}),
the corresponding Floer trajectories in $\mathscr{M}^{\Gamma}(\Lambda_{0},\Lambda_{1})$ have their the action and $\eta$ parameter uniformly bounded, i.e.:
$$
\sup_{s\in\mathbb{R}}|\mathcal{A}^{H_{s}}(u(s))| \leq \mathfrak{a}, \qquad \sup_{s\in\mathbb{R}}|\eta(s)| \leq \mathfrak{y}.
$$
As a result, we have
$$
u \in \mathscr{M}^{\Gamma}(\Lambda_{0},\Lambda_{1}) \quad \& \quad \|\nabla^{J_{s}} \mathcal{A}^{H_{s}}(u(s))\|_{L^{2}\times \mathbb{R}}<\varepsilon \quad \Longrightarrow \quad u(s)\in \mathcal{B}^{\Gamma}(\mathfrak{a},\mathfrak{y},\varepsilon).
$$
This implies that every such Floer trajectory starts and ends in $\mathcal{B}^{\Gamma}(\mathfrak{a},\mathfrak{y},\varepsilon)$ for all $\varepsilon>0$. In other words, if $u \in  \mathscr{M}^{\Gamma}(\Lambda_{0},\Lambda_{1})$, then 
$$
\forall\ \varepsilon>0,\ \exists\ S_{\varepsilon}>0,\ \forall\ |s|>S_{\varepsilon} \qquad u(s) \in \mathcal{B}^{\Gamma}(\mathfrak{a},\mathfrak{y},\varepsilon).
$$

Denote by $\Sigma \subseteq C^{\infty}(S^{1}, \mathbb{R}^{2n})$ the set of constant loops on $\mathsf{H}^{-1}(0)$ and by $\mathcal{U}_{\delta}^{1}$ the $W^{1,2}\times \mathbb{R}$ neighborhood of $\Sigma\times\{0\}$ defined by
\begin{equation}
\mathcal{U}_{\delta}^{1} := \{(v,\eta)\in C^{\infty}(S^{1}, \mathbb{R}^{2n})\times \mathbb{R}\ |\ \dist_{W^{1,2}\times \mathbb{R}}((v,\eta),\Sigma\times\{0\})<\delta\}, \label{eqn:U1}
\end{equation}

In the following Proposition we will prove that for every $\mathfrak{a},\mathfrak{y},\delta>0$ there exists $\varepsilon>0$ and a set $K_{\delta}^{1}\subseteq C^{\infty}(S^{1}, \mathbb{R}^{2n})$, bounded in $W^{1,2}\times \mathbb{R}$ norm, such that
$$
\mathcal{B}^{\Gamma}(\mathfrak{a},\mathfrak{y},\varepsilon)\subseteq K_{\delta}^{1}\cup \mathcal{U}_{\delta}^{1},
$$
provided $\|J\|_{L^{\infty}}<+\infty$.
In other words if $(v,\eta)\in \mathcal{B}^{\Gamma}(\mathfrak{a},\mathfrak{y},\varepsilon)$, then either $v$ is bounded in $L^{\infty}$ or $\dist_{W^{1,2}}(v,\Sigma)<\delta$.
\begin{prop}
Consider a homotopy $\Gamma$ as in Proposition \ref{prop:bounds1}. Fix $\mathfrak{a},\mathfrak{y}<\infty$. Then for all $\delta>0$ and $r> \sup_{x\in K}\|x\|$, there exist $\varepsilon_{2}(\delta,\|J\|_{L^{\infty}})>0$ and $\textsl{v}_{2}(\delta,r)>0$ (depending also on $\mathfrak{a}$ and $\mathfrak{y}$), such that for all $\varepsilon\in (0,\varepsilon_{2}(\delta,\|J\|_{L^{\infty}}))$ and $(v,\eta)\in\mathcal{B}^{\Gamma}(\mathfrak{a},\mathfrak{y},\varepsilon)$ one of the following holds:
\begin{enumerate}[label*=\arabic*.]
\item If $\|v\|_{L^{2}}\geq \textsl{v}_{2}$ then 
\[
\|v(t)\|\geq r \quad \forall\ t\in S^{1}\quad \textrm{and}\quad 
\dist_{W^{1,2}\times \mathbb{R}}((v,\eta),\Sigma\times\{0\})<\delta
\]
where $\Sigma$ is the set of constant loops on $\mathsf{H}^{-1}(0)$.
\item If $\|v\|_{L^{2}}\leq \textsl{v}_{2}$, then in fact $(v,\eta)$ is uniformly bounded in $W^{1,2}\times \mathbb{R}$.
\end{enumerate}
\label{prop:partition}
\end{prop}
Before presenting the proof of Proposition \ref{prop:partition}, we will first prove a lemma, which describes the set $\mathcal{B}^{\Gamma}(\mathfrak{a},\mathfrak{y},\varepsilon)$ far away.

\begin{lem}
Let $\Gamma=\{H_{s},J_{s}\}_{s \in \mathbb{R}}$ be a smooth homotopy of Hamiltonians and almost complex structures as in Proposition \ref{prop:bounds1}.
 Fix $\mathfrak{a},\mathfrak{y} >0$. Then for every $\delta>0$, there exist $\varepsilon_{1}(\delta),\textsl{v}_{1}(\delta)>0$, depending only on $\mathfrak{a}$ and $\delta$, such that whenever
 $(v,\eta) \in \mathcal{B}^{\Gamma_{J_{0}}}(\mathfrak{a},\mathfrak{y},\varepsilon_{1}(\delta))$ and $\|v\|_{L^{2}}>\textsl{v}_{1}(\delta)$, the derivative $\partial_{t}v$ and $\eta$ are both bounded by $\delta$
$$
|\eta| \leq \delta,\qquad \|\partial_{t}v\|_{L^2}\leq \delta.
$$
\label{lem:nbhd_far}
\end{lem}
\begin{proof}
Let $X^{\dagger}$ be the Liouville vector field given by Property (\ref{item:H1}) and Lemma \ref{lem:open} for all\linebreak $H_{s} \in \mathsf{H}+\mathscr{O}(\mathsf{H})$.
If we now apply (\ref{eqn:dH(Y)}) to $X^{\dagger}$, we obtain
$$
|\mathcal{A}^{H_{s}}(v,\eta)|+\|\nabla \mathcal{A}^{H_{s}}(v,\eta)\|_{L^{2}\times \mathbb{R}}(c_{1}(\|v\|_{L^{2}}+1)+|\eta|)  \geq |\eta|(c_{2}\|v\|^{2}_{L^{2}}-c_{3}).
$$
Therefore for $(v,\eta) \in \mathcal{B}^{\Gamma_{J_{0}}}(\mathfrak{a},\mathfrak{y},\varepsilon)$ with $\|v\|_{L^{2}}$ large enough, that is
$$
\|v\|_{L^{2}}>\sqrt{\frac{\varepsilon+c_{3}}{c_{2}}},
$$
the following inequality holds 
$$
\frac{\mathfrak{a}+c_{1}\varepsilon(\|v\|_{L^{2}}+1)}{c_{2}\|v\|^{2}_{L^{2}}-c_{3}-\varepsilon} \geq |\eta|.
$$
By (\ref{eqn:nablaH}) for $(v,\eta) \in \mathcal{B}^{\Gamma_{J_{0}}}(\mathfrak{a},\mathfrak{y},\varepsilon)$ one also has
$$
\|\partial_{t}v\|_{L^{2}}  \leq \|\partial_{t}v-\eta X^{H_{s}}(v)\|_{L^{2}}+|\eta|\|\nabla H_{s}(v)\|_{L^{2}}
 \leq \varepsilon + \frac{\mathfrak{a}+\varepsilon(c_{1}\|v\|_{L^{2}}+1)}{c_{2}\|v\|_{L^{2}}^{2}-c_{3}-\varepsilon}(h_{1}+M \|v\|_{L^{2}}).
$$
Note that this bound does not depend on $s$ any more.
If we define 
$$
\varepsilon_{1}(\delta):= \frac{\delta}{2} \min\{\frac{c_{2}}{2 M c_{1}},1\},
$$
then we have
\begin{align*}
\|\partial_{t}v\|_{L^{2}} & \leq \varepsilon_{1}(\delta)+\frac{\mathfrak{a}+\varepsilon_{1}(\delta) c_{1}(\|v\|_{L^{2}}+1)}{c_{2}\|v\|_{L^{2}}^{2}-c_{3}-\varepsilon_{1}(\delta)}(h_{1}+M \|v\|_{L^{2}})\\
& \leq \frac{\delta}{2}+\frac{\mathfrak{a}+\frac{\delta c_{2}}{4 M}(\|v\|_{L^{2}}+1)}{c_{2}\|v\|_{L^{2}}^{2}-c_{3}-\delta\frac{c_{2}}{4 M c_{1}}}(h_{1}+M \|v\|_{L^{2}}).
\end{align*}
If we analyze the right-hand side, we see that as $\|v\|_{L^{2}}\to \infty$,
$$
\lim_{\|v\|_{L^{2}}\to\infty}\left(\frac{\delta}{2}+\frac{\mathfrak{a}+\frac{\delta c_{2}}{4 M}(\|v\|_{L^{2}}+1)}{c_{2}\|v\|_{L^{2}}^{2}-c_{3}-\delta\frac{c_{2}}{4 M c_{1}}}(h_{1}+M \|v\|_{L^{2}})\right)=\frac{3}{4}\delta.
$$
Moreover, 
$$
\frac{\mathfrak{a}+\frac{\delta c_{2}}{4 M }(\|v\|_{L^{2}}+1)}{c_{2}\|v\|_{L^{2}}^{2}-c_{3}-\delta\frac{c_{2}}{4 M c_{1}}}  \geq |\eta|,\qquad
\textrm{and}\qquad\lim_{\|v\|_{L^{2}}\to\infty}\frac{\mathfrak{a}+\frac{\delta c_{2}}{4 M}(\|v\|_{L^{2}}+1)}{c_{2}\|v\|_{L^{2}}^{2}-c_{3}-\delta\frac{c_{2}}{4 M c_{1}}}  = 0.
$$
Therefore there exists $\textsl{v}_{1}(\delta)>0$, such that if $\|v\|_{L^{2}}\geq \textsl{v}_{1}(\delta)$ then the following inequalities both hold:
\begin{align*}
|\eta| & \leq \frac{\mathfrak{a}+\frac{\delta c_{2}}{4 M }(\|v\|_{L^{2}}+1)}{c_{2}\|v\|_{L^{2}}^{2}-c_{3}-\delta\frac{c_{2}}{4 M c_{1}}} \leq \delta, \\
\|v\|_{L^{2}} & \leq \frac{\delta}{2}+\frac{\mathfrak{a}+\frac{\delta c_{2}}{4 M}(\|v\|_{L^{2}}+1)}{c_{2}\|v\|_{L^{2}}^{2}-c_{3}-\delta\frac{c_{2}}{4 M c_{1}}}(h_{1}+M \|v\|_{L^{2}}) \leq \delta.
\end{align*}
In other words, for every $\delta>0$ there exists $\textsl{v}_{1}(\delta)>0$ depending only on $\delta,\mathfrak{a}$ and the set $\mathsf{H}+\mathcal{O}(\mathsf{H})$, such that whenever $(v,\eta)\in\mathcal{B}^{\Gamma_{J_{0}}}(\mathfrak{a},\mathfrak{y},\varepsilon_{1}(\delta))$ and $\|v\|_{L^{2}}\geq \textsl{v}_{1}(\delta)$, then 
$$
|\eta| \leq \delta\qquad \textrm{and} \qquad \|\partial_{t}v\|_{L^2}\leq \delta.
$$
\end{proof}
Now we are ready to prove Proposition \ref{prop:partition}:

\vspace{0.25cm}
\textbf{Proof of Proposition \ref{prop:partition}:}
%\begin{proof}
In our proof we will use a result from Lemma \ref{lem:banan}, which states that a Hamiltonian $\mathsf{H}$ satisfying (\ref{item:H1}) and (\ref{item:H3}), there exists $\mu(\frac{\delta}{4},\mathsf{H})>0$, such that 
$$
H^{-1}(-\mu,\mu)\subseteq \left\lbrace x\in\mathbb{R}^{2n}\ \Big|\ \dist(x,H^{-1}(0))<\frac{\delta}{4}\right\rbrace.
$$
The proof of Lemma \ref{lem:banan} can be found in \ref{ssec:GeomHam}, since it relates more to the geometrical properties of the Hamiltonians.

Now take $\varepsilon_{0}>0$ as in Lemma \ref{lem:tilde}, $\textsl{v}_{1}(\frac{\delta}{4}), \varepsilon_{1}(\frac{\delta}{4})>0$ as in Lemma \ref{lem:nbhd_far} and $\mu(\frac{\delta}{4},\mathsf{H})>0$ as in Lemma \ref{lem:banan}. Then we claim that the statement of the Proposition holds for $\varepsilon_{2}(\delta,\|J\|_{L^{\infty}})$, $\textsl{v}_{2}(\delta,r)>0$ defined as below
\begin{align*}
\textsl{v}_{2}(\delta,r) & :=\max\Big\{\textsl{v}_{1}\left(\frac{\delta}{4}\right),r+\frac{1}{4}\delta\Big\},\\
\varepsilon_{2}(\delta, \|J\|_{L^{\infty}})& :=\frac{\varepsilon_{2}(\delta)}{\|J\|_{L^{\infty}}},\\
\varepsilon_{2}(\delta) & :=\min\left\lbrace\varepsilon_{0},\varepsilon_{1}\Big(\frac{\delta}{4}\Big), \mu\Big(\frac{\delta}{4},\mathsf{H}\Big) \right\rbrace.
\end{align*}
%The choice of $\textsl{v}_{2}(\delta)\leq \frac{c_{1}+\sqrt{c_{1}^{2}+4c_{2}(c_{3}+c_{1})}}{2c_{2}}+\delta$ and $\varepsilon_{2}(\delta)\leq \varepsilon_{0}$ is not used in this proof, but will become useful in the proof of Theorem \ref{twr:ModuliCompact}.

Note that for every $J\in \mathscr{J}(\mathbb{R}^{2n},\omega_{0})$ we have
$$
\|J\|_{L^{\infty}}^{-1}\|\nabla \mathcal{A}^{H}(v,\eta)\|_{L^{2}\times \mathbb{R}}\leq \|\nabla^{J} \mathcal{A}^{H}(v,\eta)\|_{g_{J}}%\leq \max\{1,\|J\|_{L^{\infty}}\}\|\nabla \mathcal{A}^{H}(v,\eta)\|_{L^{2}\times \mathbb{R}}.
$$
Now if we denote $\Gamma_{J_{0}}:=\{H_{s},J_{0}\}_{s\in\mathbb{R}}$, then the above inequality implies 
$$
\mathcal{B}^{\Gamma}\left(\mathfrak{a},\mathfrak{y},\varepsilon_{2}(\delta, \|J\|_{L^{\infty}})\right)\subseteq 
\mathcal{B}^{\Gamma_{J_{0}}}(\mathfrak{a},\mathfrak{y},\varepsilon_{2}(\delta)).
$$
Therefore, without loss of generality, one can assume the $J_{s}$ structures to be constant and equal everywhere to $J_{0}$.

\textbf{Proof of 1:}\\
Fix $\varepsilon<\varepsilon_{2}(\delta)$ and take $(v,\eta)\in\mathcal{B}^{\Gamma_{J_{0}}}(\mathfrak{a},\mathfrak{y},\varepsilon)$, such that $\|v\|_{L^{2}}\geq \textsl{v}_{2}(r,\delta)$. 

For every $t_{0}\in S^{1}$ one has
\begin{align*}
\|v-v(t_{0})\|_{L^{2}} & \leq = \left( \int_{0}^{1}\Big\|\int_{t_{0}}^{t}\partial_{t}v(\tau)d\tau\Big\|^{2}dt \right)^{\frac{1}{2}}\\
& \leq \|\partial_{t}v\|_{L^{2}}
\end{align*}
Since $\varepsilon<\varepsilon_{2}(\delta) \leq \varepsilon_{1}\left(\frac{\delta}{4}\right)$ and $\|v\|_{L^{2}} \geq \textsl{v}_{2}(r,\delta)\geq \textsl{v}_{1}(\frac{\delta}{4}) $
so by Lemma \ref{lem:nbhd_far}.
$$
 \frac{\delta}{4} \geq \|\partial_{t}v\|_{L^{2}} \geq \|v-v(t_{0})\|_{L^{2}}.
$$
Moreover taking into account that $\|v\|_{L^{2}} \geq \textsl{v}_{2}(r,\delta)\geq r+\frac{\delta}{4}$ we obtain the bound 
$$
\|v(t_{0})\| \geq \|v\|_{L^{2}}-\|v-v(t_{0})\|_{L^{2}} \geq r +\frac{\delta}{4} - \frac{\delta}{4}  = r.
$$
Since we chose $r>\sup_{x\in K}\|x\|$, it follows that
$$
(v,\eta)\in\mathcal{B}^{\Gamma_{J_{0}}}(\mathfrak{a},\mathfrak{y},\varepsilon)\quad \& \quad \|v\|_{L^{2}}\geq \textsl{v}_{2}(\delta) \quad \Longrightarrow \quad v(t) \notin K \quad \forall\ t\in S^{1}.
$$
This implies that for all $s \in [0,1],\ H_{s}(v(t))=\mathsf{H}(v(t))$. Together with Lemma \ref{lem:nbhd_close}, this gives us that there exists $t_{0}\in[0,1]$ such that
$$
v(t_{0})\in \mathsf{H}^{-1}(-\varepsilon_{2}(\delta),\varepsilon_{2}(\delta)).
$$
Since
$\varepsilon<\varepsilon_{2}(\delta)\leq \mu\left(\frac{\delta}{4},\mathsf{H}\right)$,
then by Lemma \ref{lem:banan} there exists $t_{0}\in[0,1]$ such that
$$
\dist(v(t_{0}),\mathsf{H}^{-1}(0))<\frac{\delta}{4}.
$$
Similarly, one can we use Lemma \ref{lem:nbhd_far}. to estimate the distance of $(v,\eta)\in \mathbb{B}^{\Gamma_{J_{0}}}(\mathfrak{a},\mathfrak{y},\varepsilon)$ from $\Sigma$ we have
\begin{align*}
\dist_{L^{2}}(v,\Sigma) & \leq \dist(v(t_{0}),\mathsf{H}^{-1}(0))+\|v(t_{0})-v\|_{L^{2}}\\
& \leq \dist(v(t_{0}),\mathsf{H}^{-1}(0))+\|\partial_{t}v\|_{L^{2}}\\
& \leq \frac{\delta}{4}+\frac{\delta}{4} =\frac{\delta}{2}.
\end{align*}
By Lemma \ref{lem:nbhd_far}. we have also $|\eta|\leq \frac{\delta}{4}$. Therefore, for $(v,\eta)\in\mathcal{B}^{\Gamma_{J_{0}}}(\mathfrak{a},\mathfrak{y},\varepsilon_{2}(\delta))$ and $\|v\|_{L^{2}}\geq \textsl{v}_{2}(r,\delta)$ one has
$$
\dist_{W^{1,2}\times \mathbb{R}}((v,\eta),\Sigma\times\{0\}) \leq \dist_{L^{2}}(v,\Sigma)+\|\partial_{t}v\|_{L^{2}}+|\eta| \leq \delta.
$$
%\pagebreak
\textbf{Proof of 2:}\\
Take $(v,\eta)\in\mathcal{B}^{\Gamma_{J_{0}}}(\mathfrak{a},\mathfrak{y},\varepsilon_{2}(\delta))$ and $\|v\|_{L^{2}}\leq \textsl{v}_{2}(r,\delta)$. By assumption $|\eta|\leq \mathfrak{y}$ so we have the following estimate:
\begin{align*}
\|\partial_{t}v\|_{L^{2}} & \leq \|\partial_{t}v-\eta X^{H_{s}}\|_{L^{2}}+|\eta|\|\nabla H_{s}(v)\|_{L^{2}}\\
& \leq \varepsilon_{2}(\delta) + \mathfrak{y}(h_{1}+M\|v\|_{L^{2}})\\
& \leq \varepsilon_{2}(\delta) + \mathfrak{y}(h_{1}+M\textsl{v}_{2}(r,\delta)),
\end{align*}
where the second inequality follows from (\ref{eqn:nablaH}) and the fact that $H_{s}\in \mathsf{H}+\mathcal{O}(\mathsf{H})$.

As a result $\|(v,\eta)\|_{W^{1,2}\times\mathbb{R}}$ is uniformly bounded as claimed.
%\end{proof}

\hfill $\square$
\vspace{0.25cm}

Observe that the bounds and constants in the proposition above do not depend on the choice of the homotopy $\Gamma$, but on the fact that $H_{s}\in \mathsf{H}+\mathscr{O}(\mathsf{H})$. Only $\varepsilon_{2}(\delta,\|J\|_{L^{\infty}})$ depends on $\|J\|_{L^{\infty}}$ explicitly and the bounds on $\|\partial_{s}v\|_{L^{2}}$ in proof of (2) depend on $\mathfrak{y}$, which through Proposition \ref{prop:bounds1} depends on $\|J\|_{L^{\infty}}$, but both of these quantities depend on $\|J\|_{L^{\infty}}$ continuously.

For every $\mathfrak{a},\mathfrak{y},\delta>0$ we define a subset of $C^{\infty}(S^{1}, \mathbb{R}^{2n})$ bounded in $W^{1,2}\times \mathbb{R}$ norm in the following way:
\begin{equation}
K_{\delta}^{1}:= \left\lbrace \begin{array}{c|c}
 & |\eta|\leq \mathfrak{y}\\
(v,\eta)\in C^{\infty}(S^{1},\mathbb{R}^{2n})\times \mathbb{R} & \|v\|_{L^{2}}\leq \max\{\textsl{v}_{2},\textsl{v}_{3}\}\\
& \|\partial_{t}v\|_{L^{2}}\leq \varepsilon_{2}+\mathfrak{y}(h_{1}+M\max\{\textsl{v}_{2},\textsl{v}_{3}\})
\end{array}\right\rbrace
\label{eqn:K1}
\end{equation}
where $\varepsilon_{2}(\delta,\|J\|_{L^{\infty}})>0$ and $\textsl{v}_{2}(\delta,\max_{K\cup\mathcal{V}}\|x\|)<+\infty$ are defined as in   
Proposition \ref{prop:partition} and $\textsl{v}_{3}$ is as in (\ref{eqn:v3}). Then by Proposition \ref{prop:partition}, for every $\varepsilon< \varepsilon_{2}(\delta,\|J\|_{L^{\infty}})$ one has the following splitting:
$$
\mathcal{B}^{\Gamma}(\mathfrak{a},\mathfrak{y},\varepsilon)  \subseteq K_{\delta}^{1}\cup \mathcal{U}_{\delta}^{1}.
$$
In particular, in the setting of Theorem \ref{twr:ModuliCompact}, every Floer trajectory from $\mathscr{M}^{\Gamma}(\Lambda_{0},\Lambda_{1})$ starts in $ K_{\delta}^{1}$, i.e. by Lemma \ref{lem:ends}
$$
\mathscr{C}(\mathcal{A}^{H_{0}},N)  \cap (\mathcal{A}^{H_{0}})^{-1}([a,\max\{2 b, 1\}]\setminus \{0\}) \subseteq K_{\delta}^{1}.
$$

%% file: CritMfld2.tex
\section{Floer trajectories near the critical set}
\label{sec:critMfld}
Let us put ourselves in the setting of Theorem \ref{twr:ModuliCompact}. Having localized $\mathcal{B}^{\Gamma}(\mathfrak{a},\mathfrak{y},\varepsilon)$ we would now like to establish global $L^{2}$ bounds on $\mathscr{M}^{\Gamma}(\Lambda_{0},\Lambda_{1})$. However, to establish the $L^{2}$ bounds on the $v$ component of a Floer trajectory one has to analyze the Floer trajectory not only outside of the set of infinitesimal action derivation, but one has to also estimate how far it travels within $\mathcal{B}^{\Gamma}(\mathfrak{a},\mathfrak{y},\varepsilon)$, along the hypersurface $\Sigma \times\{0\}$.
%Having localized the set $\mathcal{B}^{\Gamma}(\mathfrak{a},\mathfrak{y},\varepsilon)$ in the previous chapter, we will now investigate the behavior %of the Floer trajectories within $\mathcal{B}^{\Gamma}(\mathfrak{a},\mathfrak{y},\varepsilon)$. 
From Proposition \ref{prop:partition}, we know that for every $\mathfrak{a},\mathfrak{y},\delta>0$ there exists $\varepsilon>0$ and a set $K_{\delta}^{1} \subseteq C^{\infty}(S^{1},\mathbb{R}^{2n})\times \mathbb{R}$, bounded in $W^{1,2}\times\mathbb{R}$  norm, such that
$$
\mathcal{B}^{\Gamma}(\mathfrak{a},\mathfrak{y},\varepsilon)\subseteq K_{\delta}^{1}\cup \mathcal{U}_{\delta}^{1}.
$$
Since $\mathcal{U}_{\delta}^{1}$ (defined in (\ref{eqn:U1})) is non-compact, we will have to find a way to ensure that the Floer trajectories don't escape to infinity within this set. 

Let us now introduce a set defined analogically to $\mathcal{U}^{1}_{\delta}$, namely
\begin{equation}
\mathcal{U}_{\delta}^{0} := \{x\in C^{\infty}(S^{1}, \mathbb{R}^{2n})\times \mathbb{R}\ |\ \dist_{L^{2}\times \mathbb{R}}((v,\eta),\Sigma\times\{0\})<\delta\}.\label{eqn:U0}
\end{equation}
In Proposition \ref{prop:boundP} we will show that due to the Morse-Bott property of the action functional along $\Sigma\times\{0\}$, there exists a $\delta>0$ such that the Floer trajectories in $\mathcal{U}_{\delta}^{0}$ cannot escape along the unbounded component of the critical set.  In fact, as we will prove in this section, the tangential component of the Floer trajectories along the critical submanifold can be estimated by the energy growth. Naturally, for every $\delta>0$, the following inclusion holds
$$
\mathcal{B}^{\Gamma}(\mathfrak{a},\mathfrak{y},\varepsilon)\subseteq K_{\delta}^{1}\cup\mathcal{U}_{\delta}^{0},
$$
thus ensuring that the Floer trajectories don't escape within $\mathcal{B}^{\Gamma}(\mathfrak{a},\mathfrak{y},\varepsilon)$.

To make this assertion more precise and to prove it, we will view $\mathcal{U}_{\delta}^{0}$ as a tubular neighborhood of $\Sigma\times\{0\}$ in $C^{\infty}(S^{1},\mathbb{R}^{2n})\times \mathbb{R}$ with respect to the metric induced by $J_{0}$. Next, we will define a projection $P$ of $\mathcal{U}_{\delta}^{0}$ onto $\Sigma\times\{0\}$ and then analyze the image of the Floer trajectories under this projection. For this analysis we will need the Taylor expansion of $\nabla^{J_{0}}\mathcal{A}^{H}$ in this tubular neighborhood with respect to the points on $\Sigma\times\{0\}$, which is the subject of the following subsections.

However, to simplify our computations we would like the Hamiltonian to be constant and the almost complex structure to be equal to $J_{0}$. Therefore, we restrict ourselves to the analysis of the Floer trajectories outside the interval $[0,1]$.  Whenever $s\notin (0,1)$ then $(H_{s},J_{s})$ is either \linebreak $(H_{0},\{J_{0,t}(\cdot,\eta)\}_{(t,\eta)\in S^{1}\times\mathbb{R}})$ or $(H_{1},\{J_{1,t}(\cdot,\eta)\}_{(t,\eta)\in S^{1}\times\mathbb{R}})$. Moreover, recall that we chose our almost complex structures for all $s\in \mathbb{R}$
$$
J_{s}=\{J_{s,t}(\cdot,\eta)\}_{S^{1}\times\mathbb{R}}\in \mathcal{J}\Big(\mathbb{R}^{2n},\omega_{0},\mathcal{V}\times\big((-\infty,-\textrm{\etat})\cup( \textrm{\etat}, \infty)\big)\Big),
$$
to be constant and equal to $J_{0}$ outside the open set $\mathcal{V}\times((-\infty,-\textrm{\etat})\cup( \textrm{\etat}, \infty))$. Therefore for any $\delta \in (0,$\etat$)$ the almost complex structures $J_{s}$ are constant and equal $J_{0}$ in $\mathcal{U}_{\delta}^{0}$. In other words
$$
J_{s,t}(v(t),\eta)=J_{0},\quad \forall\ (v,\eta)\in \mathcal{U}_{\delta}^{0},\qquad\forall\ t\in S^{1},\qquad \forall\ s\in \mathbb{R}.
$$
Therefore, throughout this section we assume $s\notin (0,1)$ and take $(H,J)=(H_{0},J_{0})$ or $(H, J)=(H_{1},J_{0})$, which will simplify the setting and calculations significantly. The importance of choosing $J_{0}$ will become apparent in the proof of Lemma \ref{lem:Taylor}.

%%%%%%%%%%Tublar neighbourhood %%%%%%%%%%%%%%%%

\subsection{Tubular neighborhood and projection onto the critical set}
\label{ssec:TubN}
The aim of this subsection is to show how for small enough $\delta$ the set $\mathcal{U}_{\delta}^{0}$ can be described as the $\delta$-tubular neighborhood of $\Sigma\times{0}$ with respect to the metric $g_{J_{0}}$ induced by the canonical almost complex structure $J_{0}$. This means that for $\delta>0$ sufficiently small there exists a diffeomorphism $\Phi$ from a $\delta$-disc subbundle in the normal bundle of $\Sigma\times {0}$ onto $\mathcal{U}^0_\delta$. In the non-compact case, the existence of such a tubular neighborhood is not obvious and depends on the geometry of $\Sigma$. More precisely, due to the linearity of the exponential map corresponding to $g_{J_{0}}$ and the fact that the normal bundle of $\Sigma$ in $\mathbb{R}^{2n}$ can be naturally identified with a closed subspace of $N(\Sigma\times\{0\})$, a $\delta$-tubular neighborhood of $\Sigma\times{0}$ exists if and only if there exists a $\delta$-tubular neighborhood of $\Sigma$ in $\mathbb{R}^{2n}$. The existence of a $\delta$-tubular neighborhood of a $0$-level set of an admissible Hamiltonian is proven in Lemma \ref{lem:TubN}. Such a tubular neighborhood comes equipped with a projection $P: \mathcal{U}^0_\delta \rightarrow \Sigma\times{0}$, such that
$$
P \circ \Phi ((v,0), (\xi,\sigma)) = (v,0), \qquad \forall\ ((v,0),(\xi,\sigma))\in N^{\delta}(\Sigma \times \{0\}).
$$
Let us relate such defined projection $P$ to the parallel transport with respect to the metric $g_{J_{0}}$. Observe that, due to the absence of curvature for $g_{J_{0}}$ and the linear structure of $C^{\infty}(S^{1},\mathbb{R}^{2n})\times \mathbb{R}$, the parallel transport along geodesics can be viewed as an identity isomorphism under the natural identification of $C^{\infty}(S^{1},v_{1}^{*}T\mathbb{R}^{2n})\times \mathbb{R} $ with $ C^{\infty}(S^{1},v_{2}^{*}T\mathbb{R}^{2n})\times \mathbb{R}$,
\begin{align*}
Pt^{\gamma}: C^{\infty}(S^{1},v_{1}^{*}T\mathbb{R}^{2n})\times \mathbb{R} &\to  C^{\infty}(S^{1},v_{2}^{*}T\mathbb{R}^{2n})\times \mathbb{R},\\
((v_{1}(t),\xi(t)),\sigma) & \mapsto ((v_{2}(t),\xi(t)),\sigma).
\end{align*}
As a result, for every $(v,\eta)\in \mathcal{U}_{\delta}^{0}$  we have the following inclusion
\begin{equation}
N_{P(v,\eta)}(\Sigma \times \{0\})\subseteq \Ker (dP_{(v,\eta)}\circ Pt^{\gamma}_{P(v,\eta)}), \label{eqn:NsubKerdP}
\end{equation}
where $Pt^{\gamma}$ is the parallel transport between $P(v,\eta)$ and $(v,\eta)$. 

%%%%%%%%%%%%%%%%%Taylor%%%%%%%%%%%%%%%%%%%%%%

\subsection{Taylor expansion of the action functional}
In the next subsection we will analyze the behavior of Floer trajectories in the neighborhood $\mathcal{U}_{\delta}^{0}$ of the critical set $\Sigma\times\{0\}$. In particular, since $\mathcal{U}_{\delta}^{0}$ has a structure of a tubular neighborhood, we will analyze the image of the Floer trajectories under the projection $P$. In this analysis we will use the Taylor expansion of $\nabla^{J_{0}}\mathcal{A}^{H}(v,\eta)$ for $(v,\eta)\in \mathcal{U}_{\delta}^{0}$ with respect to $P(v,\eta)$. 

Following Lang in \cite{lang1999}, we define the Taylor expansion of $\nabla\mathcal{A}^{H}(v,\eta)$ for $(v,\eta)$ with respect to $P(v,\eta)$ and formulate its convergence as follows
\begin{align*}
%& Pt^{\gamma_{(v,\eta)}}\big( \sum_{k=0}^{m} \frac{1}{k!}(D_{{\gamma'}_{(v,\eta)}})^{k}\nabla\mathcal{A}^{H}(P(v,\eta)) \big)\\
\nabla \mathcal{A}^{H}(v,\eta) & = Pt^{\gamma}_{(P(v,\eta))}\big( \sum_{k=0}^{m} \frac{1}{k!}(D_{\gamma'})^{k}\nabla\mathcal{A}^{H}(P(v,\eta)) \big) + \mathcal{O}(\Phi^{-1}(v,\eta))\\
\textrm{where} &\qquad \lim_{\|\Phi^{-1}(\eta,\nabla)\|_{g_{J_{0}}}\to 0}\frac{\|\mathcal{O}(\Phi^{-1}(v,\eta))\|_{g_{J_{0}}}}{\|\Phi^{-1}(v,\eta)\|_{g_{J_{0}}}^{m+1}} <+\infty,
\end{align*}
where $D_{\gamma'}$ stands for the covariant derivative along $\gamma$.

In \cite{lang1999} Lang proves the local convergence of the Taylor series in the general setting of infinite dimensional Riemannian manifolds. However, for further purposes, we are only interested in the Taylor expansion of $\nabla \mathcal{A}^{H}(v,\eta)$ up to the first order, but we would like it to be uniformly convergent to the parallel transport of its Taylor expansion at $P(v,\eta)$ on the whole $\mathcal{U}_{\delta}^{0}$. In other words we would like to prove the following lemma:
\begin{lem}
\label{lem:Taylor}
Assume that $\delta$ is small enough for $\mathcal{U}_{\delta}^{0}$ to have a structure of a tubular neighbourhood and assume $\sup_{\mathbb{R}^{2n}}\|\Hess H\| <M$. Then for all $(v,\eta)\in \mathcal{U}_{\delta}^{0}$ the following holds
\begin{align*}
\nabla \mathcal{A}^{H}(v,\eta) & = Pt^{\gamma}\big( \nabla \mathcal{A}^{H}(P(v,\eta))+ \nabla^{2}_{P(v,\eta)}\mathcal{A}^{H}(\Phi^{-1}(v,\eta)) \big) + \mathcal{O}(\Phi^{-1}(v,\eta)),\\
\textrm{where} &\qquad \|\mathcal{O}(\Phi^{-1}(v,\eta))\|_{L^{2}\times\mathbb{R}}\leq \frac{1}{2}M\|\Phi^{-1}(v,\eta)\|_{L^{2}\times\mathbb{R}}^{2}.
\end{align*}
\end{lem}
Note that in particular, this holds whenever $H$ is admissible and $\delta$ as in Lemma \ref{lem:TubN}.

Before we proceed with the analysis of the Taylor expansion of $\nabla \mathcal{A}^{H}$, recall that the Hessian of $\mathcal{A}^{H}$ is a linear map
$$
\nabla^{2}_{(v,\eta)}\mathcal{A}^{H}:T_{v}C^{\infty}(S^{1},\mathbb{R}^{2n})\times \mathbb{R}\to T_{v}C^{\infty}(S^{1},\mathbb{R}^{2n})\times \mathbb{R}, 
$$
of the form
\begin{equation}
\nabla^{2}_{(v,\eta)}\mathcal{A}^{H}(\xi, \sigma) = 
\left( \begin{array}{c}
-J_{0}(\partial_{t}\xi-\sigma X^{H}(v))- \eta\ \Hess_{v} H(\xi)\\
 -\int dH(\xi) \end{array} \right). 
 \label{eqn:HessAH}
\end{equation}
\begin{proof}
Let us try to estimate the rest of Taylor expansion near the critical hypersurface $\Sigma\times\{0\}$. Fix $(v,\eta)\in \mathcal{U}_{\delta}^{0}$. Then there exists $(\xi,\sigma)\in N^{\delta}_{P(v,\eta)}(\Sigma\times\{0\})$ and $\bar{v}\in \Sigma$, such that
$$
\Phi^{-1}(v,\eta)=(P(v,\eta),(\xi,\sigma))=((\bar{v},0),(v-\bar{v},\eta)).
$$
Then we can estimate the rest of the Taylor expansion by
\begin{align*}
\mathcal{O}(\Phi^{-1}(v,\eta)) & =\nabla\mathcal{A}^{H}_{(v,\eta)}-Pt^{\gamma}(\nabla\mathcal{A}^{H}_{P(v,\eta)}+\nabla^{2}_{P(v,\eta)}\mathcal{A}^{H}(\xi,\eta))\\
& =\left( \begin{array}{c}
-J_{0}(\partial_{t}v-\eta X^{H}(v)) \\
-\int H(v) \end{array} \right)
-  
\left( \begin{array}{c}
-J_{0}(\partial_{t}\xi-\eta X^{H}(\bar{v}))\\
 -\int dH_{\bar{v}}(\xi) \end{array} \right)\\
& = \left( \begin{array}{c}
-J_{0}(\partial_{t}v-\eta X^{H}(v)-\partial_{t}(v-\bar{v})+\eta X^{H}(\bar{v}) ) \\
-\int (H(v)-dH_{\bar{v}}(\xi)) \end{array} \right)\\
& =\left( \begin{array}{c}
\eta J_{0}( X^{H}(v)- X^{H}(\bar{v}) ) \\
-\int (H(v)-dH_{\bar{v}}(\xi)) \end{array} \right).
\end{align*}
Since the Hessian of $H$ is uniformly bounded, we obtain
\begin{align*}
\|\eta J_{0}( X^{H}(v)- X^{H}(\bar{v}) )\|_{L^{2}} & \leq |\eta|\|\nabla H(v)-\nabla H(\bar{v})\|_{L^{2}}\\
& \leq |\eta|  \|v-\bar{v}\|_{L^{2}} \Big(\int \sup_{s\in[0,1]} \|\Hess_{\bar{v}+s(v(t)-\bar{v})} H\|^{2}dt\Big)^{\frac{1}{2}}\\
& \leq M |\eta|  \|\xi\|_{L^{2}}.
\end{align*}

Analogically using once more the Taylor expansion for $H$ at $\bar{v}$ we obtain
$$
\Big|\int (H(v(t))-dH_{\bar{v}}(v(t)-\bar{v}))dt\Big| \leq  \int \frac{1}{2}\sup_{s\in[0,1]} \|\Hess_{\bar{v}+s(v(t)-\bar{v})} H\| \|v(t)-\bar{v}\|^{2}dt \leq \frac{1}{2} M \|\xi\|_{L^{2}}^{2}.
$$
The two results combined give 
\begin{align*}
\|\mathcal{O}(\Phi^{-1}(v,\eta))\|_{L^{2}\times\mathbb{R}} & =\|\nabla\mathcal{A}^{H}_{(v,\eta)}-Pt^{\gamma}(\nabla\mathcal{A}^{H}_{P(v,\eta)}+\nabla^{2}_{P(v,\eta)}\mathcal{A}^{H}(\xi,\eta))\|_{L^{2}\times \mathbb{R}}\\
& \leq \frac{1}{2} M (\|\xi\|_{L^{2}}+|\eta|)^{2}.
\end{align*}
\end{proof}
Observe, that this way we have obtained a bound on $\mathcal{O}(\Phi^{-1}(v,\eta))$ by the $L^{2}\times \mathbb{R}$ distance from $(v,\eta)$ to $\Sigma\times \{0\}$. This is due to the fact that we chose the metric $J_{0}$ to construct the tubular neighborhood of $\Sigma\times \{0\}$ and to define the parallel transport and as a consequence the first derivatives of $v$ and $\xi$ canceled each other, leaving the estimate dependent on the $L^{2}\times \mathbb{R}$ rather than $W^{1,2}\times\mathbb{R}$ distance from $\Sigma\times \{0\}$.

%%%%%%%%%%%%%Hessian%%%%%%%%%%%%%%%%%%%

\subsection{Properties of the Hessian of the action functional}
In this subsection we will prove some properties of the Hessian of the action functional both on $\Sigma\times\{0\}$ and in its tubular neighborhood. We will use these properties later in the analysis of the projection of the Floer trajectory and in the proof of Proposition \ref{prop:boundP}. The key property of the Rabinowitz action functional that will enable us to establish bounds on the Floer trajectory in the constructed tubular neighborhood is the Morse-Bott property of $\mathcal{A}^{H}$ at the hypersurface $\Sigma\times\{0\}$. This property has been proven by Cieliebak and Frauenfelder in \cite{CieliebakFrauenfelder2009} in Step 4 of the proof of Theorem B.1 and it states that for all $\bar{v}\in \Sigma,\ \Ker(\nabla^{2}_{(\bar{v},0)}\mathcal{A}^{H})=T_{\bar{v}}\Sigma \times\{0\}$. Even though Theorem B.1 is  stated for a setting different than ours, the proof shown in Step 4 can be applied to any Hamiltonian having $0$ as a regular value. 

\begin{lem}
Let $\mathcal{U}_{\delta}^{0}$ be a tubular neighborhood of $\Sigma\times\{0\}$ and let $P:\mathcal{U}_{\delta}^{0} \to \Sigma\times\{0\}$ be the associated projection. 
Then for all $(v,\eta)\in\mathcal{U}_{\delta}^{0}$
$$
\nabla^{2}_{P(v,\eta)}\mathcal{A}^{H}(\Phi^{-1}(v,\eta)) \in N_{P(v,\eta)}(\Sigma\times\{0\}).
$$
\label{lem:HessAinN}
\end{lem}
\begin{proof}
Recall that we identify the normal bundle of $\Sigma \times \{0\}$, using the metric $g_{J_{0}}$, with
\begin{equation}
N_{(v,0)}(\Sigma\times\{0\})=\Big\{ \xi \in C^{\infty}(S^{1},v^{*}T\mathbb{R}^{2n}) \ \Big|\ \langle\int_{0}^{1}\xi(t)dt,\hat{\xi}\ \rangle=0,\ \forall\  \hat{\xi}\in T_{v}H^{-1}(0) \Big\}\times \mathbb{R}.
\label{eqn:NS0}
\end{equation}

Take $(v,\eta)\in\mathcal{U}_{\delta}^{0}$. Then by definition there exists $\bar{v}\in\Sigma$ and $(\xi,\sigma)\in N^{\delta}_{(\bar{v},0)}(\Sigma\times\{0\})$ such that
$$
P(v,\eta) =(\bar{v},0),\quad \textrm{and} \quad \Phi((\bar{v},0),(\xi,\sigma)) =(v,\eta).
$$
In particular $\sigma = \eta$ and $\xi=v-\bar{v}$.

Considering $\nabla^{2} \mathcal{A}^{H}$ as in (\ref{eqn:HessAH}) and the above characterization of $N(\Sigma\times\{0\})$, to prove the lemma it is enough to show that for all $\hat{\xi}\in T_{\bar{v}}H^{-1}(0)$
$$
\langle \int J_{0}(\partial_{t}\xi-\eta X^{H}(\bar{v}))dt,\hat{\xi} \rangle =0.
$$
But this is an easy consequence of the fact that
$$
\int \partial_{t}\xi dt = 0 \qquad \textrm{and} \qquad -J_{0}X^{H}(\bar{v})=\nabla H (\bar{v}).
$$
\end{proof}

By Step 4 of the proof of Theorem B.1 in \cite{CieliebakFrauenfelder2009}, we have that $\Ker(\nabla^{2}_{(v,0)}\mathcal{A}^{H})=T_{v}\Sigma \times\{0\}$ for $v\in \Sigma$. As a consequence the restriction of $\nabla^{2}\mathcal{A}^{H}$ to the normal bundle of $\Sigma\times\{0\}$ is injective. In the following Lemma we will prove that this restriction has an even stronger property, namely it admits a uniform lower bound, whenever $\nabla H$ is large enough. In particular, for $H$ satisfying (\ref{item:H1}), due to the linear growth of the gradient of $H$ as shown in (\ref{eqn:nablaH2}), condition $\|\nabla H(v)\|\geq \frac{1}{2}$ is satisfied far enough on $\Sigma$.
\begin{lem}
If $v\in \Sigma$ and $\|\nabla H(v)\|\geq \frac{1}{2}$, then for $(\xi,\eta)\in N_{(v,0)}(\Sigma\times\{0\})$
$$
\|\nabla^{2}_{(v,0)} \mathcal{A}^{H}(\xi,\eta)\|_{L^{2}\times\mathbb{R}}\geq \frac{1}{6}( \|\xi\|_{W^{1,2}}+|\eta|).
$$
\label{lem:P}
\end{lem}
\begin{proof}
By (\ref{eqn:NS0}) for $(\xi,\eta)\in N_{(v,0)}(\Sigma\times\{0\})$ one has
\begin{align*}
\int_{0}^{1}\xi(t)dt & = \frac{\langle \int\xi(t)dt,\nabla H(v) \rangle}{\|\nabla H(v)\|^{2}}\nabla H(v)
=  \frac{ \int dH_{v}(\xi)dt}{\|\nabla H(v)\|^{2}}\nabla H(v),\\
\Big|\int dH_{v}(\xi)dt\Big| & = \Big\|\int_{0}^{1}\xi(t)dt\Big\|\ \|\nabla H(v)\|.
\end{align*}
Taking the above equality into account and recalling that for $\xi \in W^{1,2}(S^{1},\mathbb{R}^{2n})$
$$
\|\partial_{t}\xi\|_{L^{2}}+\Big\| \int \xi(t)dt \Big\|\geq \|\xi\|_{L^{2}},
$$
we can bound the Hessian of $\mathcal{A}^{H}$ from below in the following way
\begin{align*}
 \|\nabla^{2}_{(v,0)} & \mathcal{A}^{H}(\xi,\eta)\|_{L^{2}\times\mathbb{R}} =\|\partial_{t}\xi-\eta X^{H}(v)\|_{L^{2}}+\Big|\int dH_{v}(\xi)\Big|\\
 & = \|\partial_{t}\xi-\eta X^{H}(v)\|_{L^{2}}+\Big\| \int \xi(t)dt\Big\|\ \|\nabla H(v)\|\\
 & \geq  \frac{1}{3}(\|\partial_{t}\xi\|_{L^{2}}-|\eta| \| X^{H}(v)\|_{L^{2}})+\frac{2}{3}\Big\|\int (\partial_{t}\xi-\eta X^{H}(v))dt\Big\|+\Big\|\int \xi(t)dt\Big\| \|\nabla H(v)\|\\
 & =  \frac{1}{3}\|\partial_{t}\xi\|_{L^{2}}+\frac{1}{3}|\eta| \|X_{H}(v)\|_{L^{2}}+\Big\|\int \xi(t)dt\Big\|\ \|\nabla_{v}H\|\\
 & \geq \frac{1}{6}\|\partial_{t}\xi\|_{L^{2}}+\frac{1}{6}\|\xi\|_{L^{2}}+\frac{1}{3}|\eta| \| \nabla H(v)\|_{L^{2}}+\Big\|\int \xi(t)dt\Big\|\ \Big(\|\nabla H(v)\|-\frac{1}{6}\Big)\\
 & \geq \frac{1}{6}(\|\xi\|_{W^{1,2}}+|\eta|)
\end{align*}
The last inequality comes from the assumption that $\|\nabla H(v)\|\geq \frac{1}{2}$.
\end{proof}

%%%%%%%%%%%%Projection%%%%%%%%%%%%%%%%

\subsection{Projecting a Floer trajectory}
Let $u:\mathbb{R}\to C^{\infty}(S^{1},\mathbb{R}^{2n})\times \mathbb{R}$ be a Floer trajectory, i.e. $u(s)=(v(s,t),\eta(s))$ and
$$
\partial_{s}u=\left( \begin{array}{c}
\partial_{s}v\\
\partial_{s}\eta
\end{array} \right)
=\left( \begin{array}{c}
-J(\partial_{t}v-\eta X^{H}(v)) \\
-\int H(v) \end{array} \right).
$$
Let us analyze $P(u(s))$. Can we bound the distance $P(u(s))$ travels along the hypersurface $\Sigma$ for $u(s)\in\mathcal{U}_{\delta}^{0}$? 

\begin{figure}[h]
 \centering  
  \def\svgwidth{9cm}  
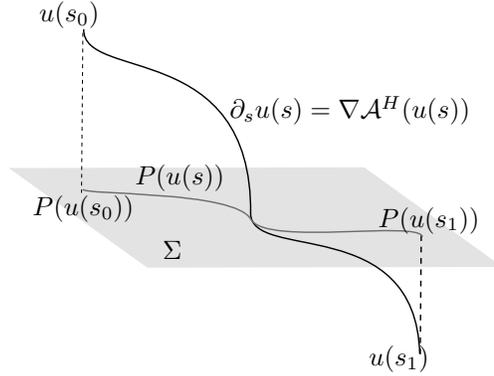
\caption{Schematic depiction of the Floer trajectory near the critical set $\Sigma\times\{0\}$.}
\end{figure}
To simplify the notation we will introduce the set $\mathscr{N}^{\delta}_{r}(\Sigma)$. Recall the tubular neighborhood $\mathcal{U}_{\delta}^{0}$ defined in subsection \ref{ssec:TubN}.
Then for small enough $\delta>0$ and any $r>0$ we will denote
\begin{equation}
\mathscr{N}^{\delta}_{r}(\Sigma):=\{ x \in \mathcal{U}_{\delta}^{0}\ |\ \|P(x)\|\geq r \}
\label{eqn:Nr}
\end{equation}
Observe that even though the definition of $\mathcal{U}^{0}_{\delta}$ depends on whether we take $\Sigma$ to be $H^{-1}_{0}(0)$ or $H^{-1}_{1}(0)$, however for $r>\sup_{x\in K}\|x\|$ the definition of $\mathscr{N}^{\delta}_{r}(\Sigma)$ does not depend on the choice of $\Sigma$, i.e. :
$$
\mathscr{N}^{\delta}_{r}(H^{-1}_{0}(0))= \mathscr{N}^{\delta}_{r}(H^{-1}_{1}(0)),
$$
since $H_{0}$ and $H_{1}$ differ only on the compact set $K$. Moreover, we have
$$
\{(v,\eta)\in\mathcal{U}^{0}_{\delta}\ |\ \|v\|_{L^{2}}\geq r+\delta\}\subseteq \mathscr{N}^{\delta}_{r}(\Sigma).
$$
Now we are ready to prove the boundedness of the projection $P(u(s))$ for a Floer trajectory $u(s)$ within $\mathscr{N}^{\delta}_{r}(\Sigma)$.

\begin{prop}
Assume $H$ is admissible, then there exist $\delta>0,\ r>0$ and $\hat{M}$, such that if $u$ is a Floer trajectory 
\begin{align*}
u & :\mathbb{R}\to C^{\infty}(S^{1},\mathbb{R}^{2n})\times \mathbb{R}\\
u(s) & =(v(s),\eta(s)),
\end{align*}
and $s_{0},s_{1}\in\mathbb{R}$ are such that
$$
\forall \ s\in[s_{0},s_{1}]  \qquad  u(s)\in \mathscr{N}^{\delta}_{r}(\Sigma),
$$
the following holds
$$
\|P(u(s_{1}))-P(u(s_{0}))\|\leq\hat{M}|\mathcal{A}^{H}(u(s_{1}))-\mathcal{A}^{H}(u(s_{0}))|.
$$
%where $\hat{M} =18 M$ and $M=\sup_{x\in\mathbb{R}^{2n}}\|\Hess_{x}H\|$ is finite by (\ref{item:H2}).
\label{prop:boundP}
\end{prop}
\begin{proof}
$H$ is admissible, hence by Lemma \ref{lem:TubN} there exists $\tilde{\delta}(H)>0$, such that for any $0<\delta\leq \tilde{\delta}(H)$ 
$\Sigma\times\{0\}$ admits a tubular neighborhood $\mathcal{U}^{0}_{\delta}$.

Now, knowing that $\partial_{s}u=\nabla \mathcal{A}^{H}(u)$ and using the Taylor expansion for $\nabla \mathcal{A}^{H}$ from Lemma \ref{lem:Taylor} 
we obtain

\begin{align*}
 P( & u (s_{1})) - P(u(s_{0})) = \int_{s_{0}}^{s_{1}} \frac{d}{ds}P(u(s))ds \\
& =  \int_{s_{0}}^{s_{1}} dP_{u(s)}\partial_{s}u(s)ds\\
& = \int_{s_{0}}^{s_{1}} dP_{u(s)}\nabla\mathcal{A}^{H}(u(s))ds\\
& = \int_{s_{0}}^{s_{1}} dP_{u(s)} \Big( Pt^{\gamma}\big( \nabla \mathcal{A}^{H}(P(u(s)))+ \nabla^{2}_{P(u(s))}\mathcal{A}^{H}(\Phi^{-1}(u(s))) \big) + \mathcal{O}(\Phi^{-1}(u(s))\Big)ds\\
& = \int_{s_{0}}^{s_{1}} dP_{u(s)} \Big( Pt^{\gamma}\big( \nabla^{2}_{P(u(s))}\mathcal{A}^{H}(\Phi^{-1}(u(s))) \big) + \mathcal{O}(\Phi^{-1}(u(s))\Big)ds.
\end{align*}

Combining results from Lemma \ref{lem:HessAinN} and (\ref{eqn:NsubKerdP}) we obtain that for all $(v,\eta)$ in $\mathcal{U}_{\delta}^{0}$
$$
\nabla^{2}_{P(v,\eta)}\mathcal{A}^{H}(\Phi^{-1}(v,\eta)) \in N_{P(v,\eta)}(\Sigma\times\{0\})\subseteq \Ker (dP_{(v,\eta)}\circ Pt^{\gamma}_{P(v,\eta)}),
$$
hence the first expression under the integral in fact vanishes.

Now we will use the estimate on $\mathcal{O}$ from Lemma \ref{lem:Taylor} to get
\begin{align*}
\|P(u(s_{1}))-P(u(s_{0}))\| & \leq \int_{s_{0}}^{s_{1}} \|dP_{u(s)}\mathcal{O}(\Phi^{-1}(u(s)))\|_{L^{2}\times\mathbb{R}}^{2}ds\\
& \leq \int_{s_{0}}^{s_{1}}\| \mathcal{O}(\Phi^{-1}(u(s)))\|_{L^{2}\times\mathbb{R}}^{2}ds\\
& \leq \frac{1}{2} M\int_{s_{0}}^{s_{1}}\|\Phi^{-1}(u(s))\|_{L^{2}\times\mathbb{R}}^{2}ds.
\end{align*}
Let us now make a change of variables. Note that
\begin{align*}
 \Psi & :[s_{0},s_{1}]\to [\mathcal{A}^{H}(u(s_{0})),\mathcal{A}^{H}(u(s_{1}))],\\
 \Psi(s) & := \mathcal{A}^{H}(u(s)),\\
 \Psi'(s)&=\|\nabla \mathcal{A}^{H} u(s)\|_{L^{2}\times\mathbb{R}}^{2},
\end{align*}
is in fact a diffeomorphism. Substituting into the equation we get
$$
 \|P(u(s_{1}))-P(u(s_{0}))\|  \leq  \frac{1}{2} M \int_{\mathcal{A}^{H}(u(s_{0}))}^{\mathcal{A}^{H}(u(s_{1}))}\frac{\|\Phi^{-1}\circ u \circ \Psi^{-1}(\tau)\|_{L^{2}\times\mathbb{R}}^{2}}{\|\nabla \mathcal{A}^{H}(u\circ \Psi^{-1}(\tau))\|_{L^{2}\times\mathbb{R}}^{2}}d\tau.
$$

Let us now estimate the denominator under the integral. By (\ref{eqn:nablaH2}) there exists $r>0$, depending only on the constants $c_{1},c_{2},c_{3}$, such that for every $x\in \mathscr{N}^{\delta}_{r}(\Sigma) \cap (\Sigma\times\{0\})$, $\|\nabla H(x)\| \geq \frac{1}{2}$. In other words, whenever  $u(s)\in \mathscr{N}^{\delta}_{r}(\Sigma)$ then for $P(u(s))$ we can use the result from Lemma \ref{lem:P} 
$$
\|\nabla^{2}_{P(u)}\mathcal{A}^{H}(\Phi^{-1}(u))\|_{L^{2}\times\mathbb{R}} \geq \frac{1}{6}\|\Phi^{-1}(u)\|_{L^{2}\times\mathbb{R}},
$$
and the estimate from Lemma \ref{lem:Taylor} to calculate
\begin{align*}
\|\nabla \mathcal{A}^{H}(u)\|_{L^{2}\times\mathbb{R}} & \geq \|\nabla^{2}_{P(u)}\mathcal{A}^{H}(\Phi^{-1}(u))\|_{L^{2}\times\mathbb{R}}-\frac{1}{2}M\|\Phi^{-1}(u)\|_{L^{2}\times\mathbb{R}}^{2}\\
& \geq \frac{1}{6}\|\Phi^{-1}(u)\|_{L^{2}\times\mathbb{R}} - \frac{1}{2}M\|\Phi^{-1}(u)\|_{L^{2}\times\mathbb{R}}^{2}\\
& = \frac{1}{2}\|\Phi^{-1}(u)\|_{L^{2}\times\mathbb{R}} (\frac{1}{3} - M \|\Phi^{-1}(u)\|_{L^{2}\times\mathbb{R}}),
\end{align*}
where $M>\sup_{\mathbb{R}^{2n}} \|\Hess H\|$.
Therefore, for 
$$
\|\Phi^{-1}(u)\|_{L^{2}\times\mathbb{R}}\leq \frac{1}{6 M},
$$
one has
$$
 \|\nabla\mathcal{A}^{H}(u)\|_{L^{2}\times\mathbb{R}} \geq \frac{1}{12}\|\Phi^{-1}(u)\|_{L^{2}\times\mathbb{R}}.
$$
Therefore, if we take $\delta>0$, such that
$$
\delta < \min \{\tfrac{1}{6M}, \tilde{\delta}(H)\},
$$
where $\tilde{\delta}(H)$ is as in Lemma \ref{lem:TubN}, then we can combine the two results above to analyze the expression under the integral
$$
\frac{\|\Phi^{-1}(u)\|^{2}_{L^{2}\times\mathbb{R}}}{\|\nabla\mathcal{A}^{H}(u)\|^{2}_{L^{2}\times\mathbb{R}}} \leq \frac{12^{2}\|\Phi^{-1}(u)\|^{2}_{L^{2}\times\mathbb{R}}}{\|\Phi^{-1}(u)\|^{2}_{L^{2}\times\mathbb{R}}} = 12^{2}.\\
$$
Therefore, taking $\hat{M}=72 M$, we obtain
\begin{align*}
\|P(u(s_{1}))-P(u(s_{0}))\|& \leq \frac{1}{2} M \int_{\mathcal{A}^{H}(u(s_{0}))}^{\mathcal{A}^{H}(u(s_{1}))}\frac{\|\Phi^{-1}\circ u \circ \Psi^{-1}(\tau)\|_{L^{2}\times\mathbb{R}}^{2}}{\|\nabla \mathcal{A}^{H}(u\circ \Psi^{-1}(\tau))\|_{L^{2}\times\mathbb{R}}^{2}}d\tau \\
& \leq \hat{M}|\mathcal{A}^{H}(u(s_{1}))-\mathcal{A}^{H}(u(s_{0}))|,
\end{align*}
as claimed.
\end{proof}

%% file: figure4A.pdf_tex
%% Creator: Inkscape inkscape 0.91, www.inkscape.org
%% Accompanies image file 'figure4A.pdf' (pdf, eps, ps)
\begingroup%
  \makeatletter%
  \providecommand\color[2][]{%
    \errmessage{(Inkscape) Color is used for the text in Inkscape, but the package 'color.sty' is not loaded}%
    \renewcommand\color[2][]{}%
  }%
  \providecommand\transparent[1]{%
    \errmessage{(Inkscape) Transparency is used (non-zero) for the text in Inkscape, but the package 'transparent.sty' is not loaded}%
    \renewcommand\transparent[1]{}%
  }%
  \providecommand\rotatebox[2]{#2}%
  \ifx\svgwidth\undefined%
    \setlength{\unitlength}{594.30938691bp}%
    \ifx\svgscale\undefined%
      \relax%
    \else%
      \setlength{\unitlength}{\unitlength * \real{\svgscale}}%
    \fi%
  \else%
    \setlength{\unitlength}{\svgwidth}%
  \fi%
  \global\let\svgwidth\undefined%
  \global\let\svgscale\undefined%
  \makeatother%
  \begin{picture}(1,0.69205941)%
    \put(0,0){\includegraphics[width=\unitlength,page=1]{figure4A.pdf}}%
    \put(0.05824283,0.7022675){\color[rgb]{0,0,0}\makebox(0,0)[lt]{\begin{minipage}{0.39229203\unitlength}\raggedright $u(s_{0})$\end{minipage}}}%
    \put(0.68419969,0.01015885){\color[rgb]{0,0,0}\makebox(0,0)[lb]{\smash{$u(s_{1})$}}}%
    \put(0.04563095,0.33427323){\color[rgb]{0,0,0}\makebox(0,0)[lt]{\begin{minipage}{0.39229203\unitlength}\raggedright $P(u(s_{0}))$\end{minipage}}}%
    \put(0.70198716,0.31189226){\color[rgb]{0,0,0}\makebox(0,0)[lt]{\begin{minipage}{0.30778727\unitlength}\raggedright $P(u(s_{1}))$\end{minipage}}}%
    \put(0.24001517,0.39165281){\color[rgb]{0,0,0}\makebox(0,0)[lt]{\begin{minipage}{0.39229203\unitlength}\raggedright $P(u(s))$\end{minipage}}}%
    \put(0.41933772,0.52268952){\color[rgb]{0,0,0}\makebox(0,0)[lt]{\begin{minipage}{0.55265344\unitlength}\raggedright $\partial_{s}u(s)=\nabla\mathcal{A}^{H}(u(s))$\end{minipage}}}%
    \put(0.29216008,0.24680254){\color[rgb]{0,0,0}\makebox(0,0)[lt]{\begin{minipage}{0.59804708\unitlength}\raggedright $\Sigma$\end{minipage}}}%
  \end{picture}%
\endgroup%

%% file: oscillations3.tex
\section{Oscillations and $L^{2}$ bounds}
\label{sec:oscillations}
The goal of this section is to show that for a homotopy $\Gamma$ satisfying the assumptions of Theorem \ref{twr:ModuliCompact}, if we fix $a,b\in \mathbb{R}$ and a compact subset $N$, such that $N\subseteq H^{-1}_{0}(0)$, then for each pair $(\Lambda_{0},\Lambda_{1})$ of connected components
$$
\Lambda_{0}\subseteq \mathscr{C}(\mathcal{A}^{H_{0}},N)\cap (\mathcal{A}^{H_{0}})^{-1}([a,\infty)), \qquad\textrm{and} \qquad
\Lambda_{1}\subseteq \Crit(\mathcal{A}^{H_{1}})\cap (\mathcal{A}^{H_{1}})^{-1}((-\infty,b]),
$$
all the Floer trajectories in the associated moduli space 
$$
\mathscr{M}^{\Gamma}(\Lambda_{0},\Lambda_{1}) \subseteq C^{\infty}(\mathbb{R} \times S^{1},\mathbb{R}^{2n}) \times C^{\infty}(\mathbb{R},\mathbb{R}),
$$
are uniformly bounded in the $L^{2}\times\mathbb{R}$ norm.

Recall that in Proposition \ref{prop:bounds1} we have already established bounds on the $\eta$ component of a Floer trajectory. However, to establish the $L^{2}$ bounds on the $v$ component, one has to analyze the Floer trajectory not only outside of set of infinitesimal action derivation, but also how far it travels within $\mathcal{B}^{\Gamma}(\mathfrak{a},\mathfrak{y},\varepsilon)$, along the hypersurface $\Sigma \times\{0\}$. To obtain the uniform bounds, we will first cover the space $C^{\infty}(S^{1}, \mathbb{R}^{2n})\times\mathbb{R}$ by three sets on which the Floer trajectory can be bounded differently.

Let $\tilde{\delta}(\textsf{H})$ be as in Lemma \ref{lem:TubN} and fix
\begin{equation}
\delta \in \big(0, \min\big\{\textrm{\etat},\tfrac{1}{6M}, \tilde{\delta}(\textsf{H})\} \big).
\label{eqn:delta}
\end{equation}
That means that the $\delta$-tubular neighborhood of $\textsf{H}^{-1}(0)$ is well defined and $\delta$ satisfies assumptions of Proposition \ref{prop:boundP}. By Proposition \ref{prop:partition} for every 
\begin{equation}
\varepsilon\in (0,\varepsilon_{2}(\tfrac{\delta}{2},\|J\|_{L^{\infty}})),
\label{eqn:varepsilon}
\end{equation}
we have the following partition of $\mathcal{B}^{\Gamma}(\mathfrak{a},\mathfrak{y},\varepsilon)$
$$
\mathcal{B}^{\Gamma}(\mathfrak{a},\mathfrak{y},\varepsilon)\subseteq K_{\delta / 2}^{1}\cup\mathcal{U}_{\delta /2}^{1}.
$$
The reason why we choose here $\frac{1}{2}\delta$ instead of $\delta$ is to estimate the number of oscillations of the Floer trajectory and will became apparent later in the proof of Lemma \ref{lem:Kbound}.

Now take $\textsl{v}_{2}(\delta / 2, \max_{x\in K\cup \mathcal{V}}\|x\|)$ as in Proposition \ref{prop:partition}, $\textsl{v}_{3}$ as in Lemma \ref{lem:ends} and $r$ as in Proposition \ref{prop:boundP} and denote 
\begin{equation}
\textsl{v}_{4}=\max\Big\{ r,\textsl{v}_{3}, \sup_{x\in K}\|x\|, \textsl{v}_{2}(\tfrac{\delta}{2},\max_{x\in K\cup\mathcal{V}}\|x\|) \Big\}.
\label{eqn:v4}
\end{equation}
Observe that $\textsl{v}_{4}$ does not depend on $\Gamma$, but only on the parameters chosen uniformly for the whole set $\mathscr{O}(\mathsf{H})$ and on the fact that $\Gamma$ satisfies (\ref{eqn:Hs2}). Now, in terms of $\textsl{v}_{4}$ we can define the following subset of $C^{\infty}(S^{1},\mathbb{R}^{2n})\times \mathbb{R}$:
\begin{equation}
K^{0}_{\delta}:= \left\lbrace
\begin{array}{c | c}
& |\eta|\leq \mathfrak{y}\\
{\smash{\raisebox{.5\normalbaselineskip}{$(v,\eta)\in C^{\infty}(S^{1},\mathbb{R}^{2n})\times \mathbb{R}$}}} & \|v\|_{L^{2}}\leq v_{4}+\delta 
\end{array}\right\rbrace
\label{eqn:K0}
\end{equation}
The set $K^{0}_{\delta}$ is bounded in $L^{2}\times\mathbb{R}$ norm. By definition of  $K^{1}_{\delta / 2}$ and  $K^{0}_{\delta}$ we know that every Floer trajectory starts in these sets due to the inclusion
$$
\mathscr{C}(\mathcal{A}^{H_{0}},N) \cap (\mathcal{A}^{H_{0}})^{-1}([a,B(\Gamma,a,b)]) \subseteq K^{1}_{\delta / 2} \subseteq  K^{0}_{\delta}.
$$
Now let $\mathscr{N}_{\textsl{v}_{4}}^{\delta}$ be the set as in (\ref{eqn:Nr}). Then we have the inclusions
$$
 \mathcal{U}_{\delta}^{0} \setminus K^{0}_{\delta}\  \subseteq\ \mathscr{N}_{\textsl{v}_{4}}^{\delta},\qquad\textrm{and}\qquad
\mathcal{B}^{\Gamma}(\mathfrak{a},\mathfrak{y},\varepsilon) \subseteq K_{\delta}^{0}\cup \mathcal{U}_{\delta}^{0}  = K_{\delta}^{0} \cup \mathscr{N}_{\textsl{v}_{4}}^{\delta}.
$$
Note that our choice of $\delta<$\etat\ and 
$$
J_{s}\in\mathscr{J}\Big(\mathbb{R}^{2n},\omega_{0},\mathcal{V}\times\big((-\infty,-\textrm{\etat})\cup( \textrm{\etat}, \infty)\big)\Big),
$$
implies that
$$
J_{s}\Big|_{\mathcal{U}^{0}_{\delta}}=J_{0},
$$
where $J_{0}$ is the standard almost complex structure.
Moreover, observe that all $H_{s}$ differ only outside a compact set, namely $H_{s}-H_{0}\in C^{\infty}_{0}(K)$. In particular the definition of $\mathscr{N}_{\textsl{v}_{4}}^{\delta}$ is independent of which $H_{s}$ we choose, since we have taken 
$$
\textsl{v}_{4} \geq  \sup_{x\in K}\|x\| .
$$
This means that if we assume $s \notin (0,1)$ then on $\mathscr{N}^{\delta}_{\textsl{v}_{4}}$ the pair $(H_{s},J_{s})$ is constant and equal either to $(H_{0},J_{0})$ or $(H_{1},J_{0})$. Finally, we have assumed $\textsl{v}_{4}\geq r$, where $r$ is as in Proposition \ref{prop:boundP}, so we can apply its results directly.

The result is that we can cover the space $C^{\infty}(S^{1}, \mathbb{R}^{2n})\times\mathbb{R}$ by three sets 
$$
K_{\delta}^{0},  \qquad \mathscr{N}_{\textsl{v}_{4}}^{\delta}, \qquad (C^{\infty}(S^{1}, \mathbb{R}^{2n})\times\mathbb{R})\setminus \mathcal{B}^{\Gamma}(\mathfrak{a},\mathfrak{y},\varepsilon),
$$
on which the Floer trajectory can be bounded differently. Whenever the Floer trajectory is outside $\mathcal{B}^{\Gamma}(\mathfrak{a},\mathfrak{y},\varepsilon)$, then its growth in the $L^{2}\times\mathbb{R}$ norm can be estimated using results from Lemma \ref{lem:Kbound}. On the other hand, Proposition \ref{prop:boundP} gives the bounds on the growth of the Floer trajectory in the $L^{2}\times\mathbb{R}$ norm, provided the Floer trajectory is in $ \mathscr{N}_{r}^{\delta}$. To obtain the uniform bounds for the whole Floer trajectory, we will also have to determine how it oscillates between the two non-compact sets, i.e. between $\mathcal{B}^{\Gamma}(\mathfrak{a},\mathfrak{y},\varepsilon)\setminus K_{\delta}^{0}$ and $(C^{\infty} (S^{1}, \mathbb{R}^{2n})\times\mathbb{R})\setminus \mathscr{N}_{\textsl{v}_{4}}^{\delta}$, which is carried out in Lemma \ref{lem:Kbound}.

Let us consider the setting of Theorem \ref{twr:ModuliCompact} and a pair of connected components
$(\Lambda_{0},\Lambda_{1})$ satisfying
\begin{gather*}
\mathscr{M}^{\Gamma}(\Lambda_{0},\Lambda_{1})  \neq \emptyset,\\
\Lambda_{0}\subseteq \mathscr{C}(\mathcal{A}^{H_{0}},N) \cap (\mathcal{A}^{H_{0}})^{-1}([a,\infty)), \quad
\Lambda_{1}\subseteq \Crit(\mathcal{A}^{H_{1}})  \cap (\mathcal{A}^{H_{1}})^{-1}((-\infty,b]),
\end{gather*}
for a fixed pair of $a,b\in \mathbb{R}$. By Lemma \ref{lem:ends} one can assume
$\Lambda_{0} \subseteq K_{\delta/ 2}^{1} \subseteq K_{\delta}^{0}$,
with the sets $K_{\delta/ 2}^{1},K_{\delta}^{0}$ defined as in (\ref{eqn:K1}) and (\ref{eqn:K0}), respectively.

Take a Floer trajectory $u \in \mathscr{M}^{\Gamma}(\Lambda_{0},\Lambda_{1})$,
and fix $s\in \mathbb{R}$. For such $s$ let us define a sequence in $\mathbb{R}$
\begin{equation}
\tau_{1}(s) \leq \tau_{2}^{-}(s) \leq \tau_{2}^{+}(s) \leq \tau_{3}^{-}(s)\dots \leq s,
\label{eqn:sequence}
\end{equation}
in the following way
\begin{align*}
\tau_{1}(s)&:=\sup\{ \tau \leq s\ |\ u(\tau)\in K_{\delta}^{0}\},\\
\tau_{2}^{-}(s)&:=\inf\{\tau_{1}(s)\leq \tau \leq s\ |\ u(\tau)\notin K_{\delta}^{0}\cup \mathscr{N}^{\delta}_{\textsl{v}_{4}}\},\\
\tau_{k}^{+}(s)&:=\inf\{\tau_{k}^{-}(s) \leq \tau \leq s\ |\ u(\tau)\in \mathcal{B}^{\Gamma}(\mathfrak{a},\mathfrak{y},\varepsilon)\},\\
\tau_{k+1}^{-}(s)&:=\inf\{\tau_{k}^{+}(s) \leq \tau \leq s\ |\ u(\tau)\notin K_{\delta}^{0}\cup \mathscr{N}^{\delta}_{\textsl{v}_{4}}\}.
\end{align*}
We stop the sequence if the set on the right-hand side becomes empty. Note that by assumption
$$
\lim_{s\to -\infty}u(s)\in \Lambda_{0} \subseteq K_{\delta / 2}^{1},
$$
hence $\tau_{1}(s)$ is well defined and finite. This means that the sequence has always at least one element.

\begin{figure}[h]
\centering
\def\svgwidth{\columnwidth}  
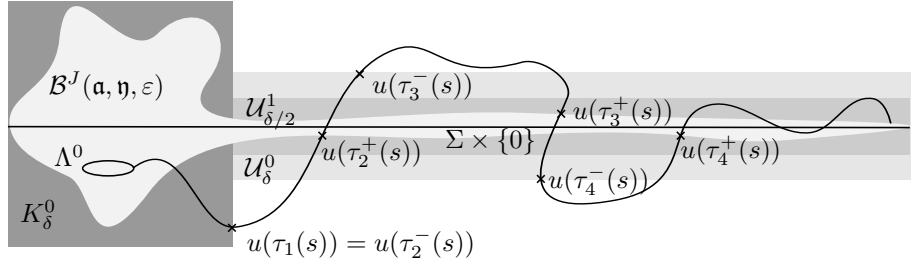
\caption{Construction of the sequence (\ref{eqn:sequence}). Here $\Lambda^{0}$ is a connected component of $\mathscr{C}(\mathcal{A}^{H_{0}},N_{0})\cap (\mathcal{A}^{H_{0}})^{-1}([a,\infty))$.}
\end{figure}

By definition the sequence satisfies the following properties
\begin{align}
& \forall\ \tau\in (\tau_{1}(s),\tau_{2}^{-}(s)), & u(\tau)\in \mathscr{N}^{\delta}_{\textsl{v}_{4}},\label{seqn1}\\
& \forall\ \tau\in (\tau_{k}^{-}(s),\tau_{k}^{+}(s)), & u(\tau)\notin \mathcal{B}^{\Gamma}(\mathfrak{a},\mathfrak{y},\varepsilon),\label{seqn2}\\
& \forall\ \tau\in (\tau_{k}^{+}(s),\tau_{k+1}^{-}(s)), & u(\tau)\in \mathscr{N}^{\delta}_{\textsl{v}_{4}} \label{seqn3}.
\end{align}
This means that the sequence carries some information on where the Floer trajectory is: in $\mathscr{N}^{\delta}_{\textsl{v}_{4}}$ or in $(C^{\infty}(S^{1},\mathbb{R}^{2n})\times\mathbb{R})\setminus \mathcal{B}^{\Gamma}(\mathfrak{a},\mathfrak{y},\varepsilon)$ (it can be in both, though). The length of the sequence indicates how many times the Floer trajectory oscillates between the two sets.

With respect to the sequence defined above, we will call an oscillation a connected component of the Floer trajectory 
of the form
$$
u^{-1}([\tau^{+}_{k}(s),\tau^{+}_{k+1}(s)]),
$$
for some $1\leq k\in \mathbb{N}$. In the following lemma we will prove that the sequence is in fact finite, thus the number of oscillations is finite, too.
\begin{lem}
Let $\Gamma=\{H_{s},J_{s}\}_{s \in \mathbb{R}}$ be a smooth homotopy of Hamiltonians and almost complex structures defined as in Theorem \ref{twr:ModuliCompact}. Fix a pair $a,b\in\mathbb{R}$ and two connected components
$$
\Lambda_{0}  \subseteq \mathscr{C}(\mathcal{A}^{H_{0}},N) \cap (\mathcal{A}^{H_{0}})^{-1}([a,\infty)),\qquad \textrm{and} \qquad
\Lambda_{1}  \subseteq \Crit(\mathcal{A}^{H_{1}}) \cap (\mathcal{A}^{H_{1}})^{-1}((-\infty,b]).
$$
Let $\mathfrak{e}$ be the bound on the energy of the Floer trajectories of $\mathscr{M}^{\Gamma}(\Lambda^{0},\Lambda^{1})$ as provided by Proposition \ref{prop:bounds1}. Fix $\delta$ as in (\ref{eqn:delta}) and $\varepsilon$ as in (\ref{eqn:varepsilon}). Fix $u \in \mathscr{M}^{\Gamma}(\Lambda^{0},\Lambda^{1})$ and $s\in \mathbb{R}$. Let
$$
\tau_{1}(s) \leq \tau_{2}^{+}(s) \leq \tau_{2}^{-}(s) \leq \tau_{3}^{+}(s)\dots \leq s,
$$
be a sequence associated to $s$ and defined as in (\ref{eqn:sequence}). Let $K\in \mathbb{N}$. Then
\begin{equation}
\sum_{k=2}^{K} \|u(\tau_{k}^{+}(s))-u(\tau_{k}^{-}(s))\|_{L^{2}(S^{1})\times \mathbb{R}} \leq \frac{\mathfrak{e}}{\varepsilon}.
\label{eqn:-k+}
\end{equation}
Moreover, this implies that the number of oscillations is in fact finite, namely
\begin{equation}
K \leq \frac{2\mathfrak{e}}{\delta\varepsilon}+1.
\label{eqn:Kbound}
\end{equation}
\label{lem:Kbound}
\end{lem}
\begin{proof}
By Proposition \ref{prop:bounds1} we have
$$
\|\nabla^{J_{s}} \mathcal{A}^{H_{s}}(u(s))\|^{2}_{L^{2}(\mathbb{R}\times S^{1})}\leq \mathfrak{e}
$$
In view of that we can estimate the time a Floer trajectory $u\in \mathscr{M}^{\Gamma}(\Lambda^{0},\Lambda^{1})$ spends outside $\mathcal{B}^{\Gamma}(\mathfrak{a},\mathfrak{y},\varepsilon)$. Denote by $\mathds{1}_{\mathcal{B}^{\Gamma}(\mathfrak{a},\mathfrak{y},\varepsilon)}$ the characteristic function of $\mathcal{B}^{\Gamma}(\mathfrak{a},\mathfrak{y},\varepsilon)$, then
\begin{align}
\mathfrak{e} & \geq \|\nabla^{J_{s}} \mathcal{A}^{H_{s}}(u(s))\|^{2}_{L^{2}(\mathbb{R}\times S^{1})} \geq \varepsilon^{2}\int_{\mathbb{R}}\big(1-\mathds{1}_{\mathcal{B}^{\Gamma}(\mathfrak{a},\mathfrak{y},\varepsilon)}(u(\tau))\big) d\tau\nonumber \\
\frac{\mathfrak{e}}{\varepsilon^{2}} & \geq \int_{\mathbb{R}}\big(1-\mathds{1}_{\mathcal{B}^{\Gamma}(\mathfrak{a},\mathfrak{y},\varepsilon)}(u(\tau))\big) d\tau. \label{eqn:time}
\end{align}
By definition
$$
 \forall\ \tau\in (\tau_{k}^{-}(s),\tau_{k}^{+}(s)), \qquad u(\tau)\notin \mathcal{B}^{\Gamma}(\mathfrak{a},\mathfrak{y},\varepsilon),
$$
therefore we can calculate
\begin{align*}
\sum_{k=2}^{K} \|u(\tau_{k}^{+}(s)) & -u(\tau_{k}^{-}(s))\|_{L^{2}\times\mathbb{R}} =\\
& = \sum_{k=2}^{K}\sqrt{ \int_{0}^{1}\Big(\int _{\tau_{k}^{-}(s)}^{\tau_{k}^{+}(s)} \|\partial_{\tau}v(\tau,t)\|d\tau\Big)^{2}dt+\Big(\int _{\tau_{k}^{-}(s)}^{\tau_{k}^{+}(s)} |\partial_{\tau}\eta(\tau)|d\tau\Big)^{2}}\\
& \leq \sum_{k=2}^{K}\sqrt{ \int_{0}^{1} \int_{\tau_{k}^{-}(s)}^{\tau_{k}^{+}(s)}1 d\tau \int _{\tau_{k}^{-}(s)}^{\tau_{k}^{+}(s)} (\|\partial_{\tau}v(\tau,t)\|^{2}+|\partial_{\tau}\eta(\tau)|^{2})d\tau dt}\\
& = \sum_{k=2}^{K}\sqrt{  \int_{\tau_{k}^{-}(s)}^{\tau_{k}^{+}(s)}1 d\tau \int _{\tau_{k}^{-}(s)}^{\tau_{k}^{+}(s)} \|\nabla^{J_{\tau}} \mathcal{A}^{H_{\tau}}(u(\tau))\|^{2}_{L^{2}\times\mathbb{R}}d\tau}\\
& \leq \sqrt{ \sum_{k=2}^{K} \int_{\tau_{k}^{-}(s)}^{\tau_{k}^{+}(s)}1 d\tau}\sqrt{\sum_{k=2}^{K} \int _{\tau_{k}^{-}(s)}^{\tau_{k}^{+}(s)} \|\nabla^{J_{\tau}} \mathcal{A}^{H_{\tau}}(u(\tau))\|^{2}_{L^{2}\times\mathbb{R}}d\tau}\\
& \leq \sqrt{ \int_{\mathbb{R}}\big(1-\mathds{1}_{\mathcal{B}(a,b,\varepsilon)}(u(\tau))\big) d\tau}\ \|\nabla^{J_{\tau}} \mathcal{A}^{H_{\tau}}(u(\tau))\|_{L^{2}(\mathbb{R}\times S^{1})}\\
& \leq \frac{\mathfrak{e}}{\varepsilon}
\end{align*}
This proves the first claim.

Due to the fact that $u(s)$ is continuous,
\begin{align*}
u(\tau_{k}^{+}(s))&\in cl(\mathcal{B}(a,b,\varepsilon)\setminus K_{\delta}^{0})\subseteq  \mathcal{U}_{\delta/ 2}^{0},\\
u(\tau_{k}^{-}(s))&\in cl((C^{\infty}(S^{1},\mathbb{R}^{2n})\times \mathbb{R})\setminus \mathcal{U}_{\delta}^{0}).
\end{align*}
Following the definition of $\mathcal{U}_{\delta}^{0}$, we get that
\begin{align*}
& dist_{L^{2}\times\mathbb{R}}(\mathcal{U}_{\delta/2}^{0}, cl(C^{\infty}(S^{1},\mathbb{R}^{2n})\times \mathbb{R}\setminus \mathcal{U}_{\delta}^{0}))=\frac{\delta}{2},\\
& dist_{L^{2}\times\mathbb{R}}(u(\tau_{k}^{-}(s)),u(\tau_{k}^{+}(s)))\geq \frac{\delta}{2}.
\end{align*}
Combining this with the previous result we obtain the claimed bound on the number of oscillations 
\begin{align*}
(K-1) \frac{\delta}{2} & \leq \sum_{k=2}^{K}\|u(\tau_{k}^{+}(s))-u(\tau_{k}^{-}(s))\|_{L^{2}\times\mathbb{R}} \leq \frac{\mathfrak{e}}{\varepsilon}\\
K & \leq \frac{2\mathfrak{e}}{\delta\varepsilon}+1.
\end{align*}
\end{proof}
After proving that the number of oscillations is finite, we will show that the $v$ component of $\mathscr{M}^{\Gamma}(\Lambda^{0},\Lambda^{1})$ is bounded in the $L^{2}$ norm. In the following proposition we first prove the $L^{2}\times \mathbb{R}$ bounds, which immediately imply the $L^{\infty}\times\mathbb{R}$ bounds inside the set of infinitesimal action derivation.

\begin{prop}
Let $\Gamma=\{H_{s},J_{s}\}_{s \in \mathbb{R}}$ be a smooth homotopy of Hamiltonians and almost complex structures defined as in Theorem \ref{twr:ModuliCompact}. Fix a pair $a,b\in\mathbb{R}$.
Then for every two connected components
$$
\Lambda_{0} \subseteq \mathscr{C}(\mathcal{A}^{H_{0}},N) \cap (\mathcal{A}^{H_{0}})^{-1}([a,\infty)),\qquad
\Lambda_{1} \subseteq \Crit(\mathcal{A}^{H_{1}})  \cap (\mathcal{A}^{H_{1}})^{-1}((-\infty,b]),
$$
the corresponding $\mathscr{M}^{\Gamma}(\Lambda^{0},\Lambda^{1})$ is uniformly bounded in the $L^{2}\times \mathbb{R}$ norm.
Moreover, for every $\varepsilon\in (0,\varepsilon_{2})$, where $\varepsilon_{2}$ is as in (\ref{eqn:varepsilon}), the corresponding set
$$
\bigcup_{u\in \mathscr{M}^{\Gamma}(\Lambda^{0},\Lambda^{1})}(u(\mathbb{R})\cap \mathcal{B}^{\Gamma}(\mathfrak{a},\mathfrak{y},\varepsilon))
$$
is uniformly bounded in $W^{1,2}(S^{1})\times \mathbb{R}$ norm.
\label{prop:Floerboundv2}
\end{prop}
\begin{proof}
Let us fix $u\in \mathscr{M}^{\Gamma}(\Lambda^{0},\Lambda^{1})$ and $s\in \mathbb{R}$. Using the sequence as in (\ref{eqn:sequence}), we will divide the Floer trajectory and bound it on each part differently.

If $s> 0$, then $(H_{s},J_{s})$ is non-constant on $(-\infty,s]$ namely it varies on the interval $[0,\min\{1,s\}]$. We can bound there the $L^{2}$ norm of the $v$ component using a similar approach as in Lemma \ref{lem:Kbound} obtaining
\begin{align}
\|v(\min\{1,s\})- v(0)\|_{L^{2}(S^{1})} & \leq \sqrt{|\min\{1,s\}| \int^{\min\{1,s\}}_{0} \|\partial_{s}v(\tau)\| ^{2}_{L^{2}}d\tau}\nonumber \\
 & \leq \|\nabla^{J_{s}} \mathcal{A}^{H_{s}}(u(s))\|_{L^{2}(S^{1}\times\mathbb{R})}\nonumber\\
 & \leq \sqrt{\mathfrak{e}} \label{eqn:01}
\end{align}

By (\ref{seqn1}) and (\ref{seqn3})
$$
\tau\in (\tau_{1}(s),\tau_{2}^{-}(s))\cup \Big(\bigcup_{k=2}^{K}(\tau_{k}^{+}(s),\tau_{k+1}^{-}(s))\Big) \quad \Longrightarrow \quad u(\tau) \in \mathscr{N}_{\textsl{v}_{4}}^{\delta}.
$$
That means that for any interval
$$
I\subseteq  \Big((\tau_{1}(s),\tau_{2}^{-}(s))\cup\Big(\bigcup_{k=2}^{K}(\tau_{k}^{+}(s),\tau_{k+1}^{-}(s))\Big)\Big)\setminus (0,1)
$$
one can apply Proposition \ref{prop:boundP}. Indeed, whenever $I \cap (0,1)=\emptyset$, then for all $\tau\in I,\ H^{\tau}$ is constant and equal either $H^{0}$ or $H^{1}$. Moreover, on $(-\infty,0)\cup (1,\infty),\ \mathcal{A}^{H_{\tau}}(u(\tau))$ is monotonically increasing and by Proposition \ref{prop:bounds1}
$$
|\mathcal{A}^{H_{\tau}}(u(\tau))| \leq \mathfrak{a}\qquad \forall\ \tau \in \mathbb{R}.
$$
Therefore, we can apply Proposition \ref{prop:boundP} and combine it with (\ref{eqn:Kbound}) and (\ref{eqn:01}) to estimate
\begin{align}
\|v(\tau_{1}(s)) & -v(\tau_{2}^{-}(s))\|_{L^{2}}+\sum_{k=2}^{K}\|v(\tau_{k}^{+}(s))-v(\tau_{k+1}^{-}(s))\|_{L^{2}} \leq \nonumber \\
& \leq  \tilde{M} \big(\mathcal{A}^{H^{0}}(u(0))- \mathcal{A}^{H^{0}}(u(\Lambda^{0}))+ \mathcal{A}^{H^{1}}(\Lambda^{1})- \mathcal{A}^{H^{1}}(u(1))\big)\nonumber \\
& + \|v(1)- v(0)\|_{L^{2}}+ K 2 \delta\nonumber \\
& \leq \sqrt{\mathfrak{e}} +4\tilde{M} \mathfrak{a} + 4\tfrac{\mathfrak{e}}{\varepsilon}+ 2\delta. \label{eqn:PL2}
\end{align}

Now we can finally estimate $\|v(s)\|_{L^{2}}$ by using the definition of $K^{0}_{\delta}$ in (\ref{eqn:K0}) and combining it with the results from (\ref{eqn:-k+}) and (\ref{eqn:PL2}) to obtain
$$
\|v(s)\|_{L^{2}}  \leq \|v(\tau_{1}(s))\|_{L^{2}} +\|v(\tau_{1}(s))-v(s)\|_{L^{2}}
 \leq \textsl{v}_{4}+\sqrt{\mathfrak{e}} +4\tilde{M} \mathfrak{a} + 5(\tfrac{\mathfrak{e}}{\varepsilon}+\delta),
$$
where the additional $2\delta$ comes from the possibility that $u(s)\in \mathscr{N}^{\delta}_{\textsl{v}_{4}}$. We have chosen $s\in \mathbb{R}$ and $u\in\mathscr{M}^{\Gamma}(\Lambda^{0},\Lambda^{1})$ arbitrary, therefore the above inequality along with the uniform bound on $\eta$ obtained in Proposition \ref{prop:bounds1} establishes that for all $u\in\mathscr{M}^{\Gamma}(\Lambda^{0},\Lambda^{1})$
$$
\|u(s)\|_{L^{2}\times\mathbb{R}} \leq \mathfrak{y}+\textsl{v}_{4}+\sqrt{\mathfrak{e}} +4\tilde{M} \mathfrak{a} + 5(\tfrac{\mathfrak{e}}{\varepsilon}+\delta)\qquad \forall\ s\in \mathbb{R},
$$
which proves the first claim.

Now take $u\in\mathscr{M}^{\Gamma}(\Lambda^{0},\Lambda^{1})$ and consider $s\in \mathbb{R}$, such that $u(s)\in \mathcal{B}^{\Gamma}(\mathfrak{a},\mathfrak{y},\varepsilon)$. We have just shown that $\|u(s)\|_{L^{2}\times\mathbb{R}}$ is uniformly bounded. Moreover, by assumption
$$
\mathcal{B}^{\Gamma}(\mathfrak{a},\mathfrak{y},\varepsilon) \subseteq K^{1}_{\delta / 2}\cup\mathcal{U}^{1}_{\delta/ 2},
$$
and by definition of $ K^{1}_{\delta / 2}$ as in (\ref{eqn:K1}) one has:
$$
\sup_{x\in K^{1}_{\delta/ 2}}\|x\|_{W^{1,2}\times\mathbb{R}} \leq \varepsilon_{2}+ \textsl{v}_{4}+\mathfrak{y}(1+h_{1}+M\textsl{v}_{4}).
$$
Therefore, we obtain a uniform bound on the $W^{1,2}\times\mathbb{R}$ norm:
\begin{align*}
\|u(s)\|_{W^{1,2}\times\mathbb{R}} \leq \textsl{v}_{4} + \max\{\varepsilon_{2}+\mathfrak{y}(1+h_{1}+M\textsl{v}_{4}),\sqrt{\mathfrak{e}} +4\tilde{M} \mathfrak{a}+ 6\delta + 5\tfrac{\mathfrak{e}}{\varepsilon}\}.
\end{align*}
Naturally, uniform bounds on the $W^{1,2}\times\mathbb{R}$ norm induce uniform bounds on the $L^{\infty}\times \mathbb{R}$ norm.
\end{proof}
First observe that uniform $W^{1,2}\times\mathbb{R}$ bounds on 
$$
\bigcup_{u\in \mathscr{M}^{\Gamma}(\Lambda^{0},\Lambda^{1})}(u(\mathbb{R})\cap \mathcal{B}^{\Gamma}(\mathfrak{a},\mathfrak{y},\varepsilon)),
$$
imply that the set
$$
\bigcup_{u\in \mathscr{M}^{\Gamma}(\Lambda^{0},\Lambda^{1})}\lim_{s\to +\infty}u(s) \subseteq \Lambda_{1},
$$
is also uniformly bounded, even in the case $\Lambda_{1}$ is non-compact, i.e. $\Lambda_{1}=H_{1}^{-1}(0)\times\{0\}$.

Moreover, note that in the bounds obtained in the above proposition $\mathfrak{y},\mathfrak{e}$ and $\varepsilon_{2}$ depend continuously on $\|J\|_{L^{\infty}}$, whereas $\textsl{v}_{4}, \delta$ and $\mathfrak{a}$ do not depend on $\|J\|_{L^{\infty}}$, but only of the fact that $\Gamma$ satisfies (\ref{eqn:Hs2}) and all the Hamiltonians are in the set $\mathsf{H}+\mathscr{O}(\mathsf{H})$ and admit a common set of parameters satisfying Properties (\ref{item:H1}), (\ref{item:H2}) and (\ref{item:H3}).

%% file: figure3A.pdf_tex
%% Creator: Inkscape inkscape 0.91, www.inkscape.org
%% Accompanies image file 'figure3A.pdf' (pdf, eps, ps)
\begingroup%
  \makeatletter%
  \providecommand\color[2][]{%
    \errmessage{(Inkscape) Color is used for the text in Inkscape, but the package 'color.sty' is not loaded}%
    \renewcommand\color[2][]{}%
  }%
  \providecommand\transparent[1]{%
    \errmessage{(Inkscape) Transparency is used (non-zero) for the text in Inkscape, but the package 'transparent.sty' is not loaded}%
    \renewcommand\transparent[1]{}%
  }%
  \providecommand\rotatebox[2]{#2}%
  \ifx\svgwidth\undefined%
    \setlength{\unitlength}{923.92451172bp}%
    \ifx\svgscale\undefined%
      \relax%
    \else%
      \setlength{\unitlength}{\unitlength * \real{\svgscale}}%
    \fi%
  \else%
    \setlength{\unitlength}{\svgwidth}%
  \fi%
  \global\let\svgwidth\undefined%
  \global\let\svgscale\undefined%
  \makeatother%
  \begin{picture}(1,0.27273132)%
    \put(0,0){\includegraphics[width=\unitlength,page=1]{figure3A.pdf}}%
    \put(0.04544614,0.19390046){\color[rgb]{0,0,0}\makebox(0,0)[lt]{\begin{minipage}{0.40785594\unitlength}\raggedright \textbf{$\mathcal{B}^{J}(\mathfrak{a},\mathfrak{y},\varepsilon)$}\end{minipage}}}%
    \put(0.26312974,0.10174067){\color[rgb]{0,0,0}\makebox(0,0)[lt]{\begin{minipage}{0.27000474\unitlength}\raggedright \textbf{$\mathcal{U}^{0}_{\delta}$}\textbf{\\ }\end{minipage}}}%
    \put(0.26512508,0.16870062){\color[rgb]{0,0,0}\makebox(0,0)[lt]{\begin{minipage}{0.29252646\unitlength}\raggedright \textbf{$\mathcal{U}^{1}_{\delta/ 2}$}\end{minipage}}}%
    \put(0.01367555,0.05148346){\color[rgb]{0,0,0}\makebox(0,0)[lt]{\begin{minipage}{0.20093519\unitlength}\raggedright \textbf{$K_{\delta}^{0}$}\end{minipage}}}%
    \put(0.48438931,0.10963495){\color[rgb]{0,0,0}\makebox(0,0)[lb]{\smash{\textbf{$\Sigma\times\{0\}$}}}}%
    \put(0.05148562,0.112778){\color[rgb]{0,0,0}\makebox(0,0)[lt]{\begin{minipage}{0.12889668\unitlength}\raggedright \textbf{$\Lambda^{0}$}\end{minipage}}}%
    \put(0.26478316,0.02043079){\color[rgb]{0,0,0}\makebox(0,0)[lt]{\begin{minipage}{0.28037656\unitlength}\raggedright \textbf{$u(\tau_{1}(s))=u(\tau_{2}^{-}(s))$}\end{minipage}}}%
    \put(0.34732361,0.12095033){\color[rgb]{0,0,0}\makebox(0,0)[lt]{\begin{minipage}{0.21002864\unitlength}\raggedright \textbf{$u(\tau^{+}_{2}(s))$}\end{minipage}}}%
    \put(0.40379962,0.19584235){\color[rgb]{0,0,0}\makebox(0,0)[lt]{\begin{minipage}{0.21002864\unitlength}\raggedright \textbf{$u(\tau^{-}_{3}(s))$}\end{minipage}}}%
    \put(0.62594125,0.16666824){\color[rgb]{0,0,0}\makebox(0,0)[lt]{\begin{minipage}{0.21002864\unitlength}\raggedright \textbf{$u(\tau^{+}_{3}(s))$}\end{minipage}}}%
    \put(0.59878135,0.08981524){\color[rgb]{0,0,0}\makebox(0,0)[lt]{\begin{minipage}{0.21002864\unitlength}\raggedright \textbf{$u(\tau^{-}_{4}(s))$}\end{minipage}}}%
    \put(0.74533325,0.1237345){\color[rgb]{0,0,0}\makebox(0,0)[lt]{\begin{minipage}{0.21002864\unitlength}\raggedright \textbf{$u(\tau^{+}_{4}(s))$}\end{minipage}}}%
    \put(0,0){\includegraphics[width=\unitlength,page=2]{figure3A.pdf}}%
  \end{picture}%
\endgroup%

%% file: MaxPrinc3.tex
\section{Maximum principle}
\label{sec:MaxPrinc}
In the previous section we have established that the Floer trajectories within the set of infinitesimal action derivation are bounded.  In this section we would like to use Aleksandrov's maximum principle to find $L^{\infty}$ bounds on the Floer trajectories outside $\mathcal{B}^{\Gamma}(\mathfrak{a},\mathfrak{y},\varepsilon)$, following the argument of Abbondandolo and Schwartz in \cite{Abbon2009}. 

\begin{twr}(\textbf{Aleksandrov's maximum principle})\\
Let $\Omega$ be a domain in $\mathbb{R}^{2}$ and $\rho:\Omega \to \mathbb{R}$ a $C^{2}$ function satisfying the elliptic differential inequality
$$
\triangle \rho + \langle h, \nabla \rho \rangle \geq f,
$$
where $h$ and $f$ are functions $h: \Omega \to \mathbb{R}^{2}$, $f:\Omega \to \mathbb{R}$. Then there exists $C>0$
$$
\sup_{\Omega} \rho \leq \sup_{\partial \Omega}\rho +C(\|h\|_{L^{2}(\Omega)})\|f^{-}\|_{L^{2}(\Omega)},
$$
provided $h$ and the negative part $f$ are in $L^{2}(\Omega)$.
\end{twr}
In order to apply Aleksandrov's maximum principle and find $L^{\infty}$ bounds on the Floer trajectories, one first has to construct a function $F$ with compact level sets, whose composition with a Floer trajectory $u$ satisfies the elliptic differential inequality
$$
\triangle (F\circ u) + \langle h, \nabla (F\circ u) \rangle \geq f %\qquad \textrm{on} \quad \Omega \subseteq \mathbb{R} \times S^{1},
$$
outside of the set of infinitesimal action derivation, i.e. on every connected component $\Omega \subseteq (\mathbb{R}\setminus \mathcal{B}^{\Gamma}(\mathfrak{a},\mathfrak{y},\varepsilon)) \times S^{1} $.
Having such an inequality, one can apply the Aleksandrov maximum principle, which gives us
$$
\sup_{\Omega} (F\circ u) \leq \sup_{\partial \Omega}(F\circ u)+ C(\|h\|_{L^{2}(\Omega)})\|f^{-}\|_{L^{2}(\Omega)},
$$
provided $h$ and the negative part of $f$ are in $L^{2}(\Omega)$.

The core of this method is to find a function satisfying all the required properties. The classical approach is to use plurisubhamonic functions. 
%In \cite{Abbon2009} Alberto Abbondandolo and Mathias Schwartz apply the Aleksandrov's maximum principle to a specific plurisubharmonic function. %, which they construct using the fact that in their setting the hypersurface is a boundary of a compact Liouville domain. 
%A general definition of plurisubharmonic functions, we recall bellow:
\begin{define}
Let $(M,\omega)$ be a symplectic manifold and $J$ an $\omega$-compatible almost complex structure.
Then a $C^{2}$ function $F:M\to\mathbb{R}$ is called plurisubharmonic if
$$
-dd^{\mathbb{C}}F=\omega,
$$
where $d^{\mathbb{C}}F=dF\circ J$.
\end{define}
The reason one uses plurisubharmonic functions is because their composition with a $J$-holomorphic curve trivially satisfies the elliptic inequality. Unfortunately, in the case of Floer trajectories proving the elliptic inequality is a little more complicated. One has to investigate how the plurisubharmonic function interacts with the Hamiltonian vector field, in particular one needs to understand the functions
$$
dF(X^{H})\qquad d^{\mathbb{C}}F(X^{H}),
$$
which appear if we calculate $d^{\mathbb{C}}(F\circ u)$. Moreover, the framework used in \cite{Abbon2009} cannot be translated directly to our case, since our hypersurface is not a boundary of a compact Liouville domain. However, in Proposition \ref{prop:max}, we will present a setting, which applies to the system from Theorem \ref{twr:ModuliCompact}.

\subsection{Elliptic differential inequality}
Let us put ourselves in the setting of Theorem \ref{twr:ModuliCompact}. In the previous section we have established that the Floer trajectories within the set of infinitesimal action derivation are bounded. More precisely, by Proposition \ref{prop:Floerboundv2} for every $\varepsilon$ as in (\ref{eqn:varepsilon}) the set 
$$
\bigcup_{u\in \mathscr{M}^{\Gamma}(\Lambda^{0},\Lambda^{1})}(u(\mathbb{R})\cap \mathcal{B}^{\Gamma}(\mathfrak{a},\mathfrak{y},\varepsilon)),
$$
is uniformly bounded in the $W^{1,2}(S^{1})\times \mathbb{R}$ norm. This enables us to choose a compact subset $K^{\infty}\subseteq \mathbb{R}^{2n}$, such that %$\mathcal{V}\cup K \subseteq K^{\infty}$ and 
\begin{align}
& \mathcal{V}\cup K \subseteq K^{\infty}, \label{eqn:VHK}\\
\forall\ (v,\eta)\in & \bigcup_{u\in \mathscr{M}^{\Gamma}(\Lambda^{0},\Lambda^{1})}(u(\mathbb{R})\cap \mathcal{B}^{\Gamma}(\mathfrak{a},\mathfrak{y},\varepsilon)) \qquad v(S^{1}) \subseteq K^{\infty}.\label{eqn:BK}
\end{align}
\begin{figure}[t]
\centering
\def\svgwidth{0.7\columnwidth}  
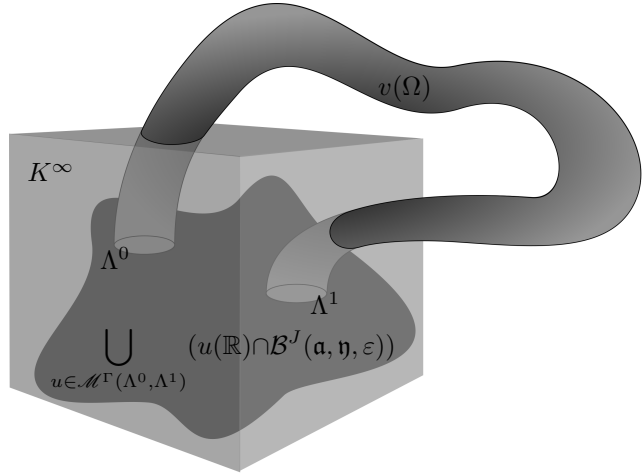
\caption{Construction of the set $\Omega$ as corresponding to the part of a Floer trajectory outside compact set $K^{\infty}$.}
\end{figure}
However, to establish uniform bounds on the whole moduli space we have to investigate the Floer trajectories outside of the compact set $K^{\infty}$. Now, if we choose $u\in \mathscr{M}^{\Gamma}(\Lambda_{0},\Lambda_{1}), u(s,t)=(v(s,t), \eta(s))$ and take a connected component
\begin{equation}
\Omega\subseteq v^{-1} (\mathbb{R}^{2n}\setminus K^{\infty}),
\label{eqn:Omg}
\end{equation}
then by (\ref{eqn:BK}) and (\ref{eqn:VHK}) the following holds for all $(s,t) \in \Omega$
\begin{align*}
H^{s}(v(s,t)) & = H^{0}(v(s,t))= H^{1}(v(s,t))\\
J_{t}(v(s,t),\eta(s)) & = J_{0}\\
\|\nabla \mathcal{A}^{H}(u(s))\|_{L^{2}\times \mathbb{R}} & > \varepsilon.
\end{align*}

In other words, we can assume that the Hamiltonian and the almost complex structure are constant on $u(\Omega)$ and the action derivation is uniformly bounded away from $0$. On the other hand, the radial function
$$
F(x):=\tfrac{1}{4}\|x\|^{2},
$$
on $(\mathbb{R}^{2n},\omega_{0})$ with the standard complex structure $J_{0}$ is an example of a plurisubharmonic function with compact level sets. 

In the following Proposition, we will prove that the radial function composed with a Floer trajectory %outside the set $\mathcal{B}^{\Gamma}(\mathfrak{a},\mathfrak{y},\varepsilon)$ 
satisfies the assumptions of Aleksandrov's maximum principle on $\Omega\subseteq v^{-1} (\mathbb{R}^{2n}\setminus K^{\infty})$, provided $u$ is globally bounded in the $L^{2}\times \mathbb{R}$ norm and the Hamiltonian function satisfies 
$$
\sup_{x\in\mathbb{R}^{2n}}\|D^{3}H\|\|x\| \leq L.
$$
This will allow us to apply Aleksandrov's maximum principle and prove that the radial function on the Floer trajectory outside of the set $\mathcal{B}^{\Gamma}(\mathfrak{a},\mathfrak{y},\varepsilon)$ is bounded, thus concluding the proof of Theorem \ref{twr:ModuliCompact}.

\begin{prop}
Let $H:\mathbb{R}^{2n}\to\mathbb{R}$ be a Hamiltonian and $u: \mathbb{R}\times S^{1} \to \mathbb{R}^{2n+1},\ u(s)=(v(s),\eta(s))$ be a solution to the Floer equation corresponding to the constant almost complex structure $J\equiv J_{0}$. Then for the radial, plurisubharmonic function $F:\mathbb{R}^{2n}\to\mathbb{R}$,
$$
F(x):= \tfrac{1}{4}\|x\|^{2}
$$
there exists a function $f:\mathbb{R}\times S^{1} \to \mathbb{R}$, such that
$$
\triangle (F\circ v) \geq f(s,t).
$$
Moreover, suppose the following conditions are satisfied:
\begin{enumerate}
\item $H$ satisfies (\ref{item:H2})
\item there exist constants $\mathfrak{e},\mathbf{v}, \mathfrak{y} >0$, such that
$$
\|\nabla \mathcal{A}^{H}(u)\|_{L^{2}(\mathbb{R}\times S^{1})}^{2} \leq \mathfrak{e},\qquad
\|v(s)\|_{L^{2}} \leq \mathbf{v}, \qquad |\eta(s)|  \leq \mathfrak{y}, \qquad \forall s\in \mathbb{R}.
$$
\item there exists a positive constant $\varepsilon>0$ and a connected, open subset $\Omega\subseteq \mathbb{R}\times S^{1}$, such that if
we define
$$
s_{0}:=\inf_{(s,t)\in\Omega}s, \qquad s_{1}:=\sup_{(s,t)\in\Omega}s,
$$
then for all $s\in (s_{0},s_{1})$ one has $\|\nabla\mathcal{A}^{H}(u(s))\|>\varepsilon$.
\end{enumerate}
Then 
$$
\| f \|_{L^{2}(\Omega)}\leq C(\mathfrak{e},\mathfrak{y},\mathbf{v},\varepsilon)<+\infty
$$
where the constant $C$ depends only on the parameters of the Hamiltonian and the constants $\varepsilon, \mathbf{v}, \mathfrak{y}, \mathfrak{e}$ in a continuous way.
\label{prop:max}
\end{prop}
\begin{proof}
From the Floer equations we have
\[
\partial_{s}v = -J(\partial_{t}v-\eta X^{H})\nonumber, \qquad \partial_{t}v = J\partial_{s}v +\eta X^{H}.\label{eqn:parT}
\]
If we plug the above into $d^{\mathbb{C}}(F\circ v)$ we obtain
\begin{align*}
d^{\mathbb{C}}(F\circ v) & = dF(\partial_{t}v)ds-dF(\partial_{s}v)dt\\
& = dF(J\partial_{s}v +\eta X^{H})ds+dF(J(\partial_{t}v-\eta X^{H}))dt\\
& = v^{*}(d^{\mathbb{C}}F)+\eta(dF(X^{H})ds-d^{\mathbb{C}}F(X^{H})dt)\\
dd^{\mathbb{C}}(F\circ v) & = v^{*}(dd^{\mathbb{C}}F)-\Big(\eta d(dF(X^{H}))(\partial_{t}v)+\partial_{s}\eta\ d^{\mathbb{C}}F(X^{H})+ \eta d(d^{\mathbb{C}}F(X^{H}))(\partial_{s}v)\Big)ds\wedge dt
\end{align*}
Let us consider the two parts of the above expression separately. Using the $\omega_{0}$ compatibility of $J$ and plurisubharmonicity of $F$, we can compute
$$
v^{*}(-dd^{\mathbb{C}}F) = \omega_{0}(\partial_{s}v,\partial_{t}v)ds\wedge dt = \omega_{0}(\partial_{s}v,J\partial_{s}v +\eta X^{H})ds\wedge dt = (\|\partial_{s}v\|^{2}+\eta dH(\partial_{s}v))ds \wedge dt.
$$
Let us now combine the above results together with the fact that
$$
-dd^{\mathbb{C}}(F\circ v)=\triangle (F\circ v)ds\wedge dt,
$$
and compute
\begin{align*}
\triangle (F\circ v) & = \omega_{0}(\partial_{s}v,\partial_{t}v) +\eta d(dF(X^{H})(\partial_{t}v)+ \partial_{s}\eta\ d^{\mathbb{C}}F(X^{H})+ \eta d(d^{\mathbb{C}}F(X^{H}))(\partial_{s}v) \\
& = \|\partial_{s}v\|^{2}+\eta (dH(\partial_{s}v)+d(d^{\mathbb{C}}F(X^{H}))(\partial_{s}v)+d(dF(X^{H}))(J\partial_{s}v +\eta X^{H})\\
& \hspace{8.15cm} +\partial_{s}\eta\ d^{\mathbb{C}}F(X^{H})\\
& = \|\partial_{s}v +\eta(\nabla H +\nabla (d^{\mathbb{C}}F(X^{H}))-J\nabla(dF(X^{H}))  )\|^{2}+\eta^{2}d(dF(X^{H}))(X^{H})\\
& \hspace{1.9cm} - \eta^{2}\|\nabla H +\nabla (d^{\mathbb{C}}F(X^{H}))-J\nabla(dF(X^{H})) \|^{2}+\partial_{s}\eta\ d^{\mathbb{C}}F(X^{H})\\
& \geq \eta^{2}(d(dF(X^{H}))(X^{H}) -\|\nabla H\|^{2} -\|\nabla (d^{\mathbb{C}}F(X^{H}))\|^{2}-\|\nabla(dF(X^{H}))\|^{2})\\
& \hspace{8.15cm} +\partial_{s}\eta\ d^{\mathbb{C}}F(X^{H}).
\end{align*}
To simplify we will introduce the following notation:
\begin{align}
f_{1}(x) & :=d_{x}(dF(X^{H}))(X^{H}) -\|\nabla H(x)\|^{2} -\|\nabla (d^{\mathbb{C}}_{x}F(X^{H}))\|^{2}-\|\nabla(dF(X^{H}))(x)\|^{2}, \label{eqn:f1}\\
f_{2}(x) & := d_{x}^{\mathbb{C}}F(X^{H}),\label{eqn:f2}\\
f(s,t) & :=  \eta^{2}(s)f_{1}(v(s,t))+\partial_{s}\eta f_{2}(v(s,t)). \nonumber
\end{align}
We will prove that $f \in L^{2}(\Omega)$ by first proving that $f \in W^{1,1}(\Omega)$ and then using the Sobolev embedding
$$
W^{1,1}(\Omega) \hookrightarrow L^{2}(\Omega).
$$
Let us now analyze the assumptions on $H$ and how they imply the boundedness of $f$ in $W^{1,1}$. Property (\ref{item:H2}) implies quadratic behavior of $H$, which in turn forces $|f_{1}(x)|$ and $|f_{2}(x)|$ to be maximally of order 2 in $\|x\|$, and $\|\nabla f_{1}(x)\|$ and $\|\nabla f_{2}(x)\|$ to be linear in $\|x\|$. In other words by Lemma \ref{lem:f1}, there exist constants $\alpha_{1},\alpha_{2},\alpha_{3},\alpha_{4}>0$, which can be expressed in terms of constants from (\ref{eqn:HessH}), (\ref{eqn:nablaH}) and (\ref{eqn:H0}), such that the following holds for all $x\in\mathbb{R}^{2n}$
\begin{eqnarray}
| f_{1}(x)| \leq \alpha_{1}(\|x\|+1)^{2}, &| f_{2}(x)| \leq \alpha_{2}(\|x\|+1)^{2},\\
\| \nabla f_{1}(x)\| \leq \alpha_{3}(\|x\|+1), & \| \nabla f_{2}(x)\| \leq \alpha_{4}(\|x\|+1).
\end{eqnarray}
Moreover, by assumption we have finiteness of the energy 
$$
\|\nabla \mathcal{A}^{H}(u)\|_{L^{2}(\mathbb{R}\times S^{1})}^{2}=\int_{\mathbb{R}}\Big(|\partial_{s}\eta(s)|^{2}+\int_{0}^{1}\|\partial_{s}v(s,t)\|^{2}dt\Big)ds \leq \mathfrak{e},
$$
which implies the bound on the length of the interval $[s_{0},s_{1}]$ estimated in (\ref{eqn:time}). More precisely, as shown in Lemma \ref{lem:Kbound}, we have the following relation
$$
u(s)\notin \mathcal{B}^{\Gamma}(\mathfrak{a},\mathfrak{y},\varepsilon) \quad \forall\ s\in [s_{0},s_{1}], \qquad \Rightarrow \qquad |s_{1}-s_{0}|\leq \tfrac{\mathfrak{e}}{\varepsilon^{2}}.
$$
Combining the above results with the global bound in the $L^{2}\times \mathbb{R}$ norm on $u(s)$, we can estimate the $W^{1,1}$ norm of $f$:
\begin{align*}
\| f \|_{L^{1}(\Omega)} &\leq \mathfrak{y}^{2} \|f_{1}(v)\|_{L^{1}(\Omega)}+\|\partial_{s}\eta f_{2}(v)\|_{L^{1}(\Omega)}\\
& \leq \alpha_{1}\mathfrak{y}^{2} (\mathbf{v}+1)^{2}|s_{1}-s_{0}|+\alpha_{2}\|\partial_{s}\eta\|_{L^{2}(\Omega)}\sqrt{|s_{1}-s_{0}|}(\mathbf{v}+1)^{2}\\
& \leq \tfrac{\mathfrak{e}}{\varepsilon}(\mathbf{v}+1)^{2} (\alpha_{1}\tfrac{\mathfrak{y}^{2}}{\varepsilon}+\alpha_{2})\\
\|\partial_{t}f\|_{L^{1}(\Omega)} & \leq (\mathfrak{y}^{2} \|\nabla f_{1}(v)\|_{L^{2}(\Omega)}+\|\partial_{s}\eta \nabla f_{2}(v)\|_{L^{2}(\Omega)})\|\partial_{t}v\|_{L^{2}(\Omega)}\\
& \leq (\alpha_{3}\mathfrak{y}^{2}(\mathbf{v}+1)\sqrt{|s_{1}-s_{0}|}+\alpha_{4}(\mathbf{v}+1)\|\partial_{s}\eta\|_{L^{2}(\Omega)})\|\partial_{t}v\|_{L^{2}(\Omega)}\\
& \leq (\mathbf{v}+1)(\alpha_{3}\mathfrak{y}^{2}\tfrac{\sqrt{\mathfrak{e}}}{\varepsilon}+\alpha_{4}\|\partial_{s}\eta\|_{L^{2}(\Omega)})(\|\partial_{s}v\|_{L^{2}(\Omega)}+\mathfrak{y}(M\mathbf{v}+h_{1})\sqrt{|s_{1}-s_{0}|})\\
& \leq \mathfrak{e}(\mathbf{v}+1)(\alpha_{3}\tfrac{\mathfrak{y}^{2}}{\varepsilon}+\alpha_{4})(1+\tfrac{\mathfrak{y}}{\varepsilon}(M\mathbf{v}+h_{1}))
\end{align*}
The second to last equation follows from the Floer equations and (\ref{eqn:nablaH}), which can be applied to the norm of $X^{H}$. Now, if we keep in mind that from the Floer equations it also follows that 
$$
\partial_{ss}\eta = - \int_{S^{1}}dH_{v}(\partial_{s}v)dt,
$$
then we can apply it to estimate the norm of $\partial_{s}f$ and obtain
\begin{align*}
\|\partial_{s}f\|_{L^{1}(\Omega)} & \leq (\mathfrak{y}^{2} \|\nabla f_{1}(v)\|_{L^{2}(\Omega)}+\|\partial_{s}\eta \nabla f_{2}(v)\|_{L^{2}(\Omega)})\|\partial_{s}v\|_{L^{2}(\Omega)}\\
& + 2\mathfrak{y}\| \partial_{s}\eta\ f_{1}(v)\|_{L^{1}(\Omega)}+\|\partial_{ss}\eta\ f_{2}(v)\|_{L^{1}(\Omega)}\\
& \leq \sqrt{\mathfrak{e}}(\mathbf{v}+1)(\alpha_{3}\tfrac{\mathfrak{y}^{2}}{\varepsilon}+\alpha_{4})\|\partial_{s}v\|_{L^{2}(\Omega)}\\
& + 2\mathfrak{y} \alpha_{1}(\mathbf{v}+1)^{2}\| \partial_{s}\eta\|_{L^{1}(\Omega)}+\alpha_{2}(\mathbf{v}+1)^{2}\| dH(\partial_{s}v)\|_{L^{1}(\Omega)}\\
& \leq \mathfrak{e}(\mathbf{v}+1)(\tfrac{1}{\varepsilon}(\mathbf{v}+1)( 2\alpha_{1}\mathfrak{y}+\alpha_{2}(M\mathbf{v}+h_{1}))+\alpha_{3}\tfrac{\mathfrak{y}^{2}}{\varepsilon}+\alpha_{4}).
\end{align*}

In this way we have bounded $W^{1,1}$ norm of $f$ on $\Omega$. Now using the Sobolev embedding of $W^{1,1}$ in $L^{2}$, we can in fact obtain that the $L^{2}$ norm of $f$ is bounded by a constant, which depends only on the parameters of the Hamiltonian, the chosen value of $\varepsilon$ and the constants $\mathbf{v}, \mathfrak{y}, \mathfrak{e}$.

%The explicit estimates on $\|f\|_{L^{1}(\Omega)}$ and $\|\nabla f\|_{L^{1}(\Omega)}$ are calculated in Appendix \ref{app:f}, Lemma \ref{lem:f}.
\end{proof}

Let us now consider a moduli space $\mathscr{M}^{\Gamma}(\Lambda_{0},\Lambda_{1})$ as in Theorem \ref{twr:ModuliCompact}. By Proposition \ref{prop:Floerboundv2} there exists a compact set $K^{\infty}\subseteq \mathbb{R}^{2n}$ satisfying (\ref{eqn:VHK}) and (\ref{eqn:BK}). Fix $u\in \mathscr{M}^{\Gamma}(\Lambda_{0},\Lambda_{1})$, $u(s,t)=(v(s,t), \eta(s))$, and take a connected component $\Omega\subseteq v^{-1} (\mathbb{R}^{2n}\setminus K^{\infty})$ as in (\ref{eqn:Omg}). Then on $\Omega$, we can apply Proposition \ref{prop:max} to a composition of $u$ with the radial function $F$ directly, which together with Aleksandrov's maximum principle assures that there exists a constant $C(\mathfrak{e},\mathfrak{y},\mathbf{v},\varepsilon)>0$, such that the following inequality is satisfied
$$
\sup_{\Omega}\|v(s,t)\|  \leq \sup_{\partial K^{\infty}}\|v(s,t)\|+C(\mathfrak{e},\mathfrak{y},\mathbf{v},\varepsilon).
$$
By Proposition \ref{prop:max} the constant $C(\mathfrak{e},\mathfrak{y},\mathbf{v},\varepsilon)>0$ does not depend on the choice of $u \in \mathscr{M}^{\Gamma}(\Lambda_{0},\Lambda_{1})$ or $\Omega$, therefore, we can conclude that for all $u\in\mathscr{M}^{\Gamma}(\Lambda_{0},\Lambda_{1})$ we have
$$
\sup_{s\in \mathbb{R}}\|u(s)\|_{L^{\infty}(S^{1})\times \mathbb{R}} \leq \sup_{x\in K^{\infty}}\|x\|+\mathfrak{y}+ C(\mathfrak{e},\mathfrak{y},\mathbf{v},\varepsilon),
$$
thus establishing uniform $L^{\infty}\times \mathbb{R}$ bounds on $\mathscr{M}^{\Gamma}(\Lambda_{0},\Lambda_{1})$ and concluding the proof of Theorem \ref{twr:ModuliCompact}.
%\end{proof}

%% file: figure5A.pdf_tex
%% Creator: Inkscape inkscape 0.91, www.inkscape.org
%% Accompanies image file 'figure5A.pdf’ (pdf, eps, ps)
\begingroup%
  \makeatletter%
  \providecommand\color[2][]{%
    \errmessage{(Inkscape) Color is used for the text in Inkscape, but the package 'color.sty' is not loaded}%
    \renewcommand\color[2][]{}%
  }%
  \providecommand\transparent[1]{%
    \errmessage{(Inkscape) Transparency is used (non-zero) for the text in Inkscape, but the package 'transparent.sty' is not loaded}%
    \renewcommand\transparent[1]{}%
  }%
  \providecommand\rotatebox[2]{#2}%
  \ifx\svgwidth\undefined%
    \setlength{\unitlength}{646.18291245bp}%
    \ifx\svgscale\undefined%
      \relax%
    \else%
      \setlength{\unitlength}{\unitlength * \real{\svgscale}}%
    \fi%
  \else%
    \setlength{\unitlength}{\svgwidth}%
  \fi%
  \global\let\svgwidth\undefined%
  \global\let\svgscale\undefined%
  \makeatother%
  \begin{picture}(1,0.7434436)%
    \put(0,0){\includegraphics[width=\unitlength]{figure5A.pdf}}%
    \put(0.06541666,0.2700293){\color[rgb]{0,0,0}\makebox(0,0)[lt]{\begin{minipage}{0.39769357\unitlength}\raggedright $$\bigcup_{u\in \mathscr{M}^{\Gamma}(\Lambda^{0},\Lambda^{1})}(u(\mathbb{R})\cap \mathcal{B}^{J}(\mathfrak{a},\mathfrak{y},\varepsilon))$$\end{minipage}}}%
    \put(0.58410866,0.62982229){\color[rgb]{0,0,0}\makebox(0,0)[lt]{\begin{minipage}{0.35142108\unitlength}\raggedright $v(\Omega)$\end{minipage}}}%
    \put(0.02883836,0.4910047){\color[rgb]{0,0,0}\makebox(0,0)[lt]{\begin{minipage}{0.30764978\unitlength}\raggedright $K^{\infty}$\end{minipage}}}%
    \put(0.14389436,0.32342316){\color[rgb]{0,0,0}\makebox(0,0)[lb]{\smash{$\Lambda^{0}$}}}%
    \put(0.47655617,0.25225492){\color[rgb]{0,0,0}\makebox(0,0)[lb]{\smash{$\Lambda^{1}$}}}%
  \end{picture}%
\endgroup%

%% file: geomTen2.tex
\section{Geometrical properties of Hamiltonians}
\label{ssec:GeomHam}
%\todo{the labeling doesn't work right}
%\todo{Relabeled from subsection to section.One has to check references.}
In this Appendix we present two lemmas, which analyze the geometrical behavior of the $0$-level set of an admissible Hamiltonian, which is used in the proof of the partition of the set of infinitesimal action derivation in Proposition \ref{prop:partition} and in the definition of the tubular neighborhood in subsection \ref{ssec:TubN}.

\begin{lem}
Let $H:\mathbb{R}^{2n}\to\mathbb{R}$ satisfy (\ref{item:H1}) and (\ref{item:H3}).
Then for all $\delta>0$ there exists a $\mu(\delta,H)>0$, such that 
$$
H^{-1}(-\mu,\mu)\subseteq \Big\{x\in\mathbb{R}^{2n}\ \Big|\ \dist(x,H^{-1}(0))<\delta\Big\}.
$$
\label{lem:banan}
\end{lem}
\begin{proof}
In the first step we show that for $\nu>0$ as in (\ref{item:H3}) $H$ has no critical points in the neighborhood $H^{-1}(-\nu,\nu)$, that is
$$
\inf_{H^{-1}(-\nu,\nu)}\|\nabla H\|>0.
$$
By (\ref{item:H1}) for every $x\in\mathbb{R}^{2n}$ we have 
$$
\|\nabla H(x)\| \geq \frac{c_{2}\|x\|^{2}-c_{3}}{c_{1}(\|x\|+1)}.
$$
On the other hand, by (\ref{item:H3}) for all $x\in H^{-1}(-\nu,\nu)$ we have
$$
\|\nabla H(x)\| \geq \frac{c_{5}}{c_{4}(\|x\|+1)}.
$$
If we denote
$$
f_{1}(r)  := \frac{c_{2}r^{2}-c_{3}}{c_{1}(r+1)},\qquad \textrm{and} \qquad
f_{2}(r)  :=  \frac{c_{5}}{c_{4}(r+1)},
$$
then we have
$$
\inf_{x\in H^{-1}(-\nu,\nu)}\|\nabla H(x)\| \geq \inf_{r\geq 0}\max\{f_{1}(r), f_{2}(r)\}.
$$
An analysis of the derivatives and the asymptotic behavior of the two functions, brings us to a conclusion that there exists an $r_{0}\in(0,+\infty)$, such that $f_{1}(r_{0})=f_{2}(r_{0})$. That means that the function $\max\{f_{1}(r), f_{2}(r)\}$ obtains its minimum at $r_{0}$ and this minimum is positive, since $f_{2}$ is everywhere positive.

Denote $\Sigma=H^{-1}(0)$. On $H^{-1}(-\nu,\nu)$ we can define a flow $\psi$ in the following way
$$
\frac{d}{dt}\psi(t,x) = \frac{\nabla H(\psi(t,x))}{\|\nabla H(\psi(t,x))\|^{2}}.
$$
Observe that
$$
H(\psi(t,x))=H(x)+t\quad\textrm{since}\quad\frac{d}{dt}H(\psi(t,x)) = 1.
$$
That means that
$$
\psi:(-\nu, \nu) \times \Sigma \to H^{-1}(-\nu, \nu),
$$
is a well defined bijection. Note that for every $(t,x)\in (-\nu, \nu) \times \Sigma$, we have
\begin{align*}
\dist(\psi(t,x),\Sigma) & \leq \|\psi(t,x)-\psi(0,x)\|\\
& \leq \int_{0}^{t} \|\frac{d}{dt}\psi(\tau,x)\| d\tau\\
& \leq \int_{0}^{t} \frac{1}{\|\nabla H(\psi(\tau,x))\|} d\tau\\
& \leq \frac{t}{\inf_{r\geq 0}\max\{f_{1}(r), f_{2}(r)\}}\\
& = \frac{H(\psi(t,x))}{\inf_{r\geq 0}\max\{f_{1}(r), f_{2}(r)\}}
\end{align*}
Therefore, for all $0<\delta$ if denote
$$
\mu(\delta,H):= \min\{\nu, \delta\inf_{r\geq 0}\max\{f_{1}(r), f_{2}(r)\}\},
$$
then for all $x\in H^{-1}(-\mu,\mu)$ we have
$$
\dist(x,\Sigma) \leq  \frac{H(x)}{\inf_{r\geq 0}\max\{f_{1}(r), f_{2}(r)\}} \leq \frac{\mu}{\inf_{r\geq 0}\max\{f_{1}(r), f_{2}(r)\}} \leq \delta.
$$
\end{proof}

The following lemma ensures that for every admissible Hamiltonian, there exists a tubular neighborhood with a constant radius, which is used in Section \ref{sec:critMfld} to prove the boundedness of Floer trajectories.

\begin{lem}
Let $H:\mathbb{R}^{m}\to \mathbb{R}$ be a smooth function. Denote $\Sigma=H^{-1}(0)$ and let $N(\Sigma)$ be its normal bundle. For $\delta>0$ denote
$$
N_{\delta}(\Sigma):=\{ (x,v)\in N(\Sigma)\ |\ \|v\|<\delta\}.
$$
If 
$$
\inf_{\Sigma}\|\nabla H\|>0 \qquad \textrm{and} \qquad \sup_{\mathbb{R}^{m}} \|D^{2}H\| = M<+\infty,
$$
then there exists $\tilde{\delta}(H)>0$, such that
$$
exp: N_{\tilde{\delta}(H)}(\Sigma) \to \{x\in\mathbb{R}^{m}\ |\ \dist(x,\Sigma)<\tilde{\delta}(H)\},
$$
is a diffeomorphism.
\label{lem:TubN}
\end{lem}
Observe that in Lemma \ref{lem:banan} it was shown that from Properties (\ref{item:H1}) and (\ref{item:H3}) it follows that:
$$
\inf_{\Sigma}\|\nabla H\| >0.
$$
As a result, any admissible Hamiltonian $H:\mathbb{R}^{2n}\to \mathbb{R}$ satisfies the assumptions of Lemma \ref{lem:TubN}, hence its $0$ level set admits a well defined $\delta$-tubular neighborhood.
\begin{proof}
The proof relies on Theorem 2.31 from \cite{Eldering2013}, which states that $\tilde{\delta}(H)>0$ exists provided $\Sigma$ is a uniformly embedded submanifold as defined in Definition 2.21 and Remark 2.22 in \cite{Eldering2013}. %\todo{Maybe here we should cite his article instead of thesis}

A hypersurface $\Sigma=H^{-1}(0)$ is a uniformly embedded submanifold of $\mathbb{R}^{m}$ if there exists $\delta_{0}>0$, such that for every $x\in \Sigma$
\begin{enumerate}[label*=\arabic*.]
\item the intersection of $\Sigma$ with a ball of radius $\delta_{0},\ B(x,\delta_{0})\cap \Sigma$ has only one connected component.
\item there exists a function $\varphi: Ker(d_{x}H) \to \mathbb{R}$, such that $\Sigma$ is locally a graph of $\varphi$ over the tangent space $T_{x}\Sigma$, i.e.
$$
B(x,\delta_{0})\cap \Sigma = \Big\{ y+\varphi(y)\frac{\nabla H(x)}{\|\nabla H(x)\|}+x\ \Big|\ y\in Ker(d_{x}H), \quad \|y\|< \delta_{0}\Big\},
$$
and the functions $\varphi$ corresponding to different $x\in\Sigma$ have their first and second derivatives uniformly bounded.
\end{enumerate}
We will first prove 1. and then 2.

\textbf{Proof of 1:}\\
We will show that for every $x\in\Sigma$ the intersection of $\Sigma$ with a ball of radius $\frac{1}{M}\inf_{\Sigma}\|\nabla H\|$:
$$
B\big(x, \tfrac{1}{M}\inf_{\Sigma}\|\nabla H\|\big)\cap \Sigma,
$$
has only one connected component. Suppose the opposite is true i.e. there exists a $x\in \Sigma$ and connected component $\widetilde{\Sigma}$ of $B(x, \frac{1}{M}\inf_{\Sigma}\|\nabla H\|)\cap \Sigma$, such that $x\notin \widetilde{\Sigma}$. Then there exists $z\in \widetilde{\Sigma}$, such that
\begin{equation}
0<\dist(\widetilde{\Sigma},x)= \|x-z\|<\frac{1}{M}\inf_{\Sigma}\|\nabla H\|.
\label{eqn:distXZ}
\end{equation}
Since $z\in \widetilde{\Sigma}$ minimizes the distance between $x$ and $\widetilde{\Sigma}$, the vector $x-z$ is perpendicular to $\widetilde{\Sigma}$, i.e.
\begin{equation}
\frac{x-z}{\|x-z\|}=\frac{\nabla H(z)}{\|\nabla H(z)\|}.
\label{eqn:deff}
\end{equation}
Define a function $f : [0,1] \to \mathbb{R}$
$$
f(t) := H(tx+(1-t)z).
$$
Observe that $f(0)=f(1)=0$, hence there exists $t_{0}\in (0,1)$, such that $f'(t_{0})=0$. 
On one hand, we will have
$$
| f'(t_{0})-f'(0)|  \leq t_{0} \sup_{t\in[0,t_{0}]}|f''(t)| = t_{0} \sup_{t\in[0,t_{0}]}|D^{2}_{tx+(1-t)z}H(x-z,x-z)| \leq M\|x-z\|^{2}.
$$
On the other hand, by (\ref{eqn:deff}) we have
$$
| f'(t_{0})-f'(0)|=| f'(0)| = d_{z}H(x-z)=\|\nabla H(z)\| \|x-z\|,
$$
and the two inequalities above combined with (\ref{eqn:distXZ}) lead to a contradiction:
$$
\|\nabla H(z)\| \leq M \|x-z\| <\inf_{\Sigma} \|\nabla H\|.
$$
\textbf{Proof of 2:}\\
Fix $x\in\Sigma$. By definition $\Sigma$ is a regular level set of of $H$ in $\mathbb{R}^{m}$, hence there exists a neighborhood $V$ of $0$ in $T_{x}\Sigma$, a neighborhood $U_{x}$ of $x$ in $\mathbb{R}^{m}$ and a smooth function $\varphi:V \to \mathbb{R}$, such that
$$
U_{x}\cap \Sigma = \Big\{ y+\varphi(y)\frac{\nabla H(x)}{\|\nabla H(x)\|}+x\ \Big|\ y\in V\Big\}.
$$
Let us analyze $\varphi$. Since $H$ vanishes on $\Sigma$ we can find the following relations between $H$ and $\varphi$:
\begin{align}
0 & = H\Big(y+\varphi(y)\frac{\nabla H(x)}{\|\nabla H(x)\|}+x\Big)\nonumber \\
0 & = \frac{d}{d y_{i}} H\Big(y+\varphi(y)\frac{\nabla H(x)}{\|\nabla H(x)\|}+x\Big)\nonumber\\
& = dH(\partial_{y_{i}})+\frac{\partial\varphi}{\partial y_{i}}(y)dH\Big(\frac{\nabla H(x)}{\|\nabla H(x)\|}\Big)\label{eqn:phi1}\\
0 & = \frac{d^{2}}{d y_{i} d y_{j}}H\Big(y+\varphi(y)\frac{\nabla H(x)}{\|\nabla H(x)\|}+x\Big)\nonumber\\
& =  \frac{d}{d y_{j}}\Big(dH(\partial_{y_{i}})+\frac{\partial\varphi}{\partial y_{i}}(y)dH\Big(\frac{\nabla H(x)}{\|\nabla H(x)\|}\Big)\Big)\nonumber\\
& = D^{2}H (\partial_{y_{i}},\partial_{y_{j}})+\frac{\partial\varphi}{\partial y_{i}}(y)D^{2}H\Big(\partial_{y_{j}},\frac{\nabla H(x)}{\|\nabla H(x)\|}\Big)+ \frac{\partial\varphi}{\partial y_{j}}(y)D^{2}H\Big(\partial_{y_{i}},\frac{\nabla H(x)}{\|\nabla H(x)\|}\Big)\nonumber\\
& + \frac{\partial\varphi}{\partial y_{i}}(y)\frac{\partial\varphi}{\partial y_{j}}(y) D^{2}H\Big(\frac{\nabla H(x)}{\|\nabla H(x)\|},\frac{\nabla H(x)}{\|\nabla H(x)\|}\Big)+ \frac{\partial^{2}\varphi}{\partial y_{i} \partial y_{j}}(y)dH\Big(\frac{\nabla H(x)}{\|\nabla H(x)\|}\Big)\label{eqn:phi2}
\end{align}
To estimate $\frac{\partial\varphi}{\partial y_{i}}(y)$, we will first show that for all $z \in B(x,\frac{1}{4M}\inf_{\Sigma})$ the following relation holds
\begin{equation}
dH_{z}(\nabla H(x))>\tfrac{1}{2}\|\nabla H(x)\|\ \|\nabla H(z)\|.
\label{eqn:dHzdHx}
\end{equation}
If we calculate the first derivative with respect to $z$ of the difference of the two terms
on both sides of this inequality, we obtain
\begin{align*}
\nabla (dH_{z}(\nabla H(x))-\tfrac{1}{2}\|\nabla H(x)\|\ \|\nabla H(z)\|) & = D^{2}_{z}H\Big(\nabla H(x)-\tfrac{1}{2}\frac{\|\nabla H(x)\|}{\|\nabla H(z)\|}\nabla H(z)\Big),\\
\|\nabla (dH_{z}(\nabla H(x))-\tfrac{1}{2}\|\nabla H(x)\|\ \|\nabla H(z)\|)\| & \leq \tfrac{3}{2}M \|\nabla H(x)\|,
 \end{align*}
which allows us to estimate
$$
dH_{z}(\nabla H(x))-\tfrac{1}{2}\|\nabla H(x)\|\ \|\nabla H(z)\| \geq \tfrac{1}{2}\|\nabla H(x)\|^{2}- \tfrac{3}{2}M \|\nabla H(x)\| \|z-x\| .
$$
Therefore, (\ref{eqn:dHzdHx}) is satisfied whenever
$$
z \in  B(x, \tfrac{1}{3M}\inf_{\Sigma}\|\nabla H\|).
$$
Now combining (\ref{eqn:phi1}) and (\ref{eqn:dHzdHx}) we can conclude that for 
$$
y+\varphi(y)\frac{\nabla H(x)}{\|\nabla H(x)\|}+x \in B(x, \tfrac{1}{3 M}\inf_{\Sigma}\|\nabla H\|)
$$
one has the following estimate
$$
\Big|\frac{\partial\varphi}{\partial y_{i}}(y)\Big| \leq \frac{\|\nabla H(x)\|\ |dH(\partial_{y_{i}})|}{|dH(\nabla H(x))|} \leq 2 
$$
Using the above estimate along with (\ref{eqn:phi2}) and (\ref{eqn:dHzdHx}) we obtain
\begin{align*}
\Big|\frac{\partial^{2}\varphi}{\partial y_{i} \partial y_{j}}(y)\Big| & \leq M \Big|dH\Big(\frac{\nabla H(x)}{\|\nabla H(x)\|}\Big)\Big|^{-1}\Big(1+\Big|\frac{\partial\varphi}{\partial y_{i}}(y)\Big|+\Big|\frac{\partial\varphi}{\partial y_{j}}(y)\Big|+\Big|\frac{\partial\varphi}{\partial y_{i}}(y)\Big|\Big|\frac{\partial\varphi}{\partial y_{j} }(y)\Big|\Big)\\
& \leq 18M (\inf_{\Sigma}\|\nabla H\|)^{-1}.
\end{align*}
This concludes the proof that $\Sigma$ is an uniformly embedded manifold for
$$
\delta_{0}=\frac{1}{3M}\inf_{\Sigma}\|\nabla H\|,
$$
and therefore by Theorem 2.31 from \cite{Eldering2013} there exists a uniform tubular neighborhood for $\Sigma$.
\end{proof}

%% file: appendix.tex
\section{Calculations on the maximum principle}
\label{app:f}
\begin{lem}
Consider $F:\mathbb{R}^{2n}\to\mathbb{R}$ to be the radial function  
$$
F(x):= \tfrac{1}{4}\|x\|^{2},$$
and $H$ a Hamiltonian satisfying property (\ref{item:H2}). Define $f_{1}, f_{2}$ as in (\ref{eqn:f1}) and (\ref{eqn:f2}). Then there exist constants $\alpha_{1},\alpha_{2},\alpha_{3},\alpha_{4}>0$, %which can be expressed in terms of constants from (\ref{eqn:HessH}), (\ref{eqn:nablaH}) and (\ref{eqn:H0}), 
such that the following holds for all $x\in\mathbb{R}^{2n}$
\begin{eqnarray}
| f_{1}(x)| \leq \alpha_{1}(\|x\|+1)^{2}, &\| \nabla f_{1}(x)\| \leq \alpha_{3}(\|x\|+1), \label{eqn:f1calc}\\
| f_{2}(x)| \leq \alpha_{2}(\|x\|+1)^{2}, & \| \nabla f_{2}(x)\| \leq \alpha_{4}(\|x\|+1).\label{eqn:f2calc}
\end{eqnarray}
\label{lem:f1}
\end{lem}
\begin{proof}
Recall the definition of $f_{1}$ and $f_{2}$ from (\ref{eqn:f1}) and (\ref{eqn:f2})
\begin{align*}
f_{1}(x) & :=d_{x}(dF(X^{H}))(X^{H}) -\|\nabla H(x)\|^{2} -\|\nabla (d^{\mathbb{C}}_{x}F(X^{H}))\|^{2}-\|\nabla(dF(X^{H}))(x)\|^{2},\\
f_{2}(x) & := d_{x}^{\mathbb{C}}F(X^{H}).
\end{align*}
Let us first investigate $f_{2}$ as the simpler of the two functions. We will bound $|f_{2}(x)|$ and $\|\nabla f_{2}(x)\|$ using the inequalities (\ref{eqn:HessH}) and (\ref{eqn:nablaH}) induced by the assumption (\ref{item:H2}) on $H$:
\begin{align}
f_{2}(x)=d^{\mathbb{C}}F(X^{H}) & = \omega_{0}(J X^{H},X^{F}) = - \langle \nabla H, \nabla F \rangle,\nonumber \\
|f_{2}(x)|=|d^{\mathbb{C}}F_{x}(X^{H})| & \leq \|\nabla H\|\ \|\nabla F\| \leq \tfrac{1}{2}\|x\|(M\|x\|+h_{1}),\\
d(d^{\mathbb{C}}F(X^{H}))(\xi) & = -(\langle \Hess_{x} H (\xi), \nabla F \rangle + \langle \nabla H, \Hess_{x} F (\xi)\rangle)\nonumber\\
& = -\langle \Hess_{x} F ( \nabla H) + \Hess_{x} H( \nabla F), \xi \rangle,\nonumber\\
\|\nabla f_{2}(x)\|=\|\nabla (d^{\mathbb{C}}F(X^{H}))\| & \leq \|\Hess_{x} F\|\ \|\nabla H\| + \|\Hess_{x} H\|\ \|\nabla F\|\nonumber \\
& \leq M\|x\|+\tfrac{1}{2}h_{1}.
\end{align}
From the above inequalities, we get (\ref{eqn:f2calc}) with 
$$
 \alpha_{2}:= \tfrac{1}{2}\max\{M, \tfrac{h_{1}^{2}}{4M}\} \qquad \textrm{and}\qquad \alpha_{4}:=\max\{M,\tfrac{1}{2}h_{1}\}.
 $$
 
We will now proceed with showing bounds on $|f_{1}(x)|$ and $\|\nabla f_{1}(x)\|$ by investigating and bounding each of the expressions separately:
\begin{align*}
dF(X^{H}) & = \omega_{0}(X^{H},X^{F}) = \langle J \nabla H, \nabla F \rangle, \\
d(dF(X^{H}))(\xi) & = \langle J \Hess_{x} H (\xi), \nabla F \rangle + \langle J \nabla H, \Hess_{x} F (\xi)\rangle\\
& = \langle \Hess_{x} F (J \nabla H) - \Hess_{x} H(J \nabla F), \xi \rangle, \\
\|\nabla (dF(X^{H}))\| & \leq \|\Hess_{x} F\|\ \|\nabla H\| + \|\Hess_{x} H\|\ \|\nabla F\| \\
& \leq \tfrac{1}{2} (M\|x\|+h_{1}) + \tfrac{1}{2} M \|x\| = M\|x\|+\tfrac{1}{2}h_{1},\\
d(\|\nabla (dF(X^{H}))\|^{2})(\xi) & = 2 \langle \nabla (dF(X^{H})), \Hess_{x}F(J \Hess_{x}H(\xi))\rangle\\
& - 2 \langle \nabla (dF(X^{H})), \Hess_{x}H(J\Hess_{x}F(\xi)) + D^{3}H(J \nabla F,\xi)\rangle,\\
\|\nabla(\|\nabla (dF(X^{H}))\|^{2})\| & \leq 2 \|\nabla (dF(X^{H}))\|(2\|\Hess_{x}F\|\ \|\Hess_{x}H\|+\|D^{3}H\|\ \|\nabla F\|)\\
& \leq 2 (M\|x\|+\tfrac{1}{2}h_{1})(M+\tfrac{1}{2}\|D^{3}H\|\ \|x\|)\\
& \leq 2(M\|x\|+\tfrac{1}{2}h_{1})(M+\tfrac{1}{2}L)\\
d(\|\nabla (d^{\mathbb{C}}F(X^{H}))\|^{2})(\xi) & = - 2 \langle \nabla (d^{\mathbb{C}}F(X^{H})),
\Hess_{x} F ( \Hess_{x} H (\xi))\\
& - 2 \langle \nabla (d^{\mathbb{C}}F(X^{H})), \Hess_{x} H( \Hess_{x} F(\xi))+D^{3}H(\nabla F, \xi)\rangle,\\
\|\nabla(\|\nabla (d^{\mathbb{C}} F(X^{H}))\|^{2})\| & \leq 2 \|\nabla (d^{\mathbb{C}}F(X^{H}))\|(2\|\Hess_{x}F\|\ \|\Hess_{x}H\|+\|D^{3}H\|\ \|\nabla F\|)\\
& \leq 2(M\|x\|+\tfrac{1}{2}h_{1})(M+\tfrac{1}{2}L),\\
d(dF(X^{H}))(X^{H}) & = \langle \Hess_{x} F (J \nabla H) - \Hess_{x} H(J \nabla F), X^{H} \rangle,\\
|d(dF(X^{H}))(X^{H})| & \leq (M\|x\|+\tfrac{1}{2}h_{1})(M\|x\|+h_{1}),\\
d(d(dF(X^{H}))(X^{H}))(\xi) & = \langle X^{H},\Hess_{x}F(J \Hess_{x}H(\xi))-\Hess_{x}H(J\Hess_{x}F(\xi))\rangle\\
& + \langle J \Hess_{x}H(\xi), \Hess_{x} F (J \nabla H) - \Hess_{x} H(J \nabla F)\rangle\\
& -\langle X^{H}, D^{3}H(J \nabla F,\xi)\rangle,
\end{align*}
\begin{align*}
\|\nabla (d(dF(X^{H}))(X^{H}))\| & \leq \|\nabla H\|(2\|\Hess_{x}F\|\ \|\Hess_{x}H\|+\|D^{3}H\|\ \|\nabla F\|)\\
& + \|\Hess_{x}H\|(\|\Hess_{x} F\|\ \|\nabla H\| + \|\Hess_{x} H\|\ \|\nabla F\|)\\
& \leq (M\|x\|+h_{1})(M+\tfrac{1}{2}L) + M( M\|x\|+\tfrac{1}{2}h_{1})
\end{align*} 
Now taking into account all the above calculations we obtain explicit bounds on $|f_{1}(x)|$ and $\|\nabla f_{1}(x)\|$ in terms of $\|x\|$:
\begin{align*}
|f_{1}(x)| & \leq |d(dF(X^{H}))(X^{H})| +\|\nabla H\|^{2} +\|\nabla (d^{\mathbb{C}}F(X^{H}))\|^{2} +\|\nabla(dF(X^{H}))\|^{2}\\
& \leq (M\|x\|+\tfrac{1}{2}h_{1})(M\|x\|+h_{1})+(M\|x\|+h_{1})^{2}+2(M\|x\|+\tfrac{1}{2}h_{1})^{2}\\
& \leq 4 M^{2}\|x\|^{2}+\tfrac{11}{2}M h_{1} \|x\| +2h_{1}^{2}\\
& \leq (2 M \|x\|+\tfrac{2}{3}h_{1})^{2},\\
\|\nabla f_{1}(x)\| & \leq \|\nabla (d(dF(X^{H}))(X^{H}))\|+\|\nabla (\|\nabla H\|^{2})\|+ \|\nabla(\|\nabla (d^{\mathbb{C}} F(X^{H}))\|^{2})\|\\
& +\|\nabla(\|\nabla (dF(X^{H}))\|^{2})\| \\
& \leq (M\|x\|+h_{1})(M+\tfrac{1}{2}L) + M( M\|x\|+\tfrac{1}{2}h_{1}) +2 M (M\|x\|+h_{1})\\
& +4(M\|x\|+\tfrac{1}{2}h_{1})(M+\tfrac{1}{2}L)\\
& \leq M(8M+\tfrac{5}{2}L)\|x\|+\tfrac{1}{2}h_{1}(7M+3L). 
\end{align*}
This way we obtain (\ref{eqn:f1calc}) with 
$$
 \alpha_{1}:= \max\{2M, \tfrac{2}{3}h_{1}\}^{2} \quad \textrm{and}\quad \alpha_{3}:=\max\{M(8M+\tfrac{5}{2}L),\tfrac{1}{2}h_{1}(7M+3L)\}.
 $$
\end{proof}

%% file: Acknow.tex
%\nospace
\section*{Acknowledgment}
We would like to thank Rob van der Vorst for suggesting to us the idea of trying to extend the definition of Rabinowitz Floer Homology to non-compact hypersurfaces, and in particular to level sets of Hamiltonian functions satisfying condition (\ref{item:H1}). Moreover, we would like to thank Alberto Abbondandolo for sharing with us his ideas about how to bound the Floer trajectories near the non-compact hypersurface using the Morse-Bott property. We would also like to thank Oliver Fabert, Thomas Rot, Kai Cieliebak and Urs Frauenfelder for fruitful discussions on this subject.

This work has been supported by the Vrije Competitie grant 613.001.111 \textit{Periodic motions on non-compact energy surfaces} of NWO. The second author also would like to thank Kai Zehmisch for his kind hospitality at the Westf\"{a}lische Wilhelms-Universit\"{a}t M\"{u}nster and for the partial financial support from the grant SFB/TRR 191 \textit{Symplectic Structures in Geometry, Algebra and Dynamics} of the DFG.